\documentclass{article}

\usepackage{amsfonts,amssymb, latexsym, amsmath, amsthm,mathrsfs, verbatim,geometry}

\usepackage{slashed}
\usepackage{amsfonts}
\usepackage{url,color}
\usepackage{enumerate}
\usepackage{esint}
\usepackage{pifont}
\usepackage{upgreek}
\usepackage{times}
\usepackage{calligra}
\usepackage{graphicx}
\usepackage{caption}
\usepackage{amscd}
\usepackage{amsthm}
\usepackage{amsmath,bm}
\usepackage{amssymb}
\usepackage{comment}
\usepackage{tikz}
\usepackage[utf8]{inputenc}
\usepackage{gensymb}
\usepackage{pifont}
\usepackage{autobreak}
\usepackage{color}
\usepackage{verbatim}
\usepackage{systeme}
\usepackage[font=scriptsize]{caption}
\usepackage[ngerman, english]{babel}
\usepackage[title,toc,page]{appendix}
\usepackage{subfig}
\usepackage{tikz}
\usetikzlibrary{calc}
\usetikzlibrary{calc}
\usetikzlibrary{arrows.meta}
\usetikzlibrary{hobby}
\usepackage{tkz-euclide}
\usepackage{bbm}

\newcommand{\norm}[1]{\left\lVert#1\right\rVert}

\newcommand{\R}{\mathbb{R}}
\newcommand{\ga}{\gamma}
\newcommand{\la}{\lambda}

\newcommand{\al}{\alpha}
\newcommand{\pa}{\partial}

\newcommand\eps{\varepsilon}

\def\be{\begin{equation}}
\def\ee{\end{equation}}
\def\beqa{\begin{equation}\begin{aligned}}
\def\eeqa{\end{aligned}\end{equation}}

\def\espz{Z}
\def\espw{W}
\def\mz{M_z}
\def\mw{M_w}
\def\lz{\mathcal L_z}
\def\lw{\mathcal L_w}
\def\nz{\mathcal N_z}
\def\nw{\mathcal N_w}
\def\nmz{\mathcal T_z}
\def\nmw{\mathcal T_w}
\def\epone{\mathtt{a_1}}
\def\eptwo{\mathtt{a_2}}
\def\epzero{\mathtt{a_0}}

\def\ssmu{\mathfrak S}
\def\ndmu{{\bf{n}_{\mu}}}

\numberwithin{equation}{section}

\usepackage{listings}

\allowdisplaybreaks

\usepackage{hyperref}
\hypersetup{
    colorlinks,
    citecolor=red,
    filecolor=green,
    linkcolor=blue,
    urlcolor=black
}
\definecolor{backcolour}{rgb}{0.95,0.95,0.92}
\lstdefinestyle{mystyle}{
    backgroundcolor=\color{backcolour},   
}

\usepackage{tabularx}
\usepackage{multirow}
\usepackage{blindtext} 
\usepackage{charter} 
\usepackage{microtype} 
\usepackage[english, ngerman]{babel} 
\usepackage{amsthm, amsmath, amssymb} 
\usepackage{float} 
\usepackage{graphicx, multicol} 
\usepackage{xcolor} 
\usepackage{marvosym, wasysym} 
\usepackage{rotating} 
\usepackage{censor} 
\usepackage{pseudocode} 
\usepackage{booktabs} 
\usepackage{csquotes} 
\newtheorem{theorem}{Theorem}[section]

\newtheorem{lemma}[theorem]{Lemma}
\newtheorem{remark}[theorem]{Remark}

\newtheorem{proposition}[theorem]{Proposition}

\newtheorem*{claim*}{Claim}
\newtheorem*{notation*}{Notation}

\title{Gradient blowup of smooth vacuum solutions to 1D compressible Euler equations}

\author{Juhi Jang\thanks{Department of Mathematics, University of Southern California, Los Angeles, CA 90089, USA, and Korea Institute for Advanced Study, Seoul, Korea.  Email: juhijang@usc.edu.}, \ Jiaqi Liu\thanks{Department of Mathematics, University of Southern California, Los Angeles, CA 90089, USA, Email: jiaqil@usc.edu}, \ and \ Nader Masmoudi\thanks{ NYUAD Research Institute, New York University Abu Dhabi, Abu Dhabi, UAE,
and Courant Institute of Mathematical Sciences, New York University, 251 Mercer St, New York,
NY 10012, USA, Email: nm30@nyu.edu}}

\date{}
\begin{document}

\maketitle

\abstract{We consider the isentropic compressible Euler equations in the half-line which govern the motion of gaseous fluids in contact with stationary vacuum boundary. 
We construct a large class of solutions that are initially smooth and square-integrable, and which, in finite time, transition to $C^{1-\mu}$ regularity for $\mu \in [1/2,1)$ near the boundary, leading to the gradient blowup at the boundary.  
It is based on stability analysis of self-similar waiting time solutions \cite{JLN2025} recently constructed by the authors.
}

\section{Introduction}

We consider one-dimensional polytropic gases governed by compressible Euler equations:
\be\label{Euler}
\begin{split}
\rho_t+ (\rho u)_x &=0, \\
(\rho u)_t+ (\rho u^2 +p)_x &= 0. 
\end{split}
\ee
Here $\rho:(t_0,T)\times \Omega \to \mathbb R_{\ge0}$, $u:(t_0,T)\times \Omega \to \mathbb R$, $p:(t_0,T)\times \Omega \to \mathbb R_{\ge0}$ represent the density, velocity, and pressure of the gas, and $\Omega$ is the spatial domain given by non-vacuum region. The pressure obeys the $\gamma$-law
\be\label{pressure}
p=\frac{\ga-1}{\ga} \rho^\gamma
\ee
where $\ga\in(1,3)$. 
The speed of the sound $c= \sqrt{\frac{dp}{d\rho}} $ is given by 
\[
c= \sqrt{(\gamma-1)}\rho^{\frac{\gamma-1}{2}}
\]
and in non-vacuum region, the Euler system \eqref{Euler} can be equivalently written as 
\be\label{Eulerh}
\begin{split}
u_t + u u_x + \frac{2}{\ga-1}cc_x&= 0,\\
c_t+uc_x+\frac{\ga-1}{2}cu_x &=0.
\end{split}
\ee
In the case of $\gamma=2$, the system \eqref{Euler} for $\rho$ and $u$ is known as the shallow water equations where $\rho$ represents the height of water level and $u$ the velocity. 

In this article, we are interested in classical 
solutions to \eqref{Euler} when they are continuously in contact with vacuum region. Specifically, we consider the Euler system \eqref{Euler} in the  spatial domain 
$$\Omega = \{ 0 < x <\infty \}$$ 
representing the support of the fluid. The left boundary $x=0$ is the first point of contact to vacuum state satisfying the following non-moving vacuum boundary conditions 
\be\label{BC1}
c (t,x) >0 \ \text{ for } \ 0<x<\infty, \quad c (t, 0)=0, \quad u (t,0)=0 . 
\ee
We consider initial data 
\be\label{IC}
c(t_0,x)= c_0(x), \ \  u(t_0,x)=u_0(x) \text{ for } x>0
\ee 
where 
\[
c_0(x)>0 \ \text{ for }x >0, \ c_0(0) =0, \ \text{ and } u_0(0)=0. 
\]

Well-posedness theory for strong solutions to \eqref{Eulerh}-\eqref{IC} depends on the behavior of  the sound speed near vacuum. Among others, if $ c_0(x) \sim x $ near $x=0$, Liu and Yang \cite{LY} showed the problem \eqref{Eulerh}-\eqref{IC} is  well-posed in Sobolev spaces at least for a short time. On the other hand, if $c_0(x) \sim x^\frac12$ near $x=0$, $(c_0^2)_x\sim 1$ near $x=0$. As it can be seen from the first equation of  \eqref{Eulerh}, it may trigger a non-trivial acceleration near the boundary; hence even if $u_0(0)=0$, the boundary may  move for $t>t_0$. In fact, well-posedness problem in such a case is best understood in the vacuum free boundary framework, where the domain is allowed to move: $\Omega =\{ b(t)<x<\infty\}$ and the left boundary is determined through the kinematic boundary condition $\dot b (t)= u(t, b(t))$. The initial condition $(c_0^2)_x\sim 1$ near $x=0$ is then propagated along the vacuum boundary such that 
\be\label{PVB}
0<\frac{\partial c^2}{ \partial x}(t,b(t)) <\infty . 
\ee
 This condition \eqref{PVB} is called  the physical vacuum boundary condition and it  arises  in various contexts ranging from gas dynamics to astrophysics when studying the dynamics of isolated compactly supported bodies \cite{CoSh2012, Jang2014, JM, LXZ, Sideris17}. There has been a lot of progress over the past decade on the local and global behavior of solutions for the vacuum free boundary problems for compressible fluids featuring \eqref{PVB}, for instance see  \cite{CoSh2012, GHJ2021a, HaJa2018-1, HaJa2016-2, HaJa2017,  IfTa2020, JaMa2009, JaMa2015}. 

In \cite{JM},  the first and third authors studied well and ill-posedness questions of various vacuum states satisfying $ c_0(x) \sim x^\alpha $ for a range of $\alpha$. 
In particular, it was shown that when $\alpha=1$ 
the boundary behavior $ c \sim x $ may not persist for long in general, and conjectured that after some time the boundary should start moving and attain the physical vacuum boundary \eqref{PVB} (so-called waiting time solutions), whereas for other ranges of $\alpha$ an instantaneous change of behavior into the physical vacuum might occur; see also \cite{LY1}. 
We note that waiting time solutions or instantaneous change of the boundary behavior are also of great interest in the context of the shallow water equation and the vacuum dam-break problem \cite{Camassa20,GB}. Since the Euler equation from gas dynamics covers the shallow water equation ($\gamma=2$), rigorous progress on the Euler system will have direct applications there. We also mention that the waiting
time behavior or the instantaneous change of behavior  of certain vacuum states have been thoroughly studied for nonlinear diffusion equations such
as the porous medium equation, for instance see \cite{CF1979,Knerr1977}.

The first existence of
waiting time solutions to the Euler system \eqref{Eulerh} undergoing the transition from  $\alpha=1$ to $\alpha=1/2$ with the physical vacuum boundary was recently given  by the authors \cite{JLN2025} inspired by recent progresses on rigorous construction of self-similar solutions in fluid and gas dynamics \cite{GHJ21,GHJ23,GHJS22,JLS23,JLS24,Jenssen18,Jenssen23,Merle22a}. The construction is based on the self-similar reduction thanks to the scaling symmetry of the Euler system and analysis of the self-similar ODE system. Moreover, an interesting recent  work  \cite{Jenssen} of Jenssen shows that some self-similar solutions to  \eqref{Euler} exhibit an instantaneous change of the vacuum boundary, namely the vacuum boundary moves immediately even if $u_0(0)=0$, while some other self-similar solutions are confined by the stationary interface $x=0$. Of course,  self-similar waiting time solutions themselves arise from carefully prepared initial data and they are unbounded in the far field. It is therefore natural and important to ask whether nonself-similar smooth and bounded initial data can also exhibit similar singular phenomena such as gradient blowup or a transient phase in which a static boundary evolves into a physical vacuum moving boundary, and also how generic they are.

The goal of this article is to construct a large class of smooth, square integrable, nonself-similar solutions to the Euler equation \eqref{Eulerh} with the stationary vacuum boundary \eqref{BC1}, which exhibit self-similar blowup \cite{JLN2025} near the blowup point. 
We are interested in vacuum solutions initially more regular than $C^2$. 
It is convenient to 
introduce 
\be\label{beta}
 \mu \in [\frac12,1) 
\ \text{ and } \  \beta= \frac{1}{1-\mu} \ge 2,  
\ee 
where $\mu$ is related to the scaling symmetry of the Euler system, and $\beta$ 
characterizes the regularity of our solutions. 
 We now state the informal statement of our main result. 

\begin{theorem}\label{thm 1.1}
    Let $\ga\in(1,3)$ and $\mu\in[\frac12,1)$.  Let $t_0<0$ with $|t_0|$ sufficiently small. There exists a finite-codimensional set of $C^{\lfloor\beta \rfloor, \beta- \lfloor\beta \rfloor }$ initial data at $t=t_0$ such that they launch strong  solutions to the vacuum boundary Euler system \eqref{Eulerh}-\eqref{IC} on $[t_0,0)$ transitioning to be $C^{1-\mu}$-H\"{o}lder continuous near $x=0$ at $t=0$. 
    The vacuum boundary remains at $x=0$ for $t\in [t_0,0)$ and the first gradient blow-up occurs at $(t,x)=(0,0)$.  The velocity and the speed of the sound are bounded in $L^2$ for each   $t\in [t_0,0)$  and the density remains positive $\rho>0$ for all $(t,x)\in[t_0,0)\times(0,\infty)$.   
\end{theorem}

A precise formulation of Theorem \ref{thm 1.1} including the specification
of function spaces and finite co-dimensional sets is given in Theorem \ref{thm 4.1}. The following remark is on the initial data and the behavior of solutions near the first blowup point.

\begin{remark} Initial data \eqref{IC} at $t_0<0$ 
leading to gradient blowups in Theorem \ref{thm 1.1}  
have the following asymptotics near the boundary $x=0$
    \be
    \begin{split}\label{initial expansion}
         u_0(x) &=  \frac{2}{(\ga+1)t_0} x + \sum_{i=2}^{\lfloor\beta \rfloor} a_i x^i + a_0 x^{\beta} + o(x^{\beta}),\\
        c_0(x)&=  -\frac{\ga-1}{(\ga+1)t_0} x + \sum_{i=2}^{\lfloor\beta \rfloor} b_i x^i -\frac{\ga-1}{2} a_0 x^{\beta} + o(x^{\beta}),
    \end{split}
    \ee
    where $a_0>0$ and $b_i:= b_i( a_2,\dots, a_i)$. Moreover, for given $0<x\ll 1$, as $t\to 0^-$ solutions constructed in Theorem \ref{thm 1.1} satisfy 
   \beqa\label{1.7}
        u(t,x) = -\tilde a x^{1-\mu}+o(|t|),\  c(t,x) = \tilde b x^{1-\mu}+o(|t|)
   \eeqa
with $ \tilde a, \tilde b>0$.
\end{remark}

\begin{notation*} 
Throughout the article, we use $\lfloor\beta \rfloor$  and  $\lceil\beta \rceil$ to denote the floor function  and the ceiling function respectively:  $\lfloor\beta \rfloor = \max\{ n \in \mathbb Z: n \le \beta \}$ and   $\lceil\beta \rceil = \min\{ n \in \mathbb Z: n \ge \beta \}$. Note that $\lceil\beta \rceil = \lfloor\beta \rfloor +1  $ if $\beta \notin \mathbb Z$ and  $\lceil\beta \rceil = \lfloor\beta \rfloor  $ if $\beta \in  \mathbb Z$. 
\end{notation*} 

To the best of our knowledge,  Theorem \ref{thm 1.1} provides  the first rigorous result featuring the descriptive gradient blowup of smooth, square integrable vacuum solutions to the compressible Euler equation. 
Loss of regularity from \eqref{initial expansion} to \eqref{1.7} is driven by the self-similar blowup of the Burgers' equation. In fact, self-similar waiting time solutions constructed in \cite{JLN2025}  are described by 
the self-similar Burgers' equation in the first Riemann variable $z$ for $t<0$, while the second Riemann variable $w$ vanishes (cf. \eqref{RI}); see Section \ref{sec:SSS} for details. Theorem \ref{thm 1.1} can be viewed as a stability result of self-similar solutions constructed in \cite{JLN2025} up to their first blowup time.     

In the study of self-similar shocks for the Burgers' equation, $C^\infty$ self-similar solutions are obtained for $\beta = 2k+1$, $k\in \mathbb N$. Moreover, it is well-known that self-similar solutions with $\beta = 3$ are the most stable one and they exhibit gradient blowup with $C^\frac13$ cusp singularity \cite{CGN18, EF00}.  Burgers self-similar shock solutions have been successfully used as base profiles or guiding principle to study the cusp-type singularity formation of other models such as  compressible Euler equations \cite{BuckmasterIyer,Buckmaster22,NealShkollerVicol25}, the Burgers' equation through transverse viscosity \cite{CGN18} or fractional KdV and fractional Burgers' equation  \cite{OhPasqualotto24}. We also mention that in \cite{NRSV24},  
stability of self-similar Burgers' equation with $\beta\in(1,\infty)$ were studied, 
and in the range of $\beta\in (1,2]$ the profiles  have been used to demonstrate stable $C^{\frac{1}{\beta}}$ cusp singularities for the 1D compressible Euler equations.

In our problem, we also use Burgers self-similar solutions as base profiles and study the perturbed dynamics. The key difference arises from the presence of the stationary vacuum boundary condition \eqref{BC1}. Solutions are confined in the region $\{x>0\}$ and we do not insist on smoothness across the boundary. In particular, our Burgers self-similar index satisfies $\beta\in \mathbb R_{\ge 2}$
and the corresponding initial data are smoother than $C^2$ near the boundary. As a result, Theorem \ref{thm 1.1} allows a broad class of initial data to evolve into H\"{o}lder continuous solutions $C^{1-\mu}$ for $\mu \in [1/2,1)$, of which first derivatives blow up at $(t,x)=(0,0)$. Moreover, our result deciphers the precise finite co-dimensional set of initial data that lead asymptotically to the self-similar gradient blowup.

In  Theorem \ref{thm 1.1}, the first blowup time  has been fixed as $t=0$ for simplicity and a systematic study of solutions for all $\mu \in [1/2,1)$. Other blowup times are possible thanks to the time translation symmetry of the Euler system. In particular, by modulating the linear coefficients \eqref{initial expansion}, one can see that initial slopes $u_x(t_0,0)$ and $c_x(t_0,0)$ determine the final blowup time, analogously to the Burgers equation; see Appendix \ref{sec: TT} for discussion on time translation symmetry and modulation. Among all self-similar Burgers profiles indexed by $\mu \in [1/2,1)$, the profile associated with $\mu=1/2$ can be viewed as the most stable one in the sense that the three conditions imposed on the initial data \eqref{initial expansion} for $\mu=1/2$, which are $a_1=b_1=0$ and $b_2=b_2(a_2)$,  are all related to the symmetries of the  system; see Remark \ref{remark 2.10}. 

We discuss the methodology of Theorem \ref{thm 1.1} along with some technical points.  The proof is based on 
stability analysis of self-similar solutions in line with recent developments on self-similar solutions and stability \cite{Buckmaster22,CGN18,CRS19,DM19, GHJS25, LWZ25, Merle97, Merle22b,OhPasqualotto24}. In Section \ref{section 2}, we introduce self-similar reformulation for the system of Riemann invariants $(\hat z, \hat w)$ in which self-similar solutions constructed in \cite{JLN2025} are realized as steady states $(\bar z, 0)$. To obtain square integrable solutions, we study perturbations around localized self-similar solutions $\mathring z = \chi \bar z $ using the time dependent cutoff function transported by the transport velocity for $\hat z - \bar z$. Two transport velocities for the perturbed system are both outgoing in the self-similar variables (cf. \eqref{PDE of tilde z w}): the damping effect is amplified upon taking the derivatives and high frequency part is stabilized. Low frequency part in general faces a norm growth just like for self-similar shocks for the Burgers’ equation, and such a growth can be controlled through modulations. 
Motivated by this observation, we decompose the perturbations into a low frequency component serving as an approximate profile for our solutions near the singular point, and a remainder containing the high frequency contribution. We remark that such a decomposition is a standard tool in the analysis of blowup problems; see, for example, \cite{Buckmaster22,CGN18,CRS19,Merle97,OhPasqualotto24}.

More specifically, we seek our solutions $(\hat z, \hat w)$ of  the following form 
\[
\binom{\hat z}{\hat w} = \binom{\mathring z}{0} +  \sum_{i=2}^{\ndmu-1} \varphi  \binom{z_i}{w_i}\frac{y^i}{i ! } + \binom{Z}{W}
\]
where $\varphi$ is a compactly supported smooth cutoff function, the second finite sum  represents modulations, where $(z_i,w_i)$ are $i$-th order Taylor coefficients of solutions at the origin, and the last term is a remainder vanishing up to $(\ndmu-1)$-th order. Here $\ndmu$ is the regularity index  \eqref{def ndmu} intimately tied to linear stability. 
In Section \ref{section 3}, the modulation ODEs for $(z_i, w_i)$ are derived by imposing vanishing conditions at the origin of  high order equations. In particular,  Lemma \ref{lemma 3.7} identifies the initial data set for $(z_i, w_i)$ for $2\le i\le \lceil\beta\rceil$ leading to the time decay,  
namely a stable manifold of the system of modulation ODEs. For the remainder $(Z,W)$, we adopt a robust approach based on energy estimates with an appropriately chosen singular weight and high order energy. Specifically, we analyze the system for $(\frac{Z}{y^\ndmu},\frac{W}{y^\ndmu})$ as well as $(\pa_y^\ndmu Z, \pa_y^\ndmu W)$, of which linear part becomes strictly dissipative with our choice of $\ndmu$. Linear stability analysis is carried out in Section \ref{section 4}. We remark that 
singular weights via suitable local vanishing conditions at the origin on perturbations of the interior profile have been successfully used  in recent years for other blowup problems \cite{chenHou22, hou20242, chenHouhuang21, LWZ25}. 
For nonlinear analysis in Section \ref{section 5}, we employ the bootstrap argument by regulating different decay or growth conditions for various norms including $\dot H^1$ and $L^2$ norms. The bootstrap assumptions are improved upon establishing several energy estimates with the aid of Hardy and interpolation inequalities. 

We note that non-integer $\beta$'s bring additional challenges as the profiles $\mathring z $ contain $y^\beta$ term (cf. \eqref{z y<<1}) and thus $\pa_y^j \mathring z$ becomes unbounded if $j \ge \lceil \beta\rceil$. The profiles enter the system for $(Z,W)$ as coefficients in the leading order part as well as as interaction terms with the modulation part. Too many derivatives may create non-integrable terms, and hence function spaces for $(Z,W)$ need to be chosen carefully. In fact, when  $\beta$ is a half-integer, the regularity requirement  to ensure strict coercivity of the linear part  induces non-integrable
singular behavior of interaction terms between low modulation variables $(z_2,w_2)$ and the profile itself. In this case, we impose the additional constraint $w_2=0$ as in \eqref{initialziwi}.

An interesting question is whether H\"{o}lder continuous solutions obtained in Theorem \ref{thm 1.1} could be developed for $t>0$ as solutions to the Euler equation. In \cite{JLN2025}, it was shown that within the class of self-similar solutions, such H\"{o}lder continuous solutions continue as solutions to the free boundary Euler satisfying the physical vacuum condition \eqref{PVB} for $t>0$ with two characteristic curves emerged as the vacuum boundary and a weak discontinuity. We leave this continuation problem of nonself-similar solutions obtained in Theorem \ref{thm 1.1} as a future study.

\section{Self-similar reformulation}\label{section 2}

The system \eqref{Eulerh} admits a two-parameter family of invariant scalings: the scaling transformation
\be\label{scaling}
u (t,x) \to \nu^{\delta -1 } u ( \frac{t}{\nu},\frac{x}{\nu^\delta }), \quad c (t,x) \to  \nu^{\delta -1 } c ( \frac{t}{\nu},\frac{x}{\nu^\delta })
\ee
for $\nu, \, \delta >0$ leaves the system invariant. This scaling symmetry is closely tied to the existence of self-similar solutions. In \cite{JLN2025}, the authors constructed a continuum two-parameter family of self-similar solutions to the Euler equation \eqref{Eulerh} with the vacuum boundary \eqref{BC1}, whose gradients blow up at the vacuum boundary in a finite time. We give more details on such self-similar solutions in Section \ref{sec:SSS}.

We will work with Riemann variables. Riemann invariants are 
defined by 
\be\label{RI}
	z = u -\frac{2}{\ga-1}c, \ \  w = u +\frac{2}{\ga-1}c 
\ee
and they satisfy
\be\label{EE z w}
	\begin{aligned}
		\partial_t z+\la_1\partial_x z &= 0,\\
		\partial_t w +\la_2 \partial_x w &=0, 
 	\end{aligned}
\ee
where two characteristic speeds $\la_1$, $\la_2$ are 
\beqa
	\la_1 = u - c=\frac{\ga+1}{4} z + \frac{3-\ga}{4} w, \\
	 \la_2 = u +c= \frac{3-\ga}{4} z + \frac{\ga+1}{4} w.
\eeqa

Inspired by the scaling symmetry \eqref{scaling}, we define 
the self similar variables $(\tau, y)$ by 
\be\label{tau y}
	\frac{d\tau }{dt} = \frac{\delta }{ (- t)} 
    \ \text{ and }\ y = \frac{x}{(-t)^\delta} . 
\ee
Here for simplicity we set the first blowup time to be $t=0$ . Due to the time translation symmetry, other blowup times $T$ can be obtained by considering $T-t$ instead of $-t$. And we have  $e^{-\mu\tau}= \mu (-t) $ where $$\mu= \frac{1}{\delta}.$$  
Let  
\beqa\label{selfsimilar ansatz}
	z(t,x) = \delta (-t)^{\delta -1} \hat z(\tau,y)= (\frac{1}{\mu})^\delta e^{(\mu-1)\tau}\hat z(\tau,y),\\
	w(t,x) = \delta (-t)^{\delta -1} \hat w(\tau,y)= (\frac{1}{\mu})^\delta e^{(\mu-1)\tau}\hat w(\tau,y).
\eeqa
Plugging \eqref{selfsimilar ansatz} into \eqref{EE z w}, renormalized unkowns  $(\hat z, \hat w)$  satisfy 
\be\label{ss EE zw}
	\begin{cases}
	\partial_\tau \hat z + (y+\hat \la_1 ) \partial_y \hat z + (\mu-1)\hat z= 0,\\
	\partial_\tau \hat w + (y+\hat \la_2 ) \partial_y \hat w + (\mu-1)\hat w = 0, 
	\end{cases}
\ee
where 
\be\label{def la1 la2}
\hat \la_1 = \frac{\ga+1}{4} \hat z + \frac{3-\ga}{4} \hat w, \quad 
	\hat \la_2 =  \frac{3-\ga}{4} \hat z + \frac{\ga+1}{4} \hat w. 
\ee

\subsection{Steady state with respect to $\tau$}\label{sec:SSS}

Adopting the result in \cite{JLN2025} into the Riemann invariants $(\hat z, \hat w)$, we have a simple wave steady state $(\bar z(y), 0)$ to \eqref{ss EE zw} satisfying 
\be \label{2.8}
	\Big(y+ \frac{\ga+1}{4}\bar z\Big) \partial_y \bar z + (\mu-1)\bar z= 0. 
\ee
Letting
\be\label{z = -2y U}
	\bar z = -2y \bar U(y), 
\ee
 $\bar U$ solves  the ODE 
\be\label{ode B}
	y \frac{d\bar U}{dy} = -\frac{[(\ga+1)\bar U-2\mu]\bar U}{(\ga+1) \bar U-2}.
\ee

\smallskip

We record the properties of $\bar U$ \cite{JLN2025}. 

\begin{lemma}
\label{lemma 3.4}  Let $\ga\in(1,3)$ and $\mu\in(0,1)$ be given. The solutions of \eqref{ode B} with $\bar U(0)= \frac{2\mu}{1+\gamma}$ satisfy the following 
\be\label{B formula}
y= K \frac{ ( \frac{2\mu}{\gamma+1} -\bar U)^{\frac{1}{\mu}-1}}{ \bar U^\frac{1}{\mu}}  \ \text{ for }  \ 0< y <\infty,  
\ee
and 
\be\label{B bc}
   \bar U' \leq 0, \ \bar U(0)= \frac{2\mu}{\gamma+1} , \  \ \lim_{y\to\infty} \bar U (y)  =0, \text{ and }  0< \bar U <  \frac{2\mu}{\gamma+1},
\ee
where $K>0$ is a free parameter. 
Moreover,  $\bar U$ enjoys the following asymptotic behavior: as $y\to 0^+$ 
\be\label{Asy at C}
\bar U (y) =  \bar U(0) - \frac{ \bar U(0)^\frac{1}{1-\mu}}{K^{\frac{\mu}{1-\mu}}} y^{\frac{\mu}{1-\mu}} + O (y^{\frac{2\mu}{1-\mu}} ),  \ \  0<y \ll 1
\ee
and as $y\to \infty$, 
\be\label{Asy at A}
\bar U(y) =  K^\mu  \bar U(0)^{1-\mu}\frac{1}{y^\mu} + o(\frac{1}{y^\mu}), \ \ y \gg  1  .
\ee
\end{lemma}
\begin{remark}\label{remark 2.3}
The corresponding Riemann invariant obtained from $\bar z(y)$, namely $\bar z(t,x) = \delta (-t)^{\delta-1} \bar z (y)  $ (cf. \eqref{selfsimilar ansatz})  
 at given initial time $t_0<0$ 
 near the vacuum boundary   
takes the from 
$$ \bar z(t_0,x) =  \tfrac{4}{(\gamma+1)t_0} x + c_1 x^{\beta}+ o(x^{\beta}), \quad 0<x\ll 1 $$ 
where $c_1>0$ is an arbitrary constant; 
as $t\to 0^-$ and $x>0$, it takes the form $$ \bar z(t,x) = - c_2 x^{1-\mu} + o(|t|), \quad x>0, \ |t|\ll 1 $$
 for some $c_2>0$. 
 \end{remark}

We note that self-similar solutions $\bar z$ are realized as a family of two parameter solutions $\bar z^{\mu, K}$, and they are closely related to self-similar shocks for Burgers equation \cite{CGN18} and \cite{EF00}. The key difference arises from the presence of the boundary at $y=0$. Solutions are confined to $\{y>0\}$; we do not insist on smoothness of solutions across $y=0$ but the regularity up to $y=0$ from the non-vacuum region. In particular, we are interested in smooth data near the vacuum boundary for the Euler system, smoother than $C^2$ at the level of Riemann variables. In view of Remark \ref{remark 2.3}, we will consider $\mu \in [\frac12, 1)$ from now.

Denote the set of $\mu$ whose corresponding $\beta$ \eqref{beta} is an integer by 
\be\label{def ssmu}
	\ssmu: = \{\mu\in[\frac12,1) \mid \beta-[\beta ]=0\}.
\ee

We collect some useful properties of $\bar z$.  
 \begin{proposition}\label{prop bar z}
Let $\gamma \in (1,3)$ and $\mu \in \bigl[\tfrac12,1\bigr)$.  
The steady state $\bar z$ satisfies the following properties:
\begin{enumerate}
    \item[$(1)$] The first derivative is given by
    \[
        \partial_y \bar z = \frac{-4(\mu-1)\bar U}{(\gamma+1)\bar U-2}.
    \]
    Moreover, $\partial_y \bar z$ is monotone increasing from $\partial_y \bar z(0)=-\frac{4\mu}{\ga+1}$ to $\lim_{y\to \infty} \pa_y \bar z(y)=0$. In particular, $ -\frac{4\mu}{\ga+1} \le \partial_y \bar z <0$
    
    \item[$(2)$] For any $y \in [0,\infty)$, one has
    \[
        1 + \frac{\gamma+1}{4}\,\partial_y \bar z \in [\,1-\mu,\,1\,).
    \]
    
    \item[$(3)$] As $0<y\ll 1$, the function $\bar z$ admits the expansion
    \begin{equation}\label{z y<<1}
        \bar z(y) = -\frac{4\mu}{\gamma+1} y \;+\; c_1 y^{\beta} 
        \;+\; O\!\left(y^{2\beta -1 }\right),
    \end{equation}
    where $c_1=c_1 (\ga,\mu,K)>0$.  
    Furthermore, there exists a constant $A>0$ such that, for $0<y\ll 1$,
    \begin{align}
        |\partial_y \bar z(y)| &\leq A, \notag \\ 
        |\partial_y^2 \bar z(y)| &\leq A\, y^{\beta-2}, \notag \\
        |\partial_y^i \bar z(y)| &\leq 
        \begin{cases}
            A\, y^{j(\beta-1)+1-i}, 
            & \text{if } \mu\in \ssmu,\; j\in \mathbb N_{\geq 1},\; (j-1)(\beta-1))+1 < i \leq j(\beta-1)+1, \\[6pt]
            A\, y^{\beta-i}, 
            & \text{if } \mu \notin \ssmu,
        \end{cases}\label{payi z y<<1}
    \end{align}
    for all $i \in \mathbb N_{\geq 3}$.
    
    \item[$(4)$] As $y \to \infty$, one has the asymptotic behavior
    \begin{equation}\label{z y>>1}
        \bar z(y) = -c_2 y^{1-\mu} + o\!\left(y^{1-\mu}\right),
    \end{equation}
    where $c_2:=c_2(\ga,\mu,K)>0$. Moreover, there exists a constant $A>0$ such that, for $y \gg 1$,
    \begin{equation}\label{payi z y>>1}
        |\partial_y^i \bar z(y)| \leq A\, y^{1-\mu-i}, 
    \end{equation}
    for all $i \in \mathbb N_{\geq 1}$.
\end{enumerate}
\end{proposition}

\begin{proof}
For $(1)$, since $\bar z = -2y\bar U$, we compute
\[
    \partial_y \bar z = -2\Bigl(\bar U + y\,\partial_y \bar U \Bigr)
    = -2\Biggl(\bar U - \frac{[(\gamma+1)\bar U - 2\mu]\bar U}{(\gamma+1)\bar U-2}\Biggr)
    = \frac{-4(\mu-1)\bar U}{(\gamma+1)\bar U - 2}.
\]
Since the function $g(r) = \frac{-4(\mu-1)r}{(\gamma+1)r - 2}$ is strictly decreasing on the interval $r\in[0, \frac{2\mu}{\ga+1}]$ and $\bar U(y)$ is strictly decreasing from $\frac{2\mu}{\ga+1}$ to $0$ as $y$ increases from $0$ to $\infty$, $ \partial_y \bar z$ is increasing from $\partial_y \bar z(0)=-\frac{4\mu}{\ga+1}$ to $\lim_{y\to \infty} \pa_y \bar z(y)=0$.

For $(2)$, one finds
\[
    1+\frac{\gamma+1}{4}\,\partial_y \bar z
    = \frac{(2-\mu)\bar U - \tfrac{2}{\gamma+1}}{\bar U - \tfrac{2}{\gamma+1}}.
\]
The function $f(r)=\frac{(2-\mu)r-\tfrac{2}{\gamma+1}}{r-\tfrac{2}{\gamma+1}}$ is decreasing for $r\in\bigl[0,\tfrac{2}{\gamma+1}\mu\bigr]$. Since $\bar U$ is decreasing in $y$, it follows that
\[
    1-\mu = \frac{(2-\mu)\bar U(0) - \tfrac{2}{\gamma+1}}{\bar U(0)-\tfrac{2}{\gamma+1}}
    < 1+\frac{\gamma+1}{4}\,\partial_y \bar z
    \leq \lim_{y\to\infty} \frac{(2-\mu)\bar U(y) - \tfrac{2}{\gamma+1}}{\bar U(y)-\tfrac{2}{\gamma+1}} = 1,
\]
where we used \eqref{Asy at C} in the last equality.

For $(3)$, the expansion \eqref{z y<<1} follows directly from \eqref{Asy at C}. Since $\beta \geq 2$, the bounds on $\partial_y \bar z$ and $\partial_y^2 \bar z$ are immediate. For higher-order derivatives, note that $\bar z \in C^{\infty}$ if $\mu \in \ssmu$, and $\bar z \in C^{[\beta],\beta-[\beta]}$ if $\mu \notin \ssmu$. This explains the dichotomy in \eqref{payi z y<<1}. A direct computation yields the stated bounds.

For $(4)$, the asymptotics \eqref{z y>>1} follow from \eqref{Asy at A}, and the derivative estimates \eqref{payi z y>>1} are then straightforward.
\end{proof}

\subsection{Profiles}

Our goal is to construct classical solutions around the steady state $(\bar z, 0)$, which exhibit the loss of smoothness resulting in the same gradient blowup behavior as the self-similar solutions near the vacuum boundary $x=0$ as described in Remark \ref{remark 2.3}, through stability analysis. Due to the growth of $\bar z (y)$ (cf. Proposition \ref{prop bar z} (4)), the corresponding $\bar z (t,x)$ is unbounded in $x$. Moreover, the change of the behavior of solutions near the boundary is a local intrinsic phenomenon for the Euler flows irregardless of the behavior in the far-field region. To this end, we will localize the self-similar solutions. To motivate a localization, 
we observe that the perturbation $\hat z -\bar z$ satisfies 
\[
\partial_\tau \left( \hat z -\bar z\right) + (y+\frac{\ga+1}{4}\bar z ) \partial_y ( \hat z -\bar z) +  \text{ lower order or  nonlinear terms } =0
\]
with the transport velocity $y+\frac{\ga+1}{4}\bar z $. We will then construct the $\tau$-dependent cutoff function $\chi (\tau, y)$ whose initial cutoff is being transported by the same velocity $y+\frac{\ga+1}{4}\bar z$.

To be specific, consider a  nonnegative and smooth function $\chi_0:\R_{\geq 0}\to\R_{\geq 0}$ satisfying 
\be\label{def chi}
	\chi_0(y)=\begin{cases}
		 1,\ y\in[0, y_0],\\
		 0,\ y\in(2y_0,\infty),
	\end{cases} \ \text{ and }\ \ 0\geq \chi_0'(y) \geq -\frac{2}{y_0} 
\ee
for some $y_0>1$. 
We define the cut-off function $\chi$ by the solution of  
\be\label{pde of chi}
	\pa_\tau \chi+(y+\frac{\ga+1}{4}\bar z )\pa_y \chi =0,\qquad \chi(\tau_0,y) = \chi_0(y), 
\ee
Here $y_0$ and $\tau_0$ are sufficiently large constants to be determined later. The existence of $\chi$ solving \eqref{pde of chi} follows from the method of characteristics.   

We define the 
profile by 
\be\label{def mathring z}
	\mathring z(\tau,y) = \chi(\tau ,y) \bar z(y),  
\ee
which satisfies 
\be\label{pde of mathring z}
	\pa_\tau \mathring z+(y+\frac{\ga+1}{4}\mathring z )\pa_y \mathring z+(\mu-1)\mathring z = -\frac{\ga+1}{4} \bar z (1-\chi) \pa_y \mathring z.
\ee	

We next present the properties of $\chi$ and $\mathring z$ 
by beginning with the flow map induced by the transport velocity $y+\frac{\ga+1}{4}\bar z $: 
\begin{proposition}\label{prop Y(tau,p)}
	Let $\ga\in(1,3)$ and $\mu\in[\frac12,1)$. The solution to the ODE 
	\be\label{Lagrangian flow map}
		\begin{cases}
			\pa_\tau Y(\tau,p) = Y(\tau,p)+\frac{\ga+1}{4}\bar z(Y(\tau,p)),\\
			Y(\tau_0,p) = p\  \text{ for } p\geq 0,
		\end{cases}
	\ee
	is nondecreasing in $\tau$ and satisfies
	\be\label{range of flow map}
		Y(\tau,p) \in [pe^{(1-\mu)(\tau-\tau_0)},pe^{(\tau-\tau_0)} ].
	\ee
\end{proposition}
\begin{proof}
By using \eqref{z = -2y U}, we have   
	\be\label{eq Y(tau,p)}
		Y(\tau,p) 
        = p e^{\int_{\tau_0}^\tau 1-\frac{\ga+1}{2} \bar U(Y(\tau',p)) d\tau'}.
	\ee
   From \eqref{B bc}, we have $1-\frac{\ga+1}{2} \bar U(Y(\tau',p)) \in[1-\mu,1]$. It then follows that $Y(\tau,p)$ is nondecreasing in $\tau$
    and satisfies \eqref{range of flow map}. 
\end{proof}

\begin{lemma}\label{le: supp chi}
Let $\chi$ be a smooth solution to 
\eqref{pde of chi} with \eqref{def chi}. 
Then, $\chi$ has the following properties
\begin{enumerate}
\item[$(1)$]  $\text{supp}(\chi)\subset [0, 2y_0e^{\tau-\tau_0}) $,
\item[$(2)$] $\text{supp}(1-\chi)\subset [y_0e^{(1-\mu)(\tau-\tau_0)},\infty) $,
\item[$(3)$] $\chi(\tau,y)\in[0,1]$, $\pa_y \chi(\tau,y)\in[-\frac{2}{y_0},0]$ for $\tau>\tau_0$ and $y\geq0$,
\item[$(4)$] $|\pa_y^i \chi(\tau,y) |\leq A y^{-i}$ for $y\in[y_0,\infty)$ with $A=A(i)$.
\end{enumerate}
\end{lemma}
\begin{proof}
	(1) and (2) follow directly from Proposition \ref{prop Y(tau,p)}. 
    For (3), 
    we first use \eqref{pde of chi} and \eqref{Lagrangian flow map} to derive
	\[
		\chi(\tau,Y(\tau,p)) = \chi(\tau_0,Y(\tau_0,p)) \in[0,1].
	\]
	Hence, the first claim in (3) follows. 
	 We then apply $\pa_y$ to \eqref{pde of chi} to obtain
	\begin{align*}
		\pa_\tau \pa_y\chi+(y+\frac{\ga+1}{4}\bar z )\pa_y \pa_y\chi = -(1+\frac{\ga+1}{4}\pa_y \bar z) \pa_y\chi.
	\end{align*}
	By using the flow map $Y(\tau,p)$ defined in \eqref{Lagrangian flow map}, we have 
	\begin{align*}
		\pa_\tau \pa_y\chi(\tau,Y(\tau,p))= -(1+\frac{\ga+1}{4}\pa_y \bar z) \pa_y\chi(\tau,Y(\tau,p)).
	\end{align*}
	We then integrate the above ODE by using the integral factor 
	\be \label{def IF}
	f(\tau) = \int (1+\frac{\ga+1}{4}\pa_y \bar z (Y(\tau',p) )) d\tau' 
	\ee 
	to obtain
	\be\label{pa_y chi}
		\pa_y\chi(\tau,Y(\tau,p)) = \pa_y\chi(\tau_0,Y(\tau_0,p)) e^{-\int_{\tau_0}^\tau (1+\frac{\ga+1}{4}\pa_y \bar z(Y(\tau',p)))\ d\tau'}.
	\ee
    By using Proposition \ref{prop bar z} (2) and \eqref{range of flow map}, it then follows that
	\begin{align*}
		|\pa_y\chi(\tau,Y(\tau,p))| \leq |\pa_y\chi(\tau_0,Y(\tau_0,p))| e^{-(1-\mu)(\tau-\tau_0)} \leq |\pa_y\chi(\tau_0,Y(\tau_0,p))|.
	\end{align*}
	Together with \eqref{def chi}, it leads to the second claim in (3).  
    
    For (4), first note that since $\chi_0$ has a compact support, there exists a constant $A:=A(i)$ such that
	\be\label{ineq pa_yi chi}
		\mid\pa_y^i \chi_0(y)\mid \leq A y^{-i} \text{ for any } y \ge y_0. 
	\ee
	We then apply $\pa_y^i $ to  \eqref{pde of chi} and obtain
	\begin{align*}
		\pa_\tau \pa_y^i\chi+(y+\frac{\ga+1}{4}\bar z )\pa_y \pa_y^{i}\chi = -i(1+\frac{\ga+1}{4}\pa_y \bar z(y)) \pa_y^i\chi - \sum_{j=1}^{i-1} c_{ij} \pa_y^{j+1}\bar z (y) \pa_y^{i-j}\chi
	\end{align*}
	where the last summation appears only when $i\geq 2$ and $c_{ij}$ are constants. Along $Y(\tau,p)$, we have   
\begin{align*}
		\pa_\tau \pa_y^i\chi(\tau,Y(\tau,p)) = -i(1+\frac{\ga+1}{4}\pa_y \bar z(Y(\tau,p))) \pa_y^i\chi (\tau,Y(\tau,p))- \sum_{j=1}^{i-1} c_{ij} \pa_y^{j+1}\bar z (Y(\tau,p)) \pa_y^{i-j}\chi(\tau,Y(\tau,p)).
	\end{align*}
	Then, by using the integral factor $if(\tau)$ where $f(\tau)$ is defined in \eqref{def IF}, we obtain
	\be\label{eq payi chi}
    \begin{split}
		\pa_y^i\chi(\tau,Y(\tau,p)) &= \pa_y^i\chi(\tau_0,Y(\tau_0,p)) e^{i(f(\tau_0)-f(\tau))}\\
        &- \sum_{j=1}^{i-1}c_{ij}\int_{\tau_0}^\tau e^{i(f(\tau')-f(\tau))}\pa_y^{j+1}\bar z (Y(\tau',p)) \pa_y^{i-j}\chi(\tau',Y(\tau',p)) d\tau'. 
	\end{split}
    \ee
We first observe that if $p\le y_0$, $\chi_0=1$. Thus $\chi (\tau, Y(\tau,p)) =1$ for each $\tau\ge \tau_0$ and $p\le y_0$, and hence for  $i\ge 1$, $\pa_y^i \chi(\tau,Y(\tau,p))=0$ for $\tau\ge \tau_0$ and $p<  y_0$. Therefore, to prove the claimed bound in (4), it is enough to consider $p \ge y_0$. 
	We will prove it by induction on $i$.  
    When $i=1$, by using \eqref{z = -2y U}, \eqref{B bc},  \eqref{eq Y(tau,p)} and \eqref{ineq pa_yi chi}, we have
	\begin{align*}
		|\pa_y \chi(\tau,Y(\tau,p))| &=| \pa_y\chi(\tau_0,Y(\tau_0,p)) |e^{-\int_{\tau_0}^\tau 1+\frac{\ga+1}{4}\pa_y \bar z(Y(\tau,p))\ d\tau}\\
		&\leq A p^{-1} e^{-\int_{\tau_0}^\tau 1-\frac{\ga+1}{2} \bar U(Y(\tau,p)) d\tau} e^{\int_{\tau_0}^\tau\frac{\ga+1}{2} Y(\tau,p) \pa_y \bar U(Y(\tau,p))d\tau }\leq A (Y(\tau,p))^{-1}.
	\end{align*}
	We now assume that for any $i< m$, $|\pa_y^i \chi(\tau,y)|\leq A y^{-i}$. Then from \eqref{eq payi chi}, we obtain 
	 \begin{align*}
	 	|\pa_y^m \chi(\tau,Y(\tau,p))| &\leq A e^{m(f(\tau_0)-f(\tau))} p^{-m} +A  \int_{\tau_0}^\tau e^{m( f(\tau')-f(\tau))} \big(Y(\tau',p)\big)^{-\mu-m} d\tau'.
	 \end{align*}
     where we have also used Proposition \ref{prop bar z} (4) and \eqref{ineq pa_yi chi}. 
	 By using \eqref{z = -2y U}, \eqref{eq Y(tau,p)} and \eqref{def IF}, we may rewrite the right-hand side as  
	 \begin{align*}
	 	 &e^{m(f(\tau_0)-f(\tau))} p^{-m} +\int_{\tau_0}^\tau e^{m( f(\tau')-f(\tau))} \big(Y(\tau',p)\big)^{-\mu-m} d\tau'\\
	 &=  \big(Y(\tau,p)\big)^{-m}\Bigg( e^{m\int_{\tau_0}^{\tau} \frac{\ga+1}{2} Y(\tau',p) \pa_y \bar U(Y(\tau',p)) d\tau'} + \int_{\tau_0}^\tau e^{m\int_{\tau'}^{\tau} \frac{\ga+1}{2} Y(\tau'',p) \pa_y \bar U(Y(\tau'',p)) d\tau''} \big(Y(\tau',p)\big)^{-\mu} d\tau' \Bigg).
	 \end{align*}
    The desired bound follows once the boundedness of the terms in the bracket is shown. 
 The first term in the bracket is bounded for any $\tau\geq \tau_0$, since $Y(\tau',p) \pa_y \bar U(Y(\tau',p))\leq0$ holds for any $\tau'\geq \tau_0$. Concerning the second term, by \eqref{range of flow map}, we first observe that $\big(Y(\tau',p)\big)^{-\mu}\leq p^{-\mu} e^{-\mu(1-\mu) (\tau'-\tau_0)} \le y_0^{-\mu} e^{-\mu(1-\mu) (\tau'-\tau_0)}  $. 
	 We then further bound
	 \[
	 	 \int_{\tau_0}^\tau e^{m\int_{\tau'}^{\tau} \frac{\ga+1}{2} Y(\tau'',p) \pa_y \bar U(Y(\tau'',p)) d\tau''} \big(Y(\tau',p)\big)^{-\mu} d\tau' \leq y_0^{-\mu} \int_{\tau_0}^\tau e^{-\mu(1-\mu) (\tau'-\tau_0)} d\tau' \le (\mu(1-\mu))^{-1}
	 \]
where we have used $y_0>1$ in the last inequality. 
     This completes the proof of (4). 
\end{proof}

We also record 
the following properties of $\mathring z$: 
\begin{proposition}\label{prop: mathring z}
Let $\ga\in(1,3)$ and $\mu\in[\frac12,1)$. The profile $\mathring z$ exhibits different asymptotic behavior in the regions $y\in[0,1]$ and $y\in[1,\infty)$. More precisely, there exists a positive constant $A:=A(\ga, \mu)$ such that
    	\beqa |\pa_y^i\mathring z|&=|\pa_y^i \bar z | \qquad \  \text{ when }\quad y\in[0,1],\\
    |\pa_y^i \mathring z| &\leq A y^{1-\mu-i} \ \  \text{ when }\quad y\in[1,\infty).\label{profile be}
    \eeqa
\end{proposition}
\begin{proof}
	It is a direct consequence of Proposition \ref{prop bar z}, Lemma \ref{le: supp chi}, and the definition \eqref{def mathring z}.
\end{proof}

\subsection{Main result} 

We introduce the perturbations $(\tilde z, \tilde w)$ around the profile $(\mathring z, 0)$: 
\be
	\begin{aligned}
		\hat z(\tau,y) = \mathring z(\tau, y) + \tilde z(\tau, y) \ \ \text{and} \ \  \hat w(\tau,y) =\tilde w(\tau, y).
	\end{aligned}
\ee
From \eqref{ss EE zw} and \eqref{pde of mathring z}, the equations of the perturbation $(\tilde z, \tilde w)$ can be written in the following form 
\be\label{PDE of tilde z w}
	\begin{aligned}
		\partial_\tau \tilde z + \Big(y+\frac{\ga+1}{4}\mathring z \Big) \partial_y \tilde z + \Big(\frac{\ga+1}{4}\partial_y \mathring z +\mu-1\Big)\tilde z +\frac{3-\ga}{4}\partial_y \mathring z \ \tilde w +(\frac{\ga+1}{4}\tilde z+\frac{3-\ga}{4}\tilde w)\partial_y \tilde z=& \\ \qquad \frac{\ga+1}{4} \bar z (1-\chi) \pa_y \mathring z ,&\\
		\partial_\tau \tilde w + \Big(y+\frac{3-\ga}{4}\mathring z \Big) \partial_y \tilde w + (\mu-1)\tilde w +(\frac{3-\ga}{4}\tilde z+\frac{\ga+1}{4}\tilde w)\partial_y \tilde w= 0.&
	\end{aligned}
\ee
Two transport velocities for $\tilde z$ and $\tilde w$ are both outgoing, and they will produce stabilizing effect (damping) for high order derivatives.  
The exact number of derivatives required to gain positive damping depends on $\mu$, more precisely $\beta$, and such a damping effect is crucial for nonlinear analysis. 

We now define the regularity index $\ndmu$ by 
\be\label{def ndmu}
        \ndmu = \lfloor  \beta + \frac12  \rfloor+1  = \lfloor \frac{3-\mu}{2(1-\mu)} \rfloor +1 .
\ee
We seek the perturbed solutions $(\tilde z, \tilde w)$ in the form of  
\be\label{sol decompose}
	\tilde z(\tau,y) = \sum_{i=2}^{\ndmu-1} \varphi(y) \frac{z_i(\tau)}{i!}y^i + Z(\tau,y) \ \ \text{ and } \ \   \tilde w(\tau,y) = \sum_{i=2}^{\ndmu-1} \varphi(y) \frac{w_i(\tau)}{i!}y^i + W(\tau,y), 
\ee
where  
$  \varphi(y) $ is a smooth cutoff function satisfying  
\be
   \varphi(y) = \begin{cases}
        1,\ y\in[0,\frac12],\\
        0,\ y\geq 1,
    \end{cases}\ \text{ and }\ \ 0\geq \varphi'(y) \geq -4.
\ee
The first summations in \eqref{sol decompose} are localized polynomials in $y$ whose coefficients are determined by the modulation ODEs, while remainders $(Z,W)$ vanish at $y=0$ up to $(\ndmu-1)$-order and they will be analyzed by using energy estimates and bootstrap argument.

To state the main result, we denote the set of $\mu$ whose corresponding $\beta$ \eqref{beta} is a half-integer by 
\be
\mathcal H = \big\{\mu\in[\frac12,1): \beta+\frac12 \in \mathbb N\big\}, 
\ee 
and 
a $2(\ndmu-1)$-dimensional open ball around $\vec v_0$ with radius $r$ by 
\be\label{set B}
    B(\vec v_0,r):=\left\{\vec v\in\mathbb R^{2\ndmu-2}: \ \left|\vec v-\vec v_0 \right| < r  \right\}.
\ee

Our main theorem is the following: 
\begin{theorem}\label{thm 4.1}
     Let $\ga\in(1,3)$ and $\mu\in[\frac12,1)$ be given. 
     Then there exist sufficiently large $\tau_0:=\tau_0(\ga,\mu)$, 
     sufficiently small $\sigma_0:=\sigma(\tau_0)$,
     and $\epzero:=\epzero(\ga,\mu)>0$   such that the following
     holds. There exist initial data $(\tilde z(\tau_0,y), \tilde w(\tau_0,y))$ in the form of \eqref{sol decompose}
     such that 
     \[\pa_y^i \espz(\tau_0,0)=\pa_y^i\espw(\tau_0,y)=0,\quad \text{ for } \quad i\in\{0,1,\dots,\ndmu-1\},\]
     and
     \be\label{initialziwi}
     (z_2, \dots, z_{\ndmu-1}; w_2,\dots, w_{\ndmu-1}) \in 
     \begin{cases}
     B(0,e^{-\epzero \tau_0})\cap I^{\lfloor \beta \rfloor}, & \text{ if } \mu \in [\frac12,1)\setminus \mathcal H \\
     B(0,e^{-\epzero \tau_0})\cap I^{\lfloor \beta \rfloor}\cap \{w_2=0\},&
     \text{ if } \mu \in \mathcal H 
     \end{cases}
     \ee 
     where $I^{\lfloor \beta \rfloor}$ is defined in \eqref{boundary condition}, and $\norm{\espz(\tau_0,\cdot),\espw(\tau_0,\cdot)}_{H^{\ndmu}}< \sigma_0$,  leading to a well-posed $H^{\ndmu}$  
     solution $(\tilde z(\tau),\tilde w(\tau))$ on any finite time interval 
     to \eqref{PDE of tilde z w} 
     with the asymptotics:
      \beqa\label{decay fact}
        \norm{(\pa_y\tilde z-\pa_y\mathring z,\pa_y\tilde w)(\tau)}_{L^\infty_y} \to 0 \quad \text{ as } \tau\to \infty
     \eeqa
     where the modified self-similar profile $\mathring z$ as defined in \eqref{def mathring z} has cutoff scale $y_0=e^{\tau_0}$ at $\tau=\tau_0$. 
    Moreover, 
    for each $\mu\in[\frac12,1)\setminus \mathcal H$, these initial data form a co-dimensional $\lfloor \beta+1 \rfloor$ subset of $H^\ndmu$, while for $\mu\in \mathcal H$, with an extra constraint $w_2\equiv0$, these initial data form a co-dimensional $\lfloor \beta+2 \rfloor$ subset of $H^\ndmu$.   
\end{theorem}

We make several remarks below. 

\begin{remark} The choice of $\ndmu$ \eqref{def ndmu} in our theorem is related to the dissipative property of the linearized operator $\lz$ in \eqref{4.5} in $L^2_{(y^{-2\ndmu}dy)}$ and $\dot H^\ndmu$. The modulation component in the decomposition \eqref{sol decompose} is truncated at order $\ndmu-1$. We distinguish three cases:
\begin{itemize}
    \item[(i)] if $\mu\in\ssmu$, then $\ndmu-1= \beta$;
    \item[(ii)] if $\mu\notin\ssmu$ with $ \lfloor \beta+\frac12\rfloor=\lfloor \beta\rfloor $, then $\ndmu-1 = \lfloor \beta\rfloor$;
    \item[(iii)] if $\mu\notin\ssmu$ with $ \lfloor \beta+\frac12\rfloor = \lfloor \beta+1\rfloor $, then $\ndmu-1 =\lfloor \beta\rfloor+1$. Consequently, for such values of $\mu$, the modulation requires one additional degree of control compared to cases (i) and (ii). 
\end{itemize}
By Remark~\ref{remark 3.7}, for $\mu\notin\ssmu$, we have the control of modulation up to the order $\lceil \beta\rceil$, which is equal to $\lfloor \beta\rfloor+1$ for $\mu\notin \ssmu$. Hence, the maximal available order of modulation control coincides with the truncation order in case (iii). Hence, in all cases, the chosen truncation order $\ndmu-1$ is compatible with the available modulation control. 
\end{remark}
\begin{remark}
    For $\mu\in \mathcal H$, an additional constraint $w_2\equiv 0$ is imposed.  This constraint originates from the interaction between the self-similar profile $\mathring z$ and the modulation component of the perturbation in our functional space $H^\ndmu$. If $w_2\neq 0$, 
    the resulting interaction term fails to be integrable in a neighborhood of $y=0$. For further discussion, we refer to Appendix \ref{Appendix C}. 
    \end{remark}

\begin{remark}\label{remark 2.10}
    By Remark~\ref{remark 3.7}, we observe that for $\mu\in[\frac12,\frac{3}{5})$, there exist three unstable directions; see also Proposition~\ref{trivial mode}. These directions are associated with the constraints $z_1=w_1=0$ and $z_2=q_2(w_2)$. 
    In the special case $\mu=\frac12$, the condition $z_1=w_1=0$ corresponds to the symmetries generated by time translation and Galilean transformation, whereas the relation $z_2=q_2(w_2)$ is associated with the scaling invariance. In contrast, when $\mu\in(\frac12,\frac{3}{5})$, only the constraints $z_1=w_1=0$ remain related to time translation, while $z_2=q_2(w_2)$ does not relate to any underlying symmetry of the system in the $H^\ndmu$ topology. In this sense, the case $\mu=\frac12$ may be viewed as the most stable one. 
\end{remark}

\section{Modulation equations and ODE system at $y=0$}\label{section 3}

In this section, we will analyze the ODE's satisfied by the modulation parameters and the coefficients $\big(\pa_y^i\tilde z(\tau,0),\pa_y^i\tilde w(\tau,0)\big)$. We denote 
\be\label{notation ziwi}
	z_i(\tau) = \pa_y^i \tilde z (\tau,0) \ \text{ and } \ w_i(\tau) = \pa_y^i \tilde w (\tau,0).
\ee
Since we are looking for solutions confined by the stationary vacuum boundary at $x=0$ up to the first gradient blowup, we must have 
\be\label{zero BC}
    \hat z(\tau,0) = \hat w (\tau,0) = 0,\ \text{ for all }\tau>\tau_0, 
\ee 
and hence $z_0=w_0\equiv 0$. 
This condition naturally removes the space translation and the Galilean transformation. We further impose $z_1 = w_1 \equiv 0$, which removes the time translation and thus the first gradient blowup of our solutions occurs at $t=0$; see Appendix \ref{sec: TT} for further discussion.  We remark that for $\mu\notin\ssmu$, the condition ensures that the interaction between the resulting modulation and the profile $\mathring z$ does not create a non-integrable singularity near $y=0$ in the analysis of $(Z,W)$ in the $H^\ndmu$ topology.

In what follows, we derive the ODE system satisfied by $(z_i,w_i)$ for $i\ge 2$, analyze the system, identify the unstable and stable coefficients for each $\mu\in[\frac{1}{2},1)$, and subsequently establish their corresponding estimates.

To derive the equations, we apply $\pa_y^i$ to \eqref{PDE of tilde z w} and isolate the $(i+1)$-th order of derivatives and $i$-th order of derivatives of unknowns to obtain 
\begin{align}
    \pa_\tau \pa_y^i\tilde z &+ \Big(y+ \frac{\ga+1}{4} \mathring z+\frac{\ga+1}{4}\tilde z+ \frac{3-\ga}{4}\tilde w\Big) \pa_y^{i+1}\tilde z\notag \\
   & +\Big(\mu+i-1+\frac{\ga+1}{4}(i+1)(\pa_y\mathring z+ \pa_y\tilde z)+\frac{3-\ga}{4}i\pa_y\tilde w\Big)\pa_y^{i} \tilde z \label{3.9r} \\
   & +\frac{3-\ga}{4} (\pa_y \mathring z+\pa_y \tilde z) \pa_y^i \tilde w 
    + N_z^i
    =\frac{\ga+1}{4} \pa_y^i(\bar z (1-\chi)\pa_y\mathring z) , \notag \\
    \pa_\tau \pa_y^i \tilde w& + \Big(y+\frac{3-\ga}{4}(\mathring z+\tilde z)+\frac{\ga+1}{4}\tilde w \Big) \partial_y^{i+1} \tilde w+\frac{3-\ga}{4}  \pa_y \tilde w \pa_y^i \tilde z\notag \\
   & +\Big(\mu+i-1+\frac{3-\ga}{4}i(\pa_y\mathring z+\pa_y\tilde z)+\frac{\ga+1}{4}(i+1)\pa_y\tilde w  \Big) \partial_y^{i} \tilde w 
    +N_w^i =0, \notag 
\end{align}
where 
\beqa\label{NzNw}
    N_z^i&=\frac{\ga+1}{4}\sum_{j=1}^{i-2}c_{ij} \pa_y^{i-j}\tilde z \pa_y^{j+1}\tilde z+\frac{\ga+1}{4}\sum_{j=0}^{i-2}c_{ij} \pa_y^{i-j}\mathring z \pa_y^{j+1}\tilde z + \frac{3-\ga}{4}\sum_{j=1}^{i-2}c_{ij} \pa_y^{i-j}\tilde w \pa_y^{j+1}\tilde z\\
    &+\frac{\ga+1}{4}\sum_{j=0}^{i-1}c_{ij}\pa_y^{i+1-j}\mathring z \pa_y^{j}\tilde z+\frac{3-\ga}{4}\sum_{j=0}^{i-1}c_{ij}\pa_y^{i+1-j}\mathring z \pa_y^j \tilde w \\
    N_w^i&=\frac{3-\ga}{4}\sum_{j=0}^{i-2}c_{ij}\pa_y^{i-j}\mathring z \pa_y^{j+1} \tilde w + \frac{3-\ga}{4}\sum_{j=1}^{i-2}c_{ij} \pa_y^{i-j}\tilde z \pa_y^{j+1}\tilde w + \frac{\ga+1}{4}\sum_{j=1}^{i-2}c_{ij} \pa_y^{i-j}\tilde w \pa_y^{j+1}\tilde w
\eeqa
with $c_{ij}$ are the binomial constants.

By using $z_0=w_0=z_1=w_1=0$ in \eqref{3.9r}, we obtain the following equations for  $(z_i,w_i)$, $i\ge 2$
\beqa\label{ode of ziwi with z1=0}
    \pa_\tau z_i + (i-1-i\mu)z_i-\frac{(3-\ga)\mu}{\ga+1} w_i  
    +N_z^i|_{y=0}&=0,\\
    \pa_\tau w_i + \Big(\mu+i-1-\frac{(3-\ga)i\mu}{\ga+1} \Big) w_i+ N_w^i|_{y=0} &=0, 
\eeqa
where $N_z^i|_{y=0}$ and $N_w^i|_{y=0}$ consist of lower-order terms with respect to $(z_i,w_i)$.  
It is convenient to introduce the anti-damping and damping  coefficients of the system \eqref{ode of ziwi with z1=0} by 
 \beqa\label{g_i,k_i}
 k_i &= k_i(\mu) = i-1-i\mu, \\
    g_i &= g_i(\mu) =  \mu+i-1-\frac{(3-\ga)i\mu}{\ga+1}. 
    \eeqa

The main result of this section is the following: 
\begin{lemma}\label{lemma 3.7}
 Let $\ga\in(1,3)$ and $\mu\in[\frac12,1)$. For each $i\in\{2,\dots,\lceil \beta \rceil \}$, the solutions $(z_i,w_i)$ to \eqref{ode of ziwi with z1=0} have the following properties:  
 \begin{enumerate}
     \item When $\mu=\frac12$, $w_2(\tau) = e^{-g_2 (\tau-\tau_0)}w_2(\tau_0)$ for  $g_2>0$ and 
      one of the following holds for $z_2$:
     \begin{enumerate}
         \item if $z_2(\tau_0) = -\frac{3-\ga}{5\ga-3} w_2(\tau_0)$, $z_2$ decays exponentially to 0 such that 
         \be
            z_2(\tau) = e^{- g_2 (\tau-\tau_0)} z_2(\tau_0);
         \ee
         \item otherwise,
         \be
            z_2(\tau) \to z_2(\tau_0)+\frac{3-\ga}{5\ga-3} w_2(\tau_0) \neq 0 \text{ as }\tau\to \infty.
         \ee
     \end{enumerate}
         \item When $\mu=(\frac12,1)$. 
         \begin{enumerate}
             \item If $z_i(\tau_0) \neq q_i(w_2(\tau_0),\dots, w_i(\tau_0))$ for some $i\in \{2,\dots,\lfloor\beta\rfloor\}$ where polynomial $q_i$ is defined in \eqref{qm}, then
             \be\label{case 2.1}
               z_i(\tau) \to \infty \text{ as }\tau\to \infty.
               \ee
             \item If $z_i(\tau_0)  = q_i(w_2(\tau_0),\dots, w_i(\tau_0))$ for all $i\in \{2,\dots,\lfloor\beta\rfloor\}$, then for each $i\in\{2,\dots,\lfloor\beta\rfloor\}$
             \begin{align}
                |z_i(\tau)|, \ |w_i(\tau)| &\leq C e^{-g_i(\tau-\tau_0)}. \label{j<=i}
      \end{align}
     Moreover, for $\mu\notin \ssmu$ and $i =\lceil \beta \rceil$,  
      \begin{align}
                |w_{\lceil \beta \rceil}(\tau)| &\leq C e^{-g_{\lceil \beta \rceil}(\tau-\tau_0)},\label{wi+1}\\
                |z_{\lceil \beta \rceil}(\tau)| &\leq e^{-k_{\lceil \beta \rceil}(\tau-\tau_0)}\left |z_{\lceil \beta \rceil}(\tau_0)-q_{\lceil \beta \rceil} \right | + Ce^{-g_{\lceil \beta \rceil}(\tau-\tau_0)}.\label{zi+1}
                \end{align}
         \end{enumerate}
 \end{enumerate}
    \end{lemma}
    \begin{proof}
    The proof follows from analyzing the solutions of \eqref{ode of ziwi with z1=0}. 
    By introducing the integral factor \eqref{g_i,k_i},
    we first obtain  
  \beqa\label{eq 3.25}
        &w_i(\tau) =e^{-g_i(\tau-\tau_0)} w_i(\tau_0) -e^{-g_i\tau} \int_{\tau_0}^\tau e^{g_i\tau'} N_w^i|_{y=0}(\tau') \ d\tau',\\
         &z_i(\tau) = e^{-k_i(\tau-\tau_0)} z_i(\tau_0) -  e^{-k_i\tau} \int_{\tau_0}^\tau e^{k_i(\tau')}\Big( N_z^i|_{y=0}(\tau')-\frac{(3-\ga)\mu}{\ga+1} w_i(\tau') \Big) \ d\tau'\\
        &= e^{-k_i(\tau-\tau_0)} \Big(z_i(\tau_0) +\frac{3-\ga}{2(\ga-1)i+\ga+1} w_i(\tau_0) \Big) -\frac{3-\ga}{2(\ga-1)i+\ga+1} w_i(\tau_0)e^{-g_i(\tau-\tau_0)}\\
        & -e^{-k_i\tau} \int_{\tau_0}^\tau\Bigg( e^{k_i\tau'}\Big( N_z^i|_{y=0}(\tau')\Big) + \frac{(3-\ga)\mu}{\ga+1} e^{(k_i-g_i)\tau'}\int_{\tau_0}^{\tau'} e^{g_i\tau''} \Big(N_w^i|_{y=0}(\tau'')\Big) \ d\tau'' \Bigg)d \tau' .
  \eeqa
When $i=2$, then by using $z_0=w_0=z_1=w_1=0$, we have $N_z^i|_{y=0} =N_w^i|_{y=0}=0$,
  \beqa\label{w2z2 sol}
        w_2(\tau) &=  e^{-g_2(\tau-\tau_0)} w_2(\tau_0),\\
        z_2(\tau) &=e^{-k_2(\tau-\tau_0)} \Big(z_2(\tau_0) +\frac{3-\ga}{5\ga-3} w_2(\tau_0) \Big) -\frac{3-\ga}{5\ga-3} w_2(\tau_0)e^{-g_2(\tau-\tau_0)}.
    \eeqa

    \underline{Case 1:} It is a direct consequence of \eqref{w2z2 sol} and $k_2=0$ for $\mu=\frac12$.

    \underline{Case 2:} For more general cases, we first notice that 
    \beqa\label{gi>0}
        g_i \geq \mu+i-1-i\mu = (i-1)(1-\mu)>0
    \eeqa
    for any $\ga\in(1,3)$, $\mu\in[\frac12,1)$ and $i\geq 2$. And,
    \beqa
        k_i = i-1-i\mu = (1-\mu) (i -\beta)
    \eeqa
    where $\beta$ is defined in \eqref{beta}. Therefore, for $\mu\in(\frac12,1)$, we have 
    \[\beta> 2 \text{ and } k_2<0.\] 
    Then, from \eqref{w2z2 sol}, $z_2$ has an exponentially growing factor $e^{-k_2(\tau-\tau_0)}$ if $z_2(\tau_0)\neq q_2(w_2(\tau_0))$, which proves \eqref{case 2.1} and \eqref{j<=i} in the case of $i=2$. For larger $i$, we notice that $k_i<0$ for $i\in \{2, \dots,\lfloor \beta \rfloor\}$. There is always a growing factor $e^{-k_i(\tau-\tau_0)}$ if its coefficient is not zero. This motivates us to define $q_i$ and we  have the following claim:
    \begin{claim*}
        For each $\mu\in(\frac12,1)$ and $i\in \{2, \dots,\lceil \beta \rceil\}$, there exist polynomials
        \be
            q_j(\tau_0):=q_j(w_2(\tau_0),\dots, w_j(\tau_0))
        \ee
        for each $j\in\{2,\dots,i\}$ 
        such that 
        if the initial data of $z_j$ satisfy 
        \be\label{zj(tau0) constraint}
            z_j(\tau_0) = q_j(w_2(\tau_0),\dots, w_j(\tau_0))  \  \text{ for all }j\in\{2,\dots,i-1\}, 
        \ee
        then 
          \beqa\label{wjzj sol with condition}
               w_i(\tau) &= p_j(\tau_0) e^{-g_i(\tau-\tau_0)}
        +
        \sum_{\substack{\ell_m\geq 0,\\ \sum_{ m=2}^ { i-1} \ell_m (m-1) = i-1 }} d_{ \ell_2,\dots, \ell_{i-1}}(\tau_0)  
        e^{-\sum_{ m=2}^ { i-1}\ell_mg_m(\tau-\tau_0)}\\
       z_i(\tau)&=(z_i(\tau_0)-q_i(\tau_0))e^{-k_i(\tau-\tau_0)}+\tilde p_i(\tau_0)e^{-g_i(\tau-\tau_0)}\\&\quad + \sum_{\substack{\ell_m\geq 0,\\ \sum_{ m=2}^ { i-1} \ell_m (m-1) = i-1 }} \tilde d_{ \ell_2,\dots, \ell_{i-1}}(\tau_0)  
       e^{-\sum_{ m=2}^ { i-1}\ell_mg_m(\tau-\tau_0)}
    \eeqa
    where $p_j(\tau_0)$ and $\tilde p_j(\tau_0)$ are polynomials in $w_2(\tau_0),\dots, w_j(\tau_0)$, and $d_{ \ell_2,\dots, \ell_{j-1}}(\tau_0)$ and $\tilde d_{ \ell_2,\dots, \ell_{j-1}}(\tau_0)$ are polynomials in $w_2(\tau_0),\dots, w_{j-1}(\tau_0)$.
    \end{claim*} 
        \noindent\emph{Proof of Claim:}
            Fixed a $\mu\in(\frac12,1)$, then $\lceil\beta\rceil\geq 3$.  \eqref{w2z2 sol} proves the case of $i=2$ with 
            \be
                 q_2(w_2(\tau_0)) :=- \frac{3-\ga}{5\ga-3} w_2(\tau_0).
            \ee
        For $i\geq 3$, we prove the statement by induction. Assume that the claim holds for all $j\in\{2,\dots, i-1\}$. Our goal is then to demonstrate that the statement also holds for $j=i$. 
        
        Using $z_0=w_0=z_1=w_1=0$ and $\pa_y^j \bar z(0)=0$ for $j\in\{2,\dots, i-1\}$ from Proposition \ref{prop bar z} (3), the equations \eqref{NzNw} reduce to 
        \begin{align*}
             N_z^i|_{y=0}&=\frac{\ga+1}{4}\sum_{m=1}^{i-2}c_{im}  z_{i-m}  z_{m+1} + \frac{3-\ga}{4}\sum_{m=1}^{i-2}c_{im} w_{i-m}  z_{m+1},\\
    N_w^i|_{y=0}&= \frac{3-\ga}{4}\sum_{m=1}^{i-2}c_{im}  z_{i-m}  w_{m+1} + \frac{\ga+1}{4}\sum_{m=1}^{i-2}c_{im}  w_{i-m} w_{m+1}.
        \end{align*}
        We focus on the term $z_{i-m}z_{m+1}$ as 
        the remaining terms can be treated analogously.
        Without loss of generality, we may assume that $m+1\leq i-m$. 
        By applying the induction assumption 
        \eqref{wjzj sol with condition} to $i-m, m+1\in\{2,\dots, i-1\}$ under the constraint assumption \eqref{zj(tau0) constraint}, we compute that
        \begin{align*}
            z_{i-m}z_{m+1} = &\Big(\tilde p_{i-m}(\tau_0)e^{-g_{i-m}(\tau-\tau_0)} + \sum_{\substack{\ell_k\geq 0,\\ \sum_{ k=2}^ { i-m-1} \ell_k (k-1) = i-m-1 }} \tilde d_{ \ell_2,\dots, \ell_{i-m-1}}(\tau_0)  
       e^{-\sum_{ k=2}^ { i-m-1}\ell_k g_k(\tau-\tau_0)}\Big)\\
       &\times \Big(\tilde p_{m+1}(\tau_0)e^{-g_{m+1}(\tau-\tau_0)} + \sum_{\substack{\ell_k\geq 0,\\ \sum_{ k=2}^ { m} \ell_k (k-1) = m }} \tilde d_{ \ell_2,\dots, \ell_{m}}(\tau_0)  
       e^{-\sum_{ k=2}^ { m}\ell_k g_k(\tau-\tau_0)}\Big)\\
       =&\Big(  \sum_{\substack{\ell_k\geq 0,\\ \sum_{ k=2}^ { i-m} \ell_k (k-1) = i-m-1 }} \tilde d_{ \ell_2,\dots, \ell_{i-m}}(\tau_0)  
       e^{-\sum_{ k=2}^ { i-m}\ell_k g_k(\tau-\tau_0)}\Big)\\
       &\times \Big(\sum_{\substack{\ell_k’\geq 0,\\ \sum_{ k=2}^ { m+1} \ell_k’(k-1) = m }} \tilde d_{ \ell_2’,\dots, \ell_{m+1}’}(\tau_0)  
       e^{-\sum_{ k=2}^ { m+1}\ell_k’ g_k(\tau-\tau_0)}\Big)\\
       =&\sum_{\substack{\tilde \ell_k\geq 0,\\ \sum_{ k=2}^ { i-m} \tilde \ell_k (k-1) = i-1 }} \tilde d_{ \tilde \ell_2,\dots, \tilde\ell_{i-m}}(\tau_0)  e^{-\sum_{ k=2}^ { i-m}\tilde  \ell_k g_k(\tau-\tau_0)}.
       \end{align*}
        
        After the summation for $m$ from 1 to $i-2$, we deduce that
         \begin{align*}
              N_w^i|_{y=0}&=  \sum_{\substack{\ell_k\geq 0,\\ \sum_{ k=2}^ { i-1} \ell_k (k-1) = i-1 }} d_{ \ell_2,\dots, \ell_{i-1}} (\tau_0) e^{-\sum_{ k=2}^ { i-1}\ell_kg_k(\tau-\tau_0)},\\
              N_z^i|_{y=0}&= \sum_{\substack{\ell_k\geq 0,\\ \sum_{ k=2}^ { i-1} \ell_k (k-1) = i-1 }} \tilde d_{ \ell_2,\dots, \ell_{i-1}}(\tau_0) e^{-\sum_{ k=2}^ { i-1}\ell_kg_k(\tau-\tau_0)}.
        \end{align*}
        where both $d_{ \ell_2,\dots, \ell_{i-1}} (\tau_0)$ and $\tilde d_{ \ell_2,\dots, \ell_{i-1}} (\tau_0)$ are polynomials of $w_2(\tau_0),\dots w_{i-1}(\tau_0)$.
       Now, using 
          \begin{align*}
        e^{-r_1\tau} \int_{\tau_0}^{\tau} e^{r_1\tau'} e^{-r_2(\tau'-\tau_0)} \ d\tau' &= \frac{1}{r_1-r_2} \big(e^{-r_2(\tau-\tau_0)}-e^{-r_1(\tau-\tau_0)}\big),\quad r_1-r_2\neq 0
        \end{align*} and  
        \begin{align*}
         g_i-\sum_{2\leq m \leq i-1}\ell_mg_m=&\mu  \frac{2(\ga-1)}{\ga+1}(1-\sum_{2\leq m \leq i-1}\ell_m) \neq 0,\\
         k_i-\sum_{2\leq m \leq i-1}\ell_mg_m
        =& -\mu-\frac{2(\ga-1)}{\ga+1}\mu\sum_{2\leq m \leq i-1}m\ell_m\neq 0,
        \end{align*}
        we evaluate the integrals in \eqref{eq 3.25} and obtain \eqref{wjzj sol with condition}
       where
      \be\label{qm}
           q_i(\tau_0):= -\frac{3-\ga}{2(\ga-1)i+\ga+1} w_i(\tau_0) - \sum_{\sum_{ m=2}^ { i-1} \ell_m (m-1) = i-1 } \tilde d'_{ \ell_2,\dots, \ell_{j-1}} (\tau_0),
        \ee
        and $p_i(\tau_0)$ and $\tilde p_i(\tau_0)$ are polynomials of $w_2(\tau_0),\dots w_{i}(\tau_0)$. $ d'_{ \ell_2,\dots, \ell_{i-1}}(\tau_0)$,  $\tilde d'_{ \ell_2,\dots, \ell_{i-1}}(\tau_0)$, and $\tilde d''_{ \ell_2,\dots, \ell_{i-1}}(\tau_0)$ are polynomials of $w_2(\tau_0),\dots w_{i-1}(\tau_0)$.
        This completes the induction and the proof of the claim.

    \

    By using the claim and \eqref{gi>0}, and
    \begin{align*}
         g_j-\sum_{2\leq m \leq j-1}\ell_mg_m=&\mu  \frac{2(\ga-1)}{\ga+1}(1-\sum_{2\leq m \leq j-1}\ell_m)<0,
    \end{align*}
    for any $j\in\{2,\dots,\lceil \beta \rceil\}$, we deduce  \eqref{case 2.1}–\eqref{zi+1}. 
\end{proof}

The polynomials obtained in Lemma \ref{lemma 3.7} provide the precise constraints for the initial data of $z_i$'s, which suppress the potential growth driven by the anti-damping unless $i>\beta$ (see $k_i$ in  \eqref{g_i,k_i}), while $w_i$'s are damped and decay exponentially fast to 0 for all $i\ge 2$ (see $g_i$ in  \eqref{g_i,k_i}).

We now define the finite-dimensional constraint set $I^i$ for $i\ge 2$  consisting of $2(i-1)$-dimensional vectors $(z_2, \dots, z_i; w_2,\dots, w_i)$ satisfying the $(i-1)$ constraints $z_j= q_j (w_2, \dots,w_j)$ for $2\le j\le i$: 
\beqa\label{boundary condition}
  I^i = \left\{ (z_2, \dots, z_i; w_2,\dots, w_i): z_j= q_j (w_2, \dots,w_j) \text{ for  } 2\le j \le i  \right\}
\eeqa
where $q_j$ is a polynomial given in \eqref{qm}.
We then have the following result:
\begin{proposition}\label{trapping}
    Let $\ga\in(1,3)$ and $\mu\in[\frac12,1)$. Let the initial data $(z_j(\tau_0), w_j(\tau_0))$ for $2\le j\le \lceil \beta \rceil$ be given. If $(z_2(\tau_0), \dots, z_{\lfloor \beta \rfloor}(\tau_0); w_2(\tau_0),\dots, w_{\lfloor \beta \rfloor} (\tau_0)) \in I^{\lfloor \beta \rfloor} $,    
   then the following holds
    \be\label{3.38}
        |(z_j(\tau),w_j(\tau)) |\leq C\Big( e^{-g_2 (\tau-\tau_0)}+(\lceil \beta \rceil-\lfloor\beta\rfloor)e^{-k_{\lceil \beta \rceil} (\tau-\tau_0)} \Big)\ \text{ for any } j\in\{2, \dots, \lceil \beta \rceil\}\text{ and }\tau\geq\tau_0.
    \ee
    Otherwise, there exists a $j\in\{2, \dots, i\}$ such that
    \be
        z_j\to \infty \ \text{ as } \ \tau\to\infty.
    \ee
\end{proposition}
\begin{proof}
    It is a direct consequence of Lemma \ref{lemma 3.7} and the fact that for $k\geq 2$
    \[
        g_k -g_2 = (1-\frac{(3-\ga)\mu}{\ga+1})(k-2) \geq 0
    \]
    for any $\ga\in(1,3)$ and $\mu\in[\frac12,1)$.
\end{proof}
\begin{remark}\label{remark 3.7}
    This proposition implies that, in order to achieve asymptotic stability, the initial data must satisfy the finite-dimensional constraints, 
    imposed by the higher order boundary constraints as stated in \eqref{boundary condition}. The number of the constraints is $\lfloor \beta+1 \rfloor$: $\lfloor \beta \rfloor - 1$ are from the higher order constraints in $I^{\lfloor \beta\rfloor}$ and the remaining two are from $z_1=w_1=0$ related to the time translation symmetry. 
\end{remark}
\begin{remark}
    For $\mu\notin \ssmu$, there is no need to impose the constraints at the order $i=\lceil \beta \rceil$ since $k_{\lceil \beta \rceil}>0$. Consequently, competition in the decay rates arises between $g_2$ and $k_{\lceil \beta \rceil}$ whose relative magnitudes are undetermined; in particular, either quantity may dominate depending on $\ga$ and $\mu$. For notational convenience, we denote 
    \be\label{g0}
        \epzero = \min\{g_2, (\lfloor\beta\rfloor-\lceil\beta\rceil +1)g_2  +  (\lceil\beta\rceil -\lfloor\beta\rfloor)k_{\lceil \beta \rceil}\} 
    \ee so that $\epzero= g_2$ if $\mu \in \ssmu$ and $\epzero=k_{\lceil \beta \rceil} $ if $\mu\notin \ssmu$. 
  The inequality \eqref{3.38} can be written as 
  \be
        |(z_j,w_j)| \leq C e^{-\epzero\tau}.
  \ee
\end{remark}

\section{$(Z,W)$ equations and linear analysis} \label{section 4}

In this section, we derive the equations satisfied by $(Z,W)$ and analyze the associated linear system. Recall \eqref{sol decompose}. For notational convenience, we denote the first sums in the decomposition \eqref{sol decompose} by 
\be\label{def mz mw}
	\mz = \sum_{i=2}^{\ndmu-1} \varphi(y) \frac{z_i(\tau)}{i!}y^i \ \ \text{ and } \ \  \mw =  \sum_{i=2}^{\ndmu-1}  \varphi(y) \frac{w_i(\tau)}{i!}y^i.
\ee
Denote the linear operators appearing in \eqref{PDE of tilde z w} by 
\beqa\label{4.5}
	\lz &=  \Big(y+\frac{\ga+1}{4}\mathring z \Big) \partial_y   + \Big(\frac{\ga+1}{4}\partial_y \mathring z +\mu-1\Big),\\
	\lw &= \Big(y+\frac{3-\ga}{4}\mathring z \Big) \partial_y  + (\mu-1).
\eeqa
With these notations and from \eqref{PDE of tilde z w}, we may write the equations for $(Z,W)$ as the first order system 
\beqa\label{full system Z W}
	\partial_\tau \espz + \lz \espz +\frac{3-\ga}{4}\partial_y \mathring z \espw
     + \nz 
    + \nmz &= \mathcal S_z \\
	\partial_\tau \espw+ \lw \espw + \nw 
    + \nmw &=0 
\eeqa
where
\beqa\label{nz nmz}
    \mathcal S_z &=   \frac{\ga+1}{4} \bar z (1-\chi) \pa_y \mathring z ,\\ 
	\nz & =  (\frac{\ga+1}{4}\mz+\frac{3-\ga}{4}\mw+\frac{\ga+1}{4}\espz+\frac{3-\ga}{4}\espw)\partial_y\espz +(\frac{\ga+1}{4}\espz+\frac{3-\ga}{4} \espw)\pa_y\mz,\\
	\nmz & =\pa_\tau \mz+ \lz \mz+\frac{3-\ga}{4}\partial_y \mathring z \mw +(\frac{\ga+1}{4}\mz+\frac{3-\ga}{4}\mw)\partial_y\mz,\\
	\nw &= (\frac{3-\ga}{4}\espz+\frac{\ga+1}{4}\espw+\frac{3-\ga}{4}\mz+\frac{\ga+1}{2}\mw)\partial_y \espw+(\frac{\ga+1}{4}\espw+\frac{3-\ga}{4} \espz)\pa_y\mw ,\\
	\nmw &=\pa_\tau \mw+  \lw \mw+(\frac{3-\ga}{4}\mz+\frac{\ga+1}{4}\mw)\partial_y \mw.
\eeqa

In the rest of this section, we study the following linear inhomogeneous problem 
\be\label{linear PDE of tilde z w}
	\begin{aligned}
	\partial_\tau \espz +\lz \espz +\frac{3-\ga}{4}\partial_y \mathring z \ \espw 
    &= F_1,\\
		\partial_\tau \espw +\lw \espw &= F_2.
	\end{aligned}
\ee

We first present the following singular-weighted $L^2$ energy estimates.  

\begin{lemma}[Weighted $L^2$ estimates]\label{weighted L^2 estimates}
	Let $\ga\in(1,3)$ and $\mu\in[\frac12,1)$. Consider smooth solutions $(Z,W)$ to the transport equations \eqref{linear PDE of tilde z w} satisfying $(\pa_y^j Z,\pa_y^jW)|_{(\tau,0)} =0$ for all $\tau\in [\tau_0,\tau_1]$ and $0\le j\le \ndmu-1$. 
	Then there exist constants $b_0 >0$ and $\eps_0>0$  such that for all $b\ge b_0$ the following holds
	\beqa
		\frac12\pa_\tau(\norm{\frac{\espz}{y^{\ndmu}}}_{L^2}^2 + b \norm{\frac{\espw}{y^{\ndmu}}}_{L^2}^2) \leq -\sigma^2 (\norm{\frac{\espz}{y^{\ndmu}}}_{L^2}^2 +b  \norm{\frac{\espw}{y^{\ndmu}}}_{L^2}^2) + \left| \langle \frac{F_1}{y^{\ndmu}},\frac{\espz}{y^{\ndmu}} \rangle\right|  + b \left| \langle \frac{F_2}{y^{\ndmu}},\frac{\espw}{y^{\ndmu}} \rangle\right|  
	\eeqa
	where 
	\[
		0< \sigma^2 = (1-\mu)\ndmu-\frac32+\frac{\mu}{2}- \epsilon_0.
	\] 
\end{lemma}
\begin{proof}
    We start with the estimation of $\frac12 \pa_\tau\norm{\frac{\espz }{y^{\ndmu}}}_{L^2}^2$. Using the $Z$ equation in \eqref{linear PDE of tilde z w} and integrating by parts, we first obtain 
	\begin{align*}
		\frac12 \pa_\tau\norm{\frac{\espz }{y^{\ndmu}}}_{L^2}^2 &= \int_0^\infty -(y+\frac{\ga+1}{4} \mathring z) y^{-2\ndmu}\espz \pa_y \espz -  (\frac{\ga+1}{4}\partial_y \mathring z +\mu-1)|\frac{\espz}{y^{\ndmu}}|^2 \ dy \\
        &\quad - \int_0^\infty   \frac{3-\ga}{4}\partial_y \mathring z {y^{-2\ndmu}}{\espw \espz}  \ dy
		+\int_0^\infty  
        y^{-2\ndmu}F_1 {\espz}{}  \ dy 
		 \\ &\leq   \int_0^\infty   \Big(-(1+\frac{\ga+1}{4}\frac{\mathring z}{y})\ndmu+\frac32-\mu-\frac{\ga+1}{8} \pa_y\chi \bar z-\frac{\ga+1}{8} \chi\pa_y \bar z \Big) |\frac{\espz}{y^{\ndmu}}|^2 \ dy \\
		 &\quad + \frac12 \left| \langle  \partial_y \mathring z \frac{\espw}{y^\ndmu}, \frac{\espz}{y^\ndmu} \rangle \right| + \left| \langle \frac{F_1}{y^\ndmu},\frac{\espz}{y^\ndmu} \rangle\right|   
         \end{align*}
      Since $\pa_y\chi \bar z\ge 0$ and $1+\frac{\ga+1}{4}\frac{\mathring z}{y} \ge 1- \mu $, $|\pa_y \bar z |\le \frac{4\mu}{\ga +1} $ by Proposition \ref{prop bar z}, and by using the Young's inequality for the second crossing term, we deduce that for any small $\eps>0$,   
         \begin{align*}
         \frac12 \pa_\tau\norm{\frac{\espz }{y^\ndmu}}_{L^2}^2  
		& \leq  \Big(-(1-\mu)\ndmu+\frac32-\frac{1}{2}\mu+\epsilon\Big) \norm{\frac{\espz}{y^{\ndmu}}}_{L^2}^2 +C_\eps \norm{\frac{\espw}{y^{\ndmu}}}_{L^2}^2 + \left| \langle \frac{F_1}{y^\ndmu},\frac{\espz}{y^\ndmu} \rangle\right| .  
	\end{align*} 
Similarly, we compute $\frac12 \pa_\tau\norm{\frac{\espw}{y^{\ndmu}}}_{L^2}^2 $
	\begin{align*}
		\frac12 \pa_\tau\norm{\frac{\espw}{y^{\ndmu}}}_{L^2}^2 &= \int_0^\infty -(y+\frac{3-\ga}{4}\mathring z )y^{-2\ndmu}\espw \pa_y \espw - (\mu-1)|\frac{\espw}{y^\ndmu}|^2\ dy+ \int_0^\infty   
        y^{-2\ndmu}F_2 {\espw}  \ dy \\
		&\leq \int_0^\infty\Big( -(1+\frac{3-\ga}{4}\frac{\mathring z}{y} )\ndmu +\frac{3}{2}-\mu+\frac{3-\ga}{8}\pa_y \mathring z \Big) |\frac{\espw}{y^\ndmu}|^2\ dy+\left| \langle \frac{F_2}{y^\ndmu},\frac{\espw}{y^\ndmu} \rangle\right|.
	\end{align*}
    Observe that $1+\frac{3-\ga}{4}\frac{\mathring z}{y} \ge 1 - \frac{3-\ga}{1+\ga}\mu$,  $\pa_y \mathring z\le \pa_y \chi \bar z$, and also  $|\pa_y \chi \bar z|\le c/y_{0}^\mu $ for some $c>0$ that follows from 
    Proposition \ref{prop bar z} and Lemma \ref{le: supp chi}. Hence by taking $y_0$ large enough that $ \frac{3-\ga}{8} c/y_{0}^\mu \le \mu /2$,  we deduce that
    \begin{align*}
		\frac12 \pa_\tau\norm{\frac{\espw}{y^{\ndmu}}}_{L^2}^2 
		&\leq\Big( -(1-\frac{3-\ga}{\ga+1}\mu )\ndmu +\frac{3}{2}-\frac\mu2 
        \Big)  \norm{\frac{\espw}{y^{\ndmu}}}_{L^2}^2 +\left| \langle \frac{F_2}{y^\ndmu},\frac{\espw}{y^\ndmu} \rangle\right|
	\end{align*}
     Recall the definition of $\ndmu$ in \eqref{def ndmu}. Choose small enough $\eps_0>0$ and then sufficiently large $b>0$ such that 
    \[
    -(1-\frac{3-\ga}{\ga+1}\mu )\ndmu +\frac{3}{2}-\frac{\mu}{2} + \frac{C_{\epsilon_0}}{b}  < -(1-\mu)\ndmu+\frac32-\frac{1}{2}\mu+\epsilon_0 <0. 
    \]
    Combining the above two inequalities, we obtain the desired estimates. 
\end{proof}

Next we derive $\dot H^\ndmu$ estimates. To that end, we apply $\pa_y^{\ndmu}$ to \eqref{linear PDE of tilde z w} to obtain the equations for $(\pa_y^\ndmu \espz, \pa_y^\ndmu \espw)$: 
\be\label{linear PDE of der of tilde z w}
	\begin{aligned}
	&\partial_\tau\pa_y^{\ndmu} \espz + (y+\frac{\ga+1}{4}\mathring z ) \partial_y^{\ndmu+1} \espz + ((\ndmu+1) (1+ \frac{(\ga+1)}{4}\partial_y \mathring z)+\mu-2)\pa_y^{\ndmu}\espz  \\
&\quad +\frac{3-\ga}{4}\sum_{i=0}^{\ndmu} \binom{\ndmu}{i} \pa_y^{\ndmu+1-i} \mathring z \pa_y^{i} \espw +\frac{\ga+1}{4}\sum_{i=0}^{\ndmu-1} \binom{\ndmu+1}{i} \pa_y^{\ndmu+1-i} \mathring z \pa_y^i \espz	
= \pa_y^\ndmu F_1  ,\\
&		\partial_\tau \pa_y^{\ndmu}\espw + (y+\frac{3-\ga}{4}\mathring z ) \partial_y^{\ndmu+1} \espw + ( \ndmu ( 1+\frac{(3-\ga)}{4} \pa_y \mathring z)+\mu-1)\pa_y^{\ndmu}\espw   \\
 & \quad    +\frac{3-\ga}{4} \sum_{i=1}^{\ndmu-1} \binom{\ndmu}{i-1} \pa_y^{\ndmu+1-i} \mathring z \pa_y^{i} \espw 
 = \pa_y^\ndmu F_2.
	\end{aligned}
\ee
\begin{lemma}[$\dot H^{\ndmu}$ estimates] \label{linear H^n estimates}
Let $\ga\in(1,3)$ and $\mu\in[\frac12,1)$. Consider smooth solutions $(Z,W)$ where $(\pa_y^j Z,\pa_y^jW)|_{(\tau,0)} =0$ for all $\tau\in [\tau_0,\tau_1]$ and $0\le j\le \ndmu-1$ to the transport equations \eqref{linear PDE of tilde z w}. 
	Then there exist constants $b_1 >0$ and $\eps_1>0$  such that for all $b\ge b_1$ the following holds
\beqa
\begin{split}
	\frac12\pa_\tau \Big( \norm{\espz}_{\dot H^{\ndmu}}^2+b \norm{\espw}_{\dot H^{\ndmu}}^2 \Big)
	&\le  -\nu^2 \Big( \norm{\espz}_{\dot H^{\ndmu}}^2+ b\norm{\espw}_{\dot H^{\ndmu}} ^2\Big) +C (\norm{\frac{\espz}{y^{\ndmu}}}_{L^2}^2 + b\norm{\frac{\espw}{y^{\ndmu}}}_{L^2}^2) \\
    &\quad +\left|\langle \pa_y^\ndmu F_1, \pa_y^\ndmu \espz \rangle \right| +b \left|\langle \pa_y^\ndmu F_2, \pa_y^\ndmu \espw \rangle \right| 
\end{split}
\eeqa	
where  $C=C(\eps_1)>0$ is a constant and 
\[
0 < \nu^2 =  (1-\mu)\ndmu - \frac32 + \frac{\mu}{2}  -  \eps_1.
\] 
\end{lemma}

\begin{proof}
 We take the $L^2$ inner product of \eqref{linear PDE of der of tilde z w} with $(\pa_y^\ndmu \espz, b  \pa_y^\ndmu \espw )$ and integrate by parts to obtain 
 \begin{align*}
 	&\frac12\pa_\tau \Big( \norm{\espz}_{\dot H^{\ndmu}}^2+ b \norm{\espw}_{\dot H^{\ndmu}}^2 \Big)\\
	&= -\int_0^\infty \Big((\ndmu+\tfrac12) (1+\tfrac{(\ga+1)}{4}\partial_y \mathring z)+\mu- 2 \Big) |\pa_y^{\ndmu} \espz|^2  dy -b \int_0^\infty \Big((\ndmu-\tfrac12)(1+\tfrac{(3-\ga)}{4}\partial_y\mathring z) +\mu- 1 \Big) | \pa_y^{\ndmu} \espw|^2  dy   \\
	&- \int_0^\infty \tfrac{3-\gamma}{4} \pa_y \mathring z \pa_y^\ndmu \espw \pa_y^\ndmu \espz dy  - \int_0^\infty \Bigg( \tfrac{3-\ga}{4} \sum_{i=0}^{\ndmu-1}\binom{\ndmu}{i}\partial_y^{\ndmu+1-i} \mathring z  \pa_y^i \espw  +\tfrac{\ga+1}{4}\sum_{i=0}^{\ndmu-1} \binom{\ndmu+1}{i} \pa_y^{\ndmu+1-i} \mathring z \pa_y^i \espz  
    \Bigg)\pa_y^{\ndmu} \espz \ dy  \\
	&-b \int_0^\infty \tfrac{3-\ga}{4}\sum_{i=1}^{\ndmu-1} \binom{\ndmu}{i-1} \pa_y^{\ndmu+1-i} \mathring z \pa_y^{i} \espw \pa_y^{\ndmu} \espw \ dy +\int_0^\infty \pa_y^\ndmu F_1  \pa_y^\ndmu\espz \ dy +b \int_0^\infty \pa_y^\ndmu F_2  \pa_y^\ndmu\espw \ dy \\
    &=: (I) + (II) + (III)
\end{align*}
We estimate $(I), \ (II), \ (III)$ in what follows. For $(I)$, using $-\frac{4\mu}{\ga+1} \le\pa_y \mathring z < 0$ from Proposition \ref{prop bar z}, we have  
\begin{align*}
(I) \le - (\ndmu (1-\mu) + \tfrac{\mu}{2} - \tfrac32) \int_0^\infty |\pa_y^\ndmu \espz|^2 dy 
-  ( \ndmu(1- \mu) + \tfrac32 
\mu  -\tfrac32 ) b \int_0^\infty |\pa_y^\ndmu \espw|^2 dy
\end{align*}
The first term of $(II)$ can be bounded by 
\[
\left| \int_0^\infty \tfrac{3-\gamma}{4} \pa_y \mathring z \pa_y^\ndmu \espw \pa_y^\ndmu \espz dy \right| \le \eps_0  \norm{\espz}_{\dot H^{\ndmu}}^2 + C_{\eps_0} \norm{\espw}_{\dot H^{\ndmu}}^2
\]
where $\eps_0>0$ is to be determined. 
To handle the rest of commutator terms in $(II)$ and $(III)$, we first estimate   
\begin{align*}
    \sum_{i=0}^{\ndmu-1}{\int_0^\infty\Big( \pa_y^{\ndmu+1-i} \mathring z \pa_y^i \espz\Big)^2 \ dy} =:  {\sum_{i=0}^{\ndmu-1} S_z^i} =: S_z 
 \end{align*}
By applying Proposition \ref{prop: mathring z}, $\mathring z$ has different behavior in the different regions of $y$ for each $0\le i\le \ndmu-1$. Consequently, we decompose $S_z^i$ into two integrals: 
 \begin{align*}
 	S_z^i = \int_0^{y_*} \Big( \pa_y^{\ndmu+1-i} \mathring z \pa_y^i \espz\Big)^2 \ dy 
	+  \int_{y_*}^{2y_0e^{\tau-\tau_0}} \Big( \pa_y^{\ndmu+1-i} \mathring z \pa_y^i \espz\Big)^2 \ dy := S_{z,1}^i+S_{z,2}^i 
 \end{align*}
 where $y_*$ is small and to be determined. 
For $S_{z,1}^i$, we separate into two cases: $\mu\in \ssmu$ and $\mu\notin \ssmu$. 
When $\mu\in\ssmu$, as in Proposition  \ref{prop bar z} (3), $\pa_y \mathring z$ and $\pa_y^{ \beta } \mathring z $ are only bounded for small $y$. By using Hardy inequality (cf. Lemma \ref{hardy 1}), we have 
 \be\label{Sz1}
 	S_{z,1}^i \le  C \int_0^{y_*} | \pa_y^i \espz|^2 \ dy \le C  y_\ast^{2(\ndmu-i)}  \norm{\espz}_{\dot H^{\ndmu}}^2.
 \ee
 When $\mu\notin \ssmu$, we estimate 
\be\label{Sz1-2}
 	S_{z,1}^i \leq C\int_0^{y_*} \Big( y^{\beta - \ndmu -1 +i}\pa_y^i \espz\Big)^2 \ dy = C\int_0^{y_*} \Big( y^{\beta -1}\frac{\pa_y^i \espz}{y^{\ndmu-i}}\Big)^2 \ dy \leq C y_\ast^{2(\beta-1)}  \norm{\espz}_{\dot H^{\ndmu}}^2
 \ee
 where we have used $\beta>1 $. In both cases, we conclude that $ S_{z,1}^i \le  C y_* \norm{\espz}_{\dot H^{\ndmu}}^2$. 
Now for $S_{z,2}^i$, we use Proposition \ref{prop: mathring z} and  Lemma \ref{G-N with weight} 
to bound for $0\le i\le \ndmu-1$
 \beqa\label{Sz3}
 	 \int_{y_*}^{2y_0e^{\tau-\tau_0}} \Big( \pa_y^{\ndmu+1-i} \mathring z \pa_y^i \espz\Big)^2 \ dy &\leq C \int_{y_*}^{2y_0e^{\tau-\tau_0}} \Big( y^{-\mu-\ndmu+i}\pa_y^i \espz\Big)^2 \ dy \le C_* \norm{\frac{\espz}{y^\ndmu}}_{L^2}^{2- \frac{2i}{\ndmu}} \norm{\espz}_{\dot H^{\ndmu}}^{\frac{2i}{\ndmu}} 
 \eeqa
 where $C_*$ depends on $y_*$.  
Hence, we deduce that 
 \be\label{Sz}
 	S_z  \leq C y_*   \norm{\espz}_{\dot H^{\ndmu}}^2 + C_* \sum_{i=0}^{\ndmu-1} \norm{\frac{\espz}{y^\ndmu}}_{L^2}^{2- \frac{2 i}{\ndmu}} \norm{\espz}_{\dot H^{\ndmu}}^{\frac{2i}{\ndmu}}. 
\ee
Similarly, we can derive 
  \be\label{Sw}
 	\sum_{i=0}^{\ndmu-1}  \int_0^\infty\Big(  \pa_y^{\ndmu+1-i} \mathring z \pa_y^i \espw\Big)^2 \ dy \leq  C y_*   \norm{\espw}_{\dot H^{\ndmu}}^2 + C_* \sum_{i=0}^{\ndmu-1}  \norm{\frac{\espw}{y^\ndmu}}_{L^2}^{2- \frac{2 i}{\ndmu}} \norm{\espw}_{\dot H^{\ndmu}}^{\frac{2i}{\ndmu}}. 
 \ee
By putting the estimates together and using Young's inequality, we arrive at  
 \begin{align*}
 	&\frac12\pa_\tau \Big( \norm{\espz}_{\dot H^{\ndmu}}^2+ b \norm{\espw}_{\dot H^{\ndmu}}^2 \Big)\\
	&\le - (\ndmu (1-\mu) + \tfrac{\mu}{2} - \tfrac32 -\eps_0 - C y_* - C_*\eps ) \int_0^\infty |\pa_y^\ndmu \espz|^2 dy 
\\
&\quad -  ( \ndmu(1- \mu) + \tfrac32 
\mu  -\tfrac32 - \tfrac{C_{\eps_0}}{b} - \tfrac{C y_*}{b} - C_*\eps ) b \int_0^\infty |\pa_y^\ndmu \espw|^2 dy \\
& \quad + C_{*, \eps} \Big(  \norm{\frac{\espz}{y^{\ndmu}}}_{L^2}^2 + b  \norm{\frac{\espw}{y^{\ndmu}}}_{L^2}^2\Big) + \left|\langle \pa_y^\ndmu F_1, \pa_y^\ndmu \espz \rangle \right| +b \left|\langle \pa_y^\ndmu F_2, \pa_y^\ndmu \espw \rangle \right|
\end{align*}
  Recall the definition of $\ndmu$ in \eqref{def ndmu}. We may choose $\eps_0>0$, $y_*>0$ and $\eps>0$ small so that 
  \[
 \ndmu (1-\mu) + \tfrac{\mu}{2} - \tfrac32 -\eps_0 - C y_* - C\eps >0  
    \]
Then for sufficiently large $b>0$,  
    \[
    \ndmu(1- \mu) + \tfrac32 
\mu  -\tfrac32 - \tfrac{C_{\eps_0}}{b} - \tfrac{C y_*}{b} - C_*\eps \ge \ndmu (1-\mu) + \tfrac{\mu}{2} - \tfrac32 -\eps_0 - C y_* - C_*\eps >0  
    \]
    which implies the desired energy inequality. 
\end{proof}

\section{Nonlinear analysis}\label{section 5}

In this section, we analyze the full nonlinear system \eqref{full system Z W} based on energy estimates and bootstrap argument. In addition to the energy  estimates in $\dot H^\ndmu$ and singular weighted spaces, we must derive low order estimates to close nonlinear estimates by regulating suitable decay or growth of norms. 

The following is our key bootstrap estimates. 
\begin{proposition}\label{bootstrap lemma}
	Let $\ga\in(1,3)$ and $\mu\in[\frac12,1)$. Assume that initial modulation variables satisfy \eqref{initialziwi}. 
	Then there exist $y_0=e^{\tau_0}$, and $\sigma_0,\epzero,\epone, \eptwo >0$ 
    such that
	\be\label{sequence a_1, a_2, a_0}
		0< \eptwo<\epone<\min\{\frac{(4\mu-1)(1-\mu)}{2}, \frac{2(1-\mu)\ndmu-3+\mu}{6},\epzero\}
	\ee
	the following holds. 
    Assume that there exists a solution 
    $(\espz, \espw)\in H^\ndmu $ 
     to the
	the system \eqref{full system Z W} 
	on $[\tau_0,\tau_1]$, satisfies 
	and the bootstrap assumption
	\begin{align}
		\norm{(\espz, \espw)}_{\dot H^\ndmu}+\norm{(\frac{\espz}{y^\ndmu},\frac{\espw}{y^\ndmu})}_{L^2}\leq 2e^{-\epone \tau}\ \text{ and }\tau\in[\tau_0,\tau_1],\\
        \norm{(\pa_y\espz, \pa_y\espw)}_{L^2} \leq 2Be^{-(\mu-\frac12)\eptwo \tau}\ \text{ and }\tau\in[\tau_0,\tau_1],\\
		\norm{(\espz,\espw)}_{L^2} \leq 2Be^{(\frac32-\mu)\tau} \ \text{ and }\tau\in[\tau_0,\tau_1]. 
	\end{align}
	Then, for $\tau_0$ sufficiently large and $\norm{(\espz(\tau_0,\cdot), \espw(\tau_0,\cdot))}_{H^\ndmu}\leq \sigma_0$, we have the following improvement of the bootstrap assumptions:
	\begin{align}
		\norm{(\espz, \espw)}_{\dot H^\ndmu}+\norm{(\frac{\espz}{y^\ndmu},\frac{\espw}{y^\ndmu})}_{L^2} \leq e^{-\epone \tau}\ \text{ and }\tau\in[\tau_0,\tau_1],\label{payn z w}\\
        \norm{(\pa_y\espz, \pa_y\espw)}_{L^2} \leq Be^{-(\mu-\frac12)\eptwo \tau}\ \text{ and }\tau\in[\tau_0,\tau_1]\label{pay z w},\\
		\norm{(\espz,\espw)}_{L^2} \leq Be^{(\frac32-\mu)\tau} \ \text{ and }\tau\in[\tau_0,\tau_1]\label{L2}. 
	\end{align} 
\end{proposition}

\

The proof of Proposition \ref{bootstrap lemma} is  based on a series of estimates of modulation terms and a priori energy estimates for $(Z,W)$.

\subsection{Estimation of modulation and source terms}

We begin with the estimation of $\mz$ and $\mw$ followed by $\nmz$, $\nmw$. 

\begin{proposition}\label{prop payi mz mw}
	Let $\ga\in(1,3)$ and $\mu\in[\frac12, 1)$. 
    Assume that initial modulation variables satisfy \eqref{initialziwi}. Then for $i\in\{0,1,\dots,\ndmu+1\}$ and $p\in[1,\infty]$, we have 
	\be
		\norm{\frac{\mz}{y}}_{L^p}+\norm{\frac{\mw}{y}}_{L^p}+\norm{ \pa_y^i \mz}_{L^p}+\norm{\pa_y^i \mw }_{L^p} \leq Ce^{-\epzero \tau} 
	\ee
	where $\epzero$ is defined in \eqref{g0} and $C=:C(\ga,\mu)$ is a positive constant. 
\end{proposition}
\begin{proof}
    We begin by recalling the definitions of $\mz$ and $\mw$ from \eqref{def mz mw}. A direct computation yields
    \begin{align*}
        \frac{\mz}{y} &= \sum_{j=2}^{\ndmu-1}  \frac{z_j(\tau)}{j!}\varphi(y)y^{j-1},\\
        \pa_y^i\mz &= \sum_{j=2}^{\ndmu-1}  \frac{z_j(\tau)}{j!}\pa_y^i(\varphi(y)y^j) = \sum_{j=2}^{\ndmu-1}  \sum_{\ell=0}^{\min \{i,j\}} c_{i\ell }\frac{z_j(\tau)}{j!}\pa_y^{i-\ell}\varphi(y)y^{j-\ell}.
    \end{align*}
    Since $\varphi$ is compactly supported and using Proposition \ref{trapping}, we have
    \[
       \norm{\frac{\mz}{y}}_{L^p}+ \norm{\pa_y^i\mz}_{L_p} \leq C \sum_{j=2}^{\ndmu-1}  | z_j(\tau) | \leq C e^{-\epzero \tau} 
    \]
    where $C$ is a constant depending on $\ndmu$. The corresponding estimates for $\mw$ follow analogously.
 \end{proof}
\begin{lemma}\label{prop Tz Tw /y}
Let $\ga\in(1,3)$ and $\mu\in[\frac12, 1)$. 
Assume that initial modulation variables satisfy \eqref{initialziwi}. Then for $i\in\{0,1,\dots,\ndmu\}$, we have 
    \be
        \norm{ (\pa_y^i\nmz,\pa_y^i\nmw)}_{L^2} +\norm{( \frac{\nmz}{y^{\ndmu}}, \frac{\nmw}{y^{\ndmu}} )}_{L^2}\le C e^{-\epzero \tau}
    \ee
    where $\epzero$ is defined in \eqref{g0} and $C=:C(\ga,\mu)$ is a positive constant.
\end{lemma}
\begin{proof}
    From the definition of $\mz$ and $\mw$,  we claim that
    \[
        \pa_y^i\nmz(\tau,0) = \pa_y^i\nmw (\tau,0)=0
    \]
    for any $i\in\{0,1,\ndmu-1\}$. 
    To justify this, we expand $\nmz$ and $\nmw$ in powers of $y$, apply $\pa_y^i$ to the resulting expressions, and then evaluate at $y=0$. The terms that remain coincide exactly with the modulation equations for $(z_i,w_i)$ given in \eqref{ode of ziwi with z1=0}.
To illustrate the argument, we present the proof for $\nmz$, and the case of $\nmw$ follows analogously. 
    By recalling \eqref{nz nmz}, we rearrange $\nmz$ in terms of $y^i$: 
    \beqa\label{reorde nmz}
       \nmz &= \varphi\sum_{i=2}^{\ndmu-1}\frac{1}{i!}\Big(\pa_\tau z_i  + z_i (i +\mu-1)   +\sum_{\ell=2}^{i-1} \frac{\ell i!}{\ell! (i+1-\ell)!} z_{\ell}(\frac{3-\ga}{4} w_{i+1-\ell}+\frac{\ga+1}{4}z_{i+1-\ell}) \Big) y^i \\
        &+\varphi\sum_{i=2}^{\ndmu-1} \Big(\frac{z_i}{i!}(i\frac{\ga+1}{4}\frac{\mathring z}{y}+\frac{\ga+1}{4}\partial_y \mathring z ) y^i +\frac{w_i}{i!}\frac{3-\ga}{4}\partial_y \mathring z y^i\Big)  \\
        &+\pa_y\varphi\sum_{i=2}^{\ndmu-1}\frac{z_i}{i!} \Big((y+\frac{\ga+1}{4}\mathring z) y^i+  (\frac{\ga+1}{4}\mz+\frac{3-\ga}{4}\mw) y^i
        \Big) \\
      &+ \varphi \sum_{i=\ndmu}^{2\ndmu -3} \sum_{\ell=2}^{i-1} \frac{\ell i!}{\ell! (i+1-\ell)!} z_{\ell}(\frac{3-\ga}{4} w_{i+1-\ell}+\frac{\ga+1}{4}z_{i+1-\ell}) \Big) y^i
 \eeqa
    Let $m\in\{0,1,\dots,\ndmu-1\}$. We then apply $\pa_y^m$ to the above equation and evaluate at $y=0$:
    \begin{align*}
        \pa_y^m \nmz(\tau,0) &= \pa_\tau z_m + (i-1 -i\mu)z_m -\frac{(3-\ga)\mu}{\ga+1} w_m+  \sum_{\ell=2}^{m-2} \frac{\ell m!}{\ell! (m+1-\ell)!}z_\ell(\frac{3-\ga}{4} w_{i-\ell}+\frac{\ga+1}{4}z_{i-\ell}) \\
       &+\sum_{i=2}^{m-1}\frac{3-\ga}{4}  \binom{m}{i} \pa_y^{m+1-i} \mathring z(0)w_i+  \binom{m}{i} \frac{\ga+1}{4} \big(\pa_y^{m-i} \mathring z (0)z_{i+1}+ \pa_y^{m+1-i} \mathring z(0) z_i\big) =0 
    \end{align*}
       by using \eqref{ode of ziwi with z1=0}. By using the compactness of $\varphi$ and Proposition \ref{trapping} , the inequality holds for $i\in\{0,1,\dots,\ndmu-1\}$. 
   We then proceed to derive the $L^2$ integrability of $\frac{\nmz}{y^\ndmu}$ and $\pa_y^\ndmu \nmz$. For $\mu\in \ssmu$, the statement holds trivially, since $\nmz$ is smooth and compactly supported. It therefore suffices to consider the case $\mu\notin \ssmu$. In \eqref{reorde nmz}, the first and fourth lines consist of monomials $y^i$ with integer $i$, while the third line vanishes at $y=0$ due to the factor $\pa_y\varphi$. Hence, these contributions are integrable. 
 The remaining terms in the second line involve $\pa_y \mathring z\, y^{i}$ and $\mathring z\, y^{i-1}$, which are only H\"{o}lder continuous for $y\ll 1$, since $\mathring z(y)\sim y+y^\beta$ as $y\ll 1$ by Proposition \ref{prop bar z}. These terms require separate treatment. 
 For $y\ll 1$, we have
    \begin{align*}
    \frac{\mathring z y^{i-1}}{y^\ndmu},  \frac{\pa_y\mathring z y^{i}}{y^\ndmu},   \pa_y^{\ndmu} \big(\mathring z y^{i-1} \big),\  \pa_y^{\ndmu} \big(\pa_y \mathring z y^{i}\big) \sim y^{\beta+i-1-\ndmu}.
   \end{align*}
 Since the initial modulation variables satisfy \eqref{initialziwi} and from $\ndmu$ defined in \eqref{def ndmu}, we have $\beta+i-1-\ndmu >-\frac12$ for $i\geq 2$, which ensures the square integrability.  This completes the proof.
\end{proof}

We now estimate far-field source terms $\mathcal S^n_z$. 

\begin{lemma}\label{estimate of Schi}
	Let $\ga\in(1,3)$ and $\mu\in [\frac12,1)$.
    Then,
	\beqa
		 \int_0^\infty (\pa_y^i\big(  \bar z (1-\chi) \pa_y \mathring z\big))^2  \ dy  \leq C y_0^{3-4\mu-2i} e ^{-(1-\mu)(2i +4\mu -3) (\tau-\tau_0)}, 
	\eeqa
	for each $1\le i\le \ndmu$. 
\end{lemma}
\begin{proof}
	From Lemma \ref{le: supp chi} and the definition $\mathring z$ in \eqref{def mathring z}, we observe that
	\[
		\text{supp}\big(\bar z (1-\chi) \pa_y \mathring z\big) \subset [y_0e^{(1-\mu)(\tau-\tau_0)}, 2y_0e^{(\tau-\tau_0)})\ \text{ for }\tau\geq \tau_0.
	\]
	Therefore, by using Proposition \ref{prop bar z} and Proposition  \ref{prop: mathring z}, we have $|\pa_y^i\big(  \bar z (1-\chi) \pa_y \mathring z\big)| \lesssim y^{1-2 \mu - i }$ for $y>y_0$, and in turn 
	\begin{align*}
		 \int_0^\infty (\pa_y^i\big(  \bar z (1-\chi) \pa_y \mathring z\big))^2  \ dy
         \lesssim \int_{y_0e^{(1-\mu)(\tau-\tau_0)}}^{2y_0e^{(\tau-\tau_0)}} y^{2-4\mu-2i } \ dy \leq C y_0^{3-4\mu-2i} e ^{-(1-\mu)(2i +4\mu -3) (\tau-\tau_0)}.
	\end{align*}
\end{proof}

\subsection{Energy estimates}
\subsubsection{High order and singular weighted estimates}

We now proceed with key energy estimates, which will allow us to improve bootstrap assumptions in Proposition \ref{bootstrap lemma}. The following lemma is relevant to the improvement of \eqref{payn z w}. 

\begin{lemma}\label{ode Hn}
	Let $\ga\in(1,3)$ and $\mu\in[\frac12,1)$. Under the assumptions of Proposition  \ref{bootstrap lemma}, there exist $C:=C(\ga,\mu)>0$ and $b>0$ such that
	\beqa\label{Hn ineq}
		 &\frac12\pa_\tau \Big( \norm{\espz}_{\dot H^{\ndmu}}^2+b \norm{\espw}_{\dot H^{\ndmu}}^2 +  C(\norm{\frac{\espz}{y^{\ndmu}}}_{L^2}^2 + b\norm{\frac{\espw}{y^{\ndmu}}}_{L^2}^2) \Big)\\
	&\le   -\frac{1}{2}\Big((1-\mu)\ndmu - \frac32 + \frac{\mu}{2} \Big)\Big( \norm{\espz}_{\dot H^{\ndmu}}^2+ b\norm{\espw}_{\dot H^{\ndmu}} ^2 +C (\norm{\frac{\espz}{y^{\ndmu}}}_{L^2}^2 + b\norm{\frac{\espw}{y^{\ndmu}}}_{L^2}^2\Big) \\
    &+C y_0^{\frac32-2\mu-\ndmu} e ^{-(1-\mu)(\ndmu +2\mu -\frac{3}{2}) (\tau-\tau_0)} e^{-\epone \tau} +Ce^{-\epzero \tau}e^{-\epone \tau}\\
     &+ CB^{1-\frac{1}{2(\ndmu-1)}}e^{-(\mu-\frac12)(1-\frac{1}{2(\ndmu-1)})\eptwo\tau} e^{-(2+\frac{1}{2(\ndmu-1)})\epone \tau}.
	\eeqa
\end{lemma}
\begin{proof}
    In the following proof, we will use $C$ to denote a generic positive constant that only depends
on $\ndmu$, $\ga$ and $\mu$. It may vary from line to line.
    From Lemma \ref{weighted L^2 estimates} and Lemma \ref{linear H^n estimates}, it remains to establish  the estimates of $\left| \langle \frac{F_1}{y^\ndmu},\frac{\espz}{y^\ndmu} \rangle\right|$,  $\left| \langle \frac{F_2}{y^\ndmu},\frac{\espw}{y^\ndmu} \rangle\right| $ , $\left|\langle \pa_y^\ndmu F_1, \pa_y^\ndmu \espz \rangle \right|$, and $ \left|\langle \pa_y^\ndmu F_2, \pa_y^\ndmu \espw \rangle \right|$. Here $F_1= \mathcal S_z - \nz - \nmz$ and $F_2= - \nw - \nmw$ where $\mathcal S_z$, $\nz$, $\nmz$, $\nw$ and $\nmw$ are defined in \eqref{nz nmz}. 
    Since the estimation corresponding to $F_2$ is the same as the one to $F_1$, we will only present the estimation of $\left|\langle \pa_y^\ndmu F_1, \pa_y^\ndmu \espz \rangle \right|$ and $\left| \langle \frac{F_1}{y^\ndmu},\frac{\espz}{y^\ndmu} \rangle\right|$. 
    
	\noindent\underline{The source term $\mathcal S_z$}:
    First of all, $\pa_y^j\mathcal S_z(\tau,0)=0$ for any $j\geq 0$ since it has a compact support away from $y=0$. Hence, by Lemma \ref{hardy 1}, 
    \[
        \norm{y^{-\ndmu}\mathcal S_z}_{L^2} \leq C\norm{\pa_y^{\ndmu}\mathcal S_z}_{L^2}.
    \]
    It then follows that
    \[
    \begin{split}
      \Big|  \int_0^\infty \frac{\mathcal S_z}{y^{\ndmu}} \frac{\espz}{y^{\ndmu}} \ dy \Big| + \Big|\int_0^\infty \pa_y^\ndmu\mathcal S_z \pa_y^\ndmu \espz \ dy \Big| 
       &\leq C\norm{\pa_y^\ndmu\mathcal S_z}_{L^2}\norm{\pa_y^\ndmu \espz}_{L^2} \\
       & \leq C {y_0^{\frac32-2\mu-\ndmu}}e ^{-(1-\mu)(\ndmu +2\mu -\frac{3}{2}) (\tau-\tau_0)} e^{-\epone \tau}.
   \end{split}
    \]
    where we have used Young's inequality in the first step, Lemma \ref{estimate of Schi}, and the bootstrap assumption in Proposition \ref{bootstrap lemma} in the last step. 
    
	\noindent\underline{The nonlinear term $\nz$}: 
    For the weighted $L^2$ estimates, we proceed as 
    \begin{align*}
      & \Big| \int_0^\infty \frac{\nz}{y^{\ndmu}} \frac{\espz}{y^{\ndmu}} \ dy \Big| = \Big|\frac{\ga+1}{4}\int_0^\infty(\pa_y\espz+\pa_y\mz) (\frac{\espz}{y^{\ndmu}})^2 \ dy +\frac{3-\ga}{4}\int_0^\infty (\pa_y \espz+\pa_y \mz)\frac{\espz \espw}{y^{2\ndmu}} \ dy\\
        &\qquad\qquad\qquad\qquad\quad + \int_0^1(\frac{\ga+1}{4}\frac{\mz}{y} +\frac{3-\ga}{4}\frac{\mw}{y} ) \frac{\pa_y \espz}{y^{\ndmu-1}}\frac{ \espz}{y^{\ndmu}} \ dy\Big| \\
       & \leq C\Big( \norm{\frac{\mz}{y}}_{L^\infty}+\norm{\frac{\mw}{y}}_{L^\infty}+\norm{\pa_y\mz}_{L^\infty}+\norm{\pa_y\espz}_{L^\infty}\Big) (\norm{\frac{\pa_y \espz}{y^{\ndmu-1}}}_{L^2}^2+\norm{\frac{\espz}{y^{\ndmu}}}_{L^2}^2+\norm{\frac{\espw}{y^{\ndmu}}}_{L^2}^2)\\
       &\leq C\Big(e^{-\epzero \tau}++B^{1-\frac{1}{2(\ndmu-1)}}e^{-(\mu-\frac12)(1-\frac{1}{2(\ndmu-1)})\eptwo\tau} e^{-\frac{1}{2(\ndmu-1)}\epone \tau}\Big) e^{-2\epone \tau}
    \end{align*}
    where we have used Proposition \ref{prop payi mz mw}, Lemma \ref{hardy 1}, Lemma \ref{Linfty decay}, and the bootstrap assumption in Proposition \ref{bootstrap lemma} in the last inequality.
    For the $\dot H^{\ndmu}$ estimates,
	\begin{align*}
	    \int_0^\infty \pa_y^{\ndmu}\nz \pa_y^{\ndmu} \espz dy&= \sum_{i=0}^{\ndmu} c_{\ndmu i} \int_0^\infty (\frac{\ga+1}{4}\pa_y^i\mz+\frac{3-\ga}{4}\pa_y^i\mw+\frac{\ga+1}{4}\pa_y^i\espz+\frac{3-\ga}{4}\pa_y^i\espw)\partial_y^{\ndmu+1-i}\espz\pa_y^{\ndmu} \espz dy\\
        &+\sum_{i=0}^{\ndmu} c_{\ndmu i} \int_0^\infty \pa_y^{\ndmu+1-i}\mz(\frac{\ga+1}{4}\pa_y^i\espz +\frac{3-\ga}{4}\pa_y^i\espw)\pa_y^{\ndmu}\espz \ dy\\
        &=: I+II.
	\end{align*}
	We first obtain an estimate of $I$ by considering three cases, depending on the value of $i$. 
	When $i=0$, by integrating by parts, we obtain
	\begin{align*}
		 &\Big|\int_0^\infty (\frac{\ga+1}{4}\mz+\frac{3-\ga}{4}\mw+\frac{\ga+1}{4}\espz+\frac{3-\ga}{4}\espw)\partial_y^{\ndmu+1}\espz\pa_y^\ndmu \espz dy\Big| \\
		 \leq&\frac12 \int_0^\infty |\frac{\ga+1}{4}\pa_y\mz+\frac{3-\ga}{4}\pa_y\mw+\frac{\ga+1}{4}\pa_y\espz+\frac{3-\ga}{4}\pa_y\espw|\mid\pa_y^\ndmu \espz\mid^2 dy \\
		 \leq & C\Big(\norm{\pa_y\mz}_{L^\infty}+\norm{\pa_y\mw}_{L^\infty} +\norm{\pa_y\espz}_{L^\infty}+\norm{\pa_y\espw}_{L^\infty}\Big) \norm{\pa_y^\ndmu\espz}_{L^2}^2\\
		 \leq &  C\Big(e^{-\epzero \tau}+B^{1-\frac{1}{2(\ndmu-1)}}e^{-(\mu-\frac12)(1-\frac{1}{2(\ndmu-1)})\eptwo\tau} e^{-\frac{1}{2(\ndmu-1)}\epone \tau}\Big) e^{-2\epone \tau}
	\end{align*}
	where we used Proposition \ref{prop payi mz mw}, Lemma \ref{Linfty decay}, and the bootstrap assumption in Proposition \ref{bootstrap lemma} in the last inequality. When $i=\ndmu$, 
    \begin{align*}
        &\Big|\int_0^\infty (\frac{\ga+1}{4}\pa_y^\ndmu\mz+\frac{3-\ga}{4}\pa_y^\ndmu\mw+\frac{\ga+1}{4}\pa_y^\ndmu\espz+\frac{3-\ga}{4}\pa_y^\ndmu\espw)\partial_y\espz\pa_y^{\ndmu} \espz dy\Big|\\
        \leq & C \norm{\pa_y \espz}_{L^\infty}\big(\norm{\pa_y^{\ndmu} \mz}_{L^2}+\norm{\pa_y^{\ndmu} \mw}_{L^2}+\norm{\pa_y^{\ndmu} \espz}_{L^2}+\norm{\pa_y^{\ndmu} \espw}_{L^2} \big) \norm{\pa_y^{\ndmu} \espz}_{L^2}\\
        \leq & C \norm{\pa_y \espz}_{L^2}^{1-\frac{1}{2(\ndmu-1)}}\norm{\pa_y^\ndmu \espz}_{L^2}^{\frac{1}{2(\ndmu-1)}}\big(\norm{\pa_y^{\ndmu} \mz}_{L^2}+\norm{\pa_y^{\ndmu} \mw}_{L^2}+\norm{\pa_y^{\ndmu} \espz}_{L^2}+\norm{\pa_y^{\ndmu} \espw}_{L^2} \big) \norm{\pa_y^{\ndmu} \espz}_{L^2}\\
        \leq & CB^{1-\frac{1}{2(\ndmu-1)}} e^{-(1-\frac{1}{2(\ndmu-1)})(\mu-\frac12)\eptwo \tau} e^{-(2+\frac{1}{2(\ndmu-1)})\epone \tau}
    \end{align*}
    where we have applied Lemma \ref{Linfty decay} in the second inequality, and Proposition \ref{prop payi mz mw}, \eqref{sequence a_1, a_2, a_0}, and the bootstrap assumption in Proposition \ref{bootstrap lemma} in the third inequality. 
	Considering $1\leq i\leq \ndmu-1$, we estimate by using Gagliardo–Nirenberg interpolation inequality and Lemma \ref{Linfty decay}:
	\begin{align}
		&\int_0^\infty |(\pa_y^i\mw+\pa_y^i\espw)\partial_y^{\ndmu+1-i}\espz\pa_y^\ndmu \espz | dy\notag \\
        &\leq C (\norm{\pa_y^i\mw}_{L^\infty}+\norm{\pa_y^i\espw}_{L^\infty}) \norm{\partial_y^{\ndmu+1-i}\espz}_{L^2} \norm{\pa_y^\ndmu \espz}_{L^2}\notag\\
		&\leq C (\norm{\pa_y^i\mw}_{L^\infty}+\norm{\pa_y \espw}_{L^2}^{1-\frac{2i-1}{2(\ndmu-1)}} \norm{\pa_y^\ndmu \espw}_{L^2}^{\frac{2i-1}{2(\ndmu-1)}} )\norm{\pa_y \espz}_{L^2}^{1-\frac{\ndmu-i}{\ndmu-1}} \norm{\pa_y^\ndmu \espz}_{L^2}^{\frac{\ndmu-i}{\ndmu-1}} \norm{\pa_y^\ndmu \espz}_{L^2}\notag\\
        &\leq  CB^{\frac{i-1}{\ndmu-1}} e^{-\frac{i-1}{\ndmu-1}(\mu-\frac12)\eptwo \tau} e^{-(1+\frac{\ndmu-i}{\ndmu-1})\epone \tau} e^{-\epzero \tau}+CB^{1-\frac{1}{2(\ndmu-1)}} e^{-(1-\frac{1}{2(\ndmu-1)})(\mu-\frac12)\eptwo \tau} e^{-(2+\frac{1}{2(\ndmu-1)})\epone \tau}\notag\\
        &\leq CB^{1-\frac{1}{2(\ndmu-1)}} e^{-(1-\frac{1}{2(\ndmu-1)})(\mu-\frac12)\eptwo \tau} e^{-(2+\frac{1}{2(\ndmu-1)})\epone \tau}\label{est 4.28}
        \end{align}
where we have used $1\leq i \leq \ndmu$, \eqref{sequence a_1, a_2, a_0} and the assumption in Proposition \ref{bootstrap lemma} to obtain the last inequality.
	The remaining two terms can be estimated analogously. Combining these estimates, we obtain
	\begin{align*}
		|I| \leq  e^{-\epzero\tau} e^{-2\epone\tau}+ CB^{1-\frac{1}{2(\ndmu-1)}} e^{-(1-\frac{1}{2(\ndmu-1)})(\mu-\frac12)\eptwo \tau} e^{-(2+\frac{1}{2(\ndmu-1)})\epone \tau}.
	\end{align*}
For $II$, when $1\leq i\leq \ndmu$, it can be estimated in the same manner as the first term in \eqref{est 4.28}. For $i=0$, we estimate
\begin{align*}
   \Big| \int_0^\infty\pa_y^{\ndmu+1}\mz (\frac{\ga+1}{4}\espz +\frac{3-\ga}{4}\espw)\pa_y^{\ndmu}\espz \ dy\Big|  
    &\leq \int_0^1 \Big|\pa_y^{\ndmu+1}\mz(\frac{\ga+1}{4}\frac{\espz}{y^{\ndmu}} +\frac{3-\ga}{4}\frac{\espw}{y^{\ndmu}})\pa_y^{\ndmu}\espz \Big| \ dy \\
    &\leq C \norm{\pa_y^{\ndmu+1}\mz}_{L^\infty}(\norm{\pa_y^{\ndmu}\espz}_{L^2}^2+\norm{\pa_y^{\ndmu}\espw}_{L^2}^2)\\
    &\leq C e^{-\epzero\tau} e^{-2\epone\tau}
\end{align*}
where we have used Young's inequality and Lemma \ref{hardy 1} in the second inequality, and the bootstrap assumption in Proposition \ref{bootstrap lemma} to obtain the last inequality. Combining all the estimations above, we derive
\[
     \int_0^\infty \Big| \frac{\nz}{y^{\ndmu}} \frac{\espz}{y^{\ndmu}}\Big| +| \pa_y^{\ndmu}\nz \pa_y^{\ndmu} \espz| \ dy  \leq C   e^{-\epzero\tau} e^{-2\epone\tau}  + CB^{1-\frac{1}{2(\ndmu-1)}} e^{-(1-\frac{1}{2(\ndmu-1)})(\mu-\frac12)\eptwo \tau} e^{-(2+\frac{1}{2(\ndmu-1)})\epone \tau}.
\]

 	\noindent\underline{The nonlinear term $\nmz$:} By using H\"{o}lder inequality, we obtain
	\begin{align*}
		\Big| \int_0^\infty \pa_y^\ndmu\nmz\pa_y^\ndmu \espz dy\Big|+\Big|\int_0^\infty \frac{\nmz}{y^{\ndmu}} \frac{\espz}{y^{\ndmu}} dy \Big|
		\le \norm{\pa_y^\ndmu \nmz}_{L^2} \norm{\espz}_{\dot H^\ndmu}+\norm{\frac{\nmz^0}{y^{\ndmu}}}_{L^2} \norm{\frac{\espz}{y^{\ndmu}}}_{L^2}
        \leq C e^{-\epzero \tau} e^{-\epone \tau}
	\end{align*}
    where we have applied Lemma \ref{prop Tz Tw /y} and the bootstrap assumption in Proposition \ref{bootstrap lemma} in the last step. 
   Combining all the estimations above, choosing $\epsilon_0=\epsilon_1=\frac{1}{2}\big((1-\mu)\ndmu - \frac32 + \frac{\mu}{2} \big)$ in Lemma \ref{weighted L^2 estimates} and Lemma \ref{linear H^n estimates}, and using the fact $e^{-\epzero \tau}e^{-2\epone \tau} \leq e^{-\epzero \tau}e^{-\epone \tau}$, we arrive at \eqref{Hn ineq}.
\end{proof}

\subsubsection{$\dot H^1 $ estimates}
We next present key $\dot H^1$ estimates, which are relevant to the improvement of \eqref{pay z w}. 

\begin{lemma}\label{lemma nonlinear H^1}
	 Let $\ga\in(1,3)$ and $\mu\in[\frac12,1)$. Under the assumptions of Proposition \ref{bootstrap lemma}, there exists a constant $C:=C(\ga,\mu)>0$ such that
	\begin{align*}
    \pa_\tau\Big(\norm{\pa_y\espz}_{L^2}^2+\norm{\pa_y\espw}_{L^2}^2\Big)
			&\leq  -(2\mu-1 ) \Big(\norm{\pa_y \espz}_{L^2}^2 +\norm{\pa_y \espw}_{L^2} ^2\Big) + C   e^{-\epone \tau} + CB^2 e^{-\frac{\mu\epone}{2(\ndmu-1)} \tau }e^{-(2\mu-1)\eptwo \tau} \\
          &+CB y_0^{\frac{-4\mu +1}{2}} 
            e^{\frac{(-4\mu+1)(1-\mu)}{2}(\tau-\tau_0)}e^{-(\mu-\frac12)\eptwo\tau}.
		 \end{align*}
\end{lemma}

\begin{proof} 
From \eqref{full system Z W}, the equations for $(\pa_y Z, \pa_y W)$ read as  
	\beqa\label{full system pay Z W}
	\partial_\tau\pa_y \espz + (y+\frac{\ga+1}{4}\mathring z ) \partial_y^{2} \espz + (\frac{(\ga+1)}{2}\partial_y \mathring z+\mu)\pa_y\espz +\frac{3-\ga}{4}\pa_y(\pa_y \mathring z \espw)
    +\frac{\ga+1}{4}\pa_y^2 \mathring z \espz& \\
+\pa_y\nz
+ \pa_y\nmz& = \pa_y\mathcal S_z,\\
		\partial_\tau \pa_y\espw + (y+\frac{3-\ga}{4}\mathring z ) \partial_y^2 \espw + (\frac{(3-\ga)}{4}\pa_y \mathring z+\mu)\pa_y\espw  + \pa_y\nw 
        + \pa_y\nmw & = 0.
\eeqa
The energy estimates of \eqref{full system pay Z W} lead to 
\begin{align*}
    \frac12\pa_\tau\Big(\norm{\pa_y\espz}_{L^2}^2+  \norm{\pa_y\espw}_{L^2}^2\Big) = I+ II + III
\end{align*}
where
\begin{align*}
	I&= \int_{0}^\infty -\Big(\frac{3(\ga+1)}{8}\partial_y \mathring z+\mu-\frac{1}{2} \Big) \mid \pa_y \espz\mid^2 - \Big(\frac{(3-\ga)}{8}\partial_y\mathring z+\mu-\frac{1}{2} \Big) \mid \pa_y \espw\mid^2 \ dy  \\
	II&=-\int_{0}^\infty \Bigg( \frac{3-\ga}{4}\pa_y\Big(\partial_y \mathring z  \espw\Big) +\frac{\ga+1}{4} \pa_y^2 \mathring z  \espz
    \Big)  \Bigg)\pa_y \espz \ dy \\
	III&=  \int_{0}^\infty \pa_y\mathcal S_z \pa_y Z dy- \int_{0}^\infty \Big( \pa_y\nz 
    + \pa_y\nmz \Big) \pa_y Z dy  - \int_{0}^\infty\Big( \pa_y\nw 
    + \pa_y\nmw\Big) \pa_y W dy.
	\end{align*}
    We present the estimation of $I$, $II$, $III$ in what follows. 
    
    \noindent\underline{The term $I$:} We first observe that  
    \[
    \begin{split}
    I &= - (\mu - \frac12) (\norm{\pa_y\espz}_{L^2}^2+ b \norm{\pa_y\espw}_{L^2}^2 ) - \int_{0}^\infty \frac{3(\ga+1)}{8}\partial_y \mathring z \mid \pa_y \espz\mid^2 +  \frac{(3-\ga)}{8}\partial_y\mathring z  \mid \pa_y \espw \mid^2 \ dy 
    \end{split}
    \] 
    It is sufficient to estimate the last integral and we will focus on the first term as the second one follows analogously. By recalling \eqref{profile be} from Proposition \ref{prop: mathring z},  
    	$$ |\pa_y\mathring z|\leq \begin{cases}
		A,\ \text{for } y\in[0,1],\\
		Ay^{-\mu},\ \text{for } y \ge 1,
	\end{cases}$$
	for some constant $A:=A(\ga,\mu)$,   
   we split the integral into two: 
   \[
  \int_{0}^\infty |\partial_y \mathring z |   \mid \pa_y \espz\mid^2  dy  \le A\int_{0}^{e^{\alpha\tau}}  \mid \pa_y \espz\mid^2  dy  +A \int_{e^{\alpha\tau}}^{\infty}
 y^{-\mu} \mid \pa_y \espz\mid^2   dy =: A (I_1+  I_2)
    \]
   where $\alpha>0$ is small to be determined. For $I_1$, by Hardy inequality (cf. Lemma \ref{hardy 1}),  
    \[
I_1 = \int_{0}^{e^{\al\tau}}y^{2(\ndmu-1)}\mid \frac{\pa_y \espz}{y^{\ndmu-1}}\mid^2 \ dy  
		\lesssim e^{2(\ndmu-1)\al\tau} \norm{\espz}_{\dot H^\ndmu}^2
    \]
 For $I_2$, we bound by 
 \[
 I_2 \le e^{-\mu \alpha \tau} \norm{\pa_y \espz}_{L^2} ^2 
 \]
	Hence, by the bootstrap assumption of Proposition \ref{bootstrap lemma}, we arrive 
	\beqa\label{I-1}
     \int_{0}^\infty\frac{3(\ga+1)}{8}  |\partial_y \mathring z|   \mid \pa_y \espz\mid^2  dy   \le C e^{(2(\ndmu-1)\al-2\epone) \tau}  + C B^2 e^{-(\mu\al+(2\mu-1)\eptwo)\tau} 
	\eeqa and  
   similarly, 
	\beqa\label{I-2}
      \int_{0}^\infty\frac{(3-\ga)}{8}  | \partial_y \mathring z  | \mid \pa_y \espw\mid^2  dy  \le C e^{(2(\ndmu-1)\al-2\epone) \tau}  + C B^2 e^{-(\mu\al+(2\mu-1)\eptwo)\tau}
	\eeqa
     for some constant $C>0$. 
	Therefore, combining \eqref{I-1} and \eqref{I-2}, we have 
	\beqa\label{I}
		I \leq -(\mu-\frac{1}{2} ) \Big(\norm{\pa_y \espz}_{L^2}^2 +\norm{\pa_y \espw}_{L^2} ^2\Big) + C e^{(2(\ndmu-1)\al-2\epone) \tau} + CB^2 e^{-(\mu\al+(2\mu-1)\eptwo)\tau}.
	\eeqa
	
   \noindent \underline{The term $II$:} We estimate $II$ term by term: 
	\begin{align*}
		II 
		&=-\frac{3-\ga}{4} \int_{0}^\infty  \pa_y \mathring z\pa_y \espw \pa_y \espz dy - \frac{3-\ga}{4}\int_{0}^\infty \pa_y^2 \mathring z\espw \pa_y \espz dy- \frac{\ga+1}{4} \int_{0}^\infty \pa_y^2 \mathring z \espz\pa_y \espz dy.
	\end{align*}
    The estimation of the first term is similar to the previous case $I$:   
	\beqa\label{II-2}
		\left| \int_{0}^\infty  \pa_y \mathring z\pa_y \espw \pa_y \espz dy \right|  &\leq \frac12  \int_{0}^\infty | \pa_y \mathring z | \mid\pa_y \espw \mid^2 dy  +\frac12   \int_{0}^\infty | \pa_y \mathring z | \mid\pa_y \espz \mid^2 dy \\
		&\leq C e^{(2(\ndmu-1)\al-2\epone) \tau} + CB^2 e^{-(\mu\al+(2\mu-1)\eptwo)\tau}.
	\eeqa
	For the second term, we apply Holder inequality and obtain
	\begin{align*}
		\left| \int_{0}^\infty \pa_y^2 \mathring z\espw \pa_y \espz dy \right| \leq \Big(\int_{0}^\infty \mid\pa_y^2 \mathring z  \espw\mid^2  dy\Big)^{\frac12} \Big(\int_{0}^\infty\mid\pa_y \espz \mid^2 dy\Big)^{\frac12}
	\end{align*}
	As in the previous case,  
    due to the different behaviors of $\pa_y^2\mathring z$ for $y\ll 1$ and $1\ll y$
	 from Proposition \ref{prop: mathring z}: 
	$$ |\pa_y^2\mathring z|\leq \begin{cases}
		Ay^{\beta-2},\ \text{for } y \in [0,1],\\
		Ay^{-\mu-1},\ \text{for }  y \ge  1,
	\end{cases}$$
	we  separate the domain of integration 
	\begin{align*} 
       &\int_{0}^\infty \mid\pa_y^2 \mathring z\espw\mid^2 dy = \int_0^{1}\mid\pa_y^2 \mathring z\espw\mid^2 dy + \int_{1}^{e^{\al\tau}} \mid\pa_y^2 \mathring z\espw\mid^2 dy +\int_{e^{\al\tau}}^\infty\mid\pa_y^2 \mathring z\espw\mid^2 dy \\
        &\leq A\int_0^{1} y^{2\beta-4} \mid \espw\mid^2 dy +  A\int_{1}^{e^{\al\tau}} y^{-2\mu-2}\mid\espw\mid^2 dy + A\int_{e^{\al\tau}}^\infty y^{-2\mu-2}\mid\espw\mid^2 dy\\
        &\leq A (1+ e^{2(\ndmu-1-\mu)\alpha\tau}) \int_0^{e^{\al\tau}}\mid\frac{\espw}{y^\ndmu}\mid^2 dy + Ae^{-2\mu\al\tau}\int_{e^{\al\tau}}^\infty \mid  \frac{\espw}{y}\mid^2 \ dy
	\end{align*}
   By using 
   the bootstrap assumptions of Proposition  \ref{bootstrap lemma}, 
	\be\label{II-1}
		\left| \int_{0}^\infty \pa_y^2 \mathring z\espw \pa_y \espz dy \right| \leq \Big( C  e^{((\ndmu-1-\mu)\al -\epone )\tau }+C B e^{-\mu\al\tau}e^{-(\mu-\frac12)\eptwo\tau} \Big) B e^{-(\mu-\frac12)\eptwo\tau}.
	\ee
	The third term can be estimated in the same way as in the second term by noting that
    \[
		\left| \int_{0}^\infty \pa_y^2 \mathring z\espz \pa_y \espz dy \right| \leq \Big(\int_{0}^\infty \mid\pa_y^2 \mathring z  \espz \mid^2  dy\Big)^{\frac12} \Big(\int_{0}^\infty\mid\pa_y \espz \mid^2 dy\Big)^{\frac12}
    \]
    We deduce that 
    \be\label{II-3}
		\left| \int_{0}^\infty \pa_y^2 \mathring z\espz \pa_y \espz dy \right| \leq \Big( C  e^{((\ndmu-1-\mu)\al -\epone )\tau }+C B e^{-\mu\al\tau}e^{-(\mu-\frac12)\eptwo\tau} \Big) B e^{-(\mu-\frac12)\eptwo\tau}.
	\ee
 Combining  \eqref{II-2}, \eqref{II-1}, and \eqref{II-3}, 
         we conclude 
		 \beqa\label{II}
		 	II \leq & \  C e^{( 2(\ndmu-1)\al-2\epone) \tau} 
            + C B^2 e^{-(\mu\al + (2\mu-1)\eptwo)\tau}   
		 \eeqa
	for some constant $C>0$. Here we have used the Cauchy-Schwartz inequality to bound the linear-in-$B$ term in \eqref{II-1}, and \eqref{II-3} by the right-hand-side of \eqref{II}. 
    
		\noindent \underline{The term $III$}: The first term of $III$ can be bounded by using H\"{o}lder inequality, Lemma \ref{estimate of Schi} and the bootstrap assumptions of Proposition \ref{bootstrap lemma}:
		 \beqa\label{II-4}
		 	\Big| \int_{0}^\infty \pa_y\mathcal S_z \pa_y Z dy\Big| & \leq \norm{\pa_y\mathcal S_z}_{L^2} \norm{\pa_y Z}_{L^2}
			\leq CB y_0^{\frac{-4\mu +1}{2}}  
            e^{\frac{(-4\mu+1)(1-\mu)}{2}(\tau-\tau_0)}e^{-(\mu-\frac12)\eptwo\tau}.
		 \eeqa
The estimation of remaining two integrals in $III$ is the same. Therefore, we only present the estimation of the first remaining integral involving $\pa_y\nz$ and $\pa_y\nmz$. We first split $\pa_y\nz$ into two parts: 
		 \begin{align*}
         \pa_y\nz &  = \pa_y\Big( (\frac{\ga+1}{4}\mz+\frac{3-\ga}{4}\mw+\frac{\ga+1}{4}\espz+\frac{3-\ga}{4}\espw)\partial_y\espz\Big)  +\pa_y\Big( (\frac{\ga+1}{4}\espz+\frac{3-\ga}{4} \espw)\pa_y\mz\Big)\\
        &  = \nz^{1,1} + \nz^{1,2}
         \end{align*}
    For $\int \nz^{1,1} \pa_y \espz dy $,  by integration by parts, we obtain 
		 \begin{align*}
		 	- \int_0^\infty \nz^{1,1} \pa_y \espz dy 
			&= - \frac12 \int_0^\infty  \pa_y(\frac{\ga+1}{4}\mz+\frac{3-\ga}{4}\mw+\frac{\ga+1}{4}\espz+\frac{3-\ga}{4}\espw)\mid\partial_y\espz \mid^2 dy \\
			&\leq A ( \norm{\pa_y \mz}_{L^\infty} +\norm{\pa_y \mw}_{L^\infty}) \int_0^{1}|\pa_y Z|^2 dy   + A  (  \norm{\pa_y \espz}_{L^\infty} +\norm{\pa_y \espw}_{L^\infty} ) \norm{\pa_y Z}_{L^2}^2.
		 \end{align*}
		 We observe that $\int_0^1|\pa_y Z|^2 dy \lesssim \int_0^1|\frac{\pa_y Z}{y^{\ndmu-1}}|^2 dy \lesssim \norm{\espz}_{\dot H^\ndmu}^2 $ by Hardy inequality. By using Proposition \ref{prop payi mz mw}, Lemma  \ref{Linfty decay} 
         and  the assumptions of Proposition \ref{bootstrap lemma}, we further bound 
		 $$
		 	- \int_0^\infty \nz^{1,1} \pa_y \espz dy \leq C e^{-\epzero \tau} e^{-2\epone \tau} 
            + C B^{3- \frac{1}{2(\ndmu-1)}} e^{\frac{-\epone}{2(\ndmu-1)}\tau} e^{-\left(3-\frac{1}{2(\ndmu-1)}\right)(\mu-\frac12)\eptwo \tau}.
		 $$
		 Next, since $\text{supp}(\mz), \text{supp}(\mw) \in [0,1]$, 
         by using Young's inequality, we bound
		 \begin{align*}
         	& - \int_0^\infty \nz^{1,2} \pa_y \espz dy\\
            &= - \int_0^\infty  (\frac{\ga+1}{4}\espz+\frac{3-\ga}{4} \espw)\pa_y^2\mz  \pa_y \espz dy - \int_0^\infty (\frac{\ga+1}{4}\pa_y \espz+\frac{3-\ga}{4} \pa_y \espw)\pa_y\mz \pa_y \espz dy \\ 
			&\leq (\norm{\pa_y^2\mz}_{L^\infty}+\norm{\pa_y\mz}_{L^\infty})\Big( \int_0^{1} \mid \espz\mid^2+ \mid \espw\mid^2 dy + \int_0^{1}  \mid \pa_y\espw\mid^2 +  \mid \pa_y\espz\mid^2 dy)\\
			&\leq (\norm{\pa_y^2\mz}_{L^\infty}+\norm{\pa_y\mz}_{L^\infty}) 
            \Big( \int_0^{1}  \mid \frac{\espz}{y^\ndmu}\mid^2+ \mid \frac{\espw}{y^\ndmu}\mid^2  dy + \int_0^{1}  \mid \frac{\pa_y\espw}{y^{\ndmu-1}}\mid^2 dy+\int_0^{1}  \mid \frac{\pa_y\espz}{y^{\ndmu-1}}\mid^2 dy)\\
			&\leq (\norm{\pa_y^2\mz}_{L^\infty}+\norm{\pa_y\mz}_{L^\infty}) 
            \Big( \norm{\espw}_{\dot H^\ndmu}^2+\norm{\espz}_{\dot H^\ndmu}^2\Big)
		 \end{align*}
		 where we have used 
         Lemma \ref{hardy 1} in the last step. Hence,
		 $$
		 	- \int_0^\infty \nz^{1,2} \pa_y \espz dy 
            \leq C e^{-\epzero \tau}e^{-2\epone \tau}.
		 $$
		 For the last piece, recalling the expression \eqref{nz nmz} of $\pa_y\nmz$ 
		 \begin{align*}
		 	- \int_0^\infty   \pa_y\nmz \pa_y Z  dy &= - \int_0^1\pa_y\Big(\pa_\tau \mz+  \lz \mz+\frac{3-\ga}{4}\partial_y \mathring z \mw +(\frac{\ga+1}{4}\mz+\frac{3-\ga}{4}\mw)\partial_y\mz\Big) \pa_y Z  dy \\
			&\leq C e^{-\epzero \tau} (\int_0^{1} |\frac{\pa_y Z}{y^{\ndmu-1}}|^2 dy)^\frac12 \le  C e^{-\epzero \tau} \norm{\espz}_{\dot H^\ndmu}  
			\leq  C e^{-(\epzero+ \epone )\tau}  
		 \end{align*}
	where we used Proposition \ref{prop payi mz mw} in the first inequality and Hardy inequality (cf. Lemma \ref{hardy 1}) 
    in the second inequality. Hence,  
	\beqa\label{III}
	 III&\leq CB y_0^{\frac{-4\mu +1}{2}} 
            e^{\frac{(-4\mu+1)(1-\mu)}{2}(\tau-\tau_0)}e^{-(\mu-\frac12)\eptwo\tau}+ C e^{-(\epzero+\epone) \tau}\\ 
            &\quad + C B^{3-\frac{1}{2(\ndmu-1)}} e^{\frac{-\epone}{2(\ndmu-1)}\tau} e^{-\left(3-\frac{1}{2(\ndmu-1)}\right)(\mu-\frac12)\eptwo \tau}.
	\eeqa
		 
         Putting the estimates together from \eqref{I}, \eqref{II} and \eqref{III}, and choosing $\al = \frac{\epone}{2(\ndmu-1)}$ and  with $\epone<\epzero$, we arrive at 
		 \beqa\notag 
\frac12&\pa_\tau \Big(\norm{\pa_y\espz}_{L^2}^2+\norm{\pa_y\espw}_{L^2}^2\Big)\\ 
			\leq & -(\mu-\frac{1}{2} ) \Big(\norm{\pa_y \espz}_{L^2}^2 +\norm{\pa_y \espw}_{L^2} ^2\Big)  + C e^{(2(\ndmu-1)\al-2\epone) \tau} + CB^2 e^{-\mu\al \tau }e^{-(2\mu-1)\eptwo \tau} \\ 
			&
           +CB y_0^{\frac{-4\mu +1}{2}} 
            e^{\frac{(-4\mu+1)(1-\mu)}{2}(\tau-\tau_0)}e^{-(\mu-\frac12)\eptwo\tau} + C e^{-(\epzero+\epone )\tau}
            + C B^{3-\frac{1}{2(\ndmu-1)}} e^{\frac{-\epone}{2(\ndmu-1)}\tau} e^{-\left(3-\frac{1}{2(\ndmu-1)}\right)(\mu-\frac12)\eptwo \tau}  \\ 
          \leq & -(\mu-\frac{1}{2} ) \Big(\norm{\pa_y \espz}_{L^2}^2 +\norm{\pa_y \espw}_{L^2} ^2\Big)  + C   e^{-\epone \tau} + CB^2 e^{-\mu\al \tau }e^{-(2\mu-1)\eptwo \tau} \\
          &+CB y_0^{\frac{-4\mu +1}{2}} 
            e^{\frac{(-4\mu+1)(1-\mu)}{2}(\tau-\tau_0)}e^{-(\mu-\frac12)\eptwo\tau}
		 \eeqa
        At the second inequality we used   $e^{-(\epzero+\epone )\tau} \le e^{(2(\ndmu-1)\al-2\epone) \tau}  =  e^{-\epone \tau}$ and   $e^{\frac{-\epone}{2(n-1)}\tau} e^{-\left(3-\frac{1}{2(n-1)}\right)(\mu-\frac12)\eptwo \tau}  \le e^{-\mu\al \tau }e^{-(2\mu-1)\eptwo \tau} $. 
		 \end{proof}

\subsubsection{$L^2$ estimates}

Final $L^2$ energy estimates provide the growth control of the $L^2$ norm in $\tau$, which will be crucial for the improvement of \eqref{L2}. 

\begin{lemma}\label{lemma L2}
    Let $\ga\in(1,3)$ and $\mu\in[\frac{1}{2},1)$.  Under the assumptions of Proposition \ref{bootstrap lemma}, we have
    \begin{align}
        \pa_\tau\Big(\norm{\espz}_{L^2}^2+\norm{\espw}_{L^2}^2\Big)\leq &(3-2\mu) (\norm{\espz}_{L^2}^2 +  \norm{\espw}_{L^2}^2) +CB^2 e^{-\frac{\mu\epone}{2\ndmu-\mu}\tau} e^{(3-2\mu)\tau}+CBe^{(\frac{3}{2}-\mu)\tau}(G[\mu])^{\frac12}\notag \\
            &+B^{1-\frac{1}{2(\ndmu-1)}} e^{-(1-\frac{1}{2(\ndmu-1)})(\mu-\frac12)\eptwo \tau}e^{-\frac{1}{2(\ndmu-1)}\epone\tau}B^2 e^{(3-2\mu)\tau} \label{l2 main}
    \end{align}
    where $G[\mu]$ is defined in \eqref{Fmu}.
\end{lemma}
\begin{proof}
     By following a similar computation step 
     as in Lemma \ref{weighted L^2 estimates}, we obtain
     \[
     \begin{split}
\frac12\pa_\tau(\norm{\espz}_{L^2}^2 +  \norm{\espw}_{L^2}^2) \leq & (\frac32-\mu) (\norm{\espz}_{L^2}^2 +  \norm{\espw}_{L^2}^2)+ \int_0^\infty    |\pa_y \mathring z| |\espz|^2+|\pa_y \mathring z | |\espw|^2 +| \pa_y\mathring z| \espw\espz\ dy \\
&+\left| \langle F_1,\espz \rangle\right|+\left| \langle F_2,\espw \rangle\right|.
    \end{split}
     \]
     We first estimate the quadratic term on the first line. By recalling Proposition \ref{prop: mathring z}
     $$ |\pa_y\mathring z|\leq \begin{cases}
		A,\ \text{for } y\in[0,1],\\
		Ay^{-\mu},\ \text{for } y \ge 1,
	\end{cases}$$
    We split the integral domains into three pieces:
    \beqa\label{L2 linear}
        \int_0^\infty |\pa_y\mathring z| \espz^2 \ dy &= \int_0^1 |\pa_y\mathring z| \espz^2 \ dy+ \int_1^{e^{\al\tau}}|\pa_y\mathring z |\espz^2 \ dy +\int_{e^{\al\tau}}^\infty |\pa_y\mathring z |\espz^2 \ dy \\
        &\lesssim \int_0^1  \espz^2 \ dy+ \int_1^{e^{\al\tau}} y^{-\mu}\espz^2 \ dy +\int_{e^{\al\tau}}^\infty  y^{-\mu} \espz^2 \ dy \\
        &\lesssim (1+e^{(2\ndmu-\mu)\al\tau})\int_0^{e^{\al\tau}}  \left|\frac{\espz}{y^\ndmu}\right|^2 \ dy + e^{-\mu\al\tau} \norm{\espz}_{L^2}^2 \\
        &\lesssim e^{(2\ndmu-\mu)\al\tau-2\epone \tau}+B^2 e^{-\mu\al\tau} e^{(3-2\mu)\tau}.
    \eeqa
    where $\al$ will be determined later and we have used the assumptions of Proposition \ref{bootstrap lemma} in the last inequality.
        The remaining two terms are estimated in a similar manner, with the cross terms treated first by applying Young’s inequality. We now proceed to estimate the terms regarding $F_1$ and $F_2$. Since the estimate for $F_2$ follows in an analogous manner to that of $F_1$, we only present the detailed estimate for $F_1$. Recall the expression 
           $ F_1 = \mathcal S_z -\nz - \nmz$ 
        where $\mathcal S_z$, $\nz$ and $\nmz$ are defined in \eqref{nz nmz}. We estimate term by term.

        \noindent\underline{The source term $\mathcal S_z$}: By using Holder inequality,  
        \begin{align*}
          \Big|  \int_0^\infty \bar z (1-\chi) \pa_y \mathring z \espz \ dy \Big| \lesssim \int_{y_0e^{(1-\mu)(\tau-\tau_0)}}^{2y_0e^{\tau-\tau_0}}y^{1-2\mu} |Z| \ dy &\lesssim \norm{\espz}_{L^2} \big(\int_{y_0e^{(1-\mu)(\tau-\tau_0)}}^{2y_0e^{\tau-\tau_0}} y^{2-4\mu} \ dy\big)^{\frac12}.
        \end{align*}
        Now, depending on $\mu$, the integral of $y^{2-4\mu}$ is solved differently. 
        \beqa\label{Fmu}
            \int_{y_0e^{(1-\mu)(\tau-\tau_0)}}^{2y_0e^{\tau-\tau_0}} y^{2-4\mu} \ dy \leq \begin{cases}
                Cy_0^{3-4\mu} e^{(3-4\mu)(\tau-\tau_0)}, \ \text{ when }\mu\in[\frac12,\frac34),\\
                \ln(2y_0e^{\tau-\tau_0}),  \ \text{ when }\mu=\frac34,\\
                Cy_0^{3-4\mu}e^{(3-4\mu)(1-\mu)(\tau-\tau_0)}, \ \text{ when }\mu\in(\frac34,1).
            \end{cases}
        \eeqa
        For notational convenience, we denote the right-hand side of \eqref{Fmu} by $G[\mu]$. Hence, by using the assumptions of Proposition \ref{bootstrap lemma}, we demonstrate
        \be\label{L2 Sz}
            \Big| \int_0^\infty \bar z (1-\chi) \pa_y \mathring z \espz \ dy \Big| \leq CBe^{(\frac{3}{2}-\mu)\tau} (G[\mu])^{\frac12}.
        \ee
        
       \noindent \underline{The nonlinear term $\nz$}:
        We first do the integration by parts, 
        \begin{align*}
           & \int_0^\infty  (\frac{\ga+1}{4}\mz+\frac{3-\ga}{4}\mw+\frac{\ga+1}{4}\espz+\frac{3-\ga}{4}\espw)\partial_y\espz \espz+(\frac{\ga+1}{4}\espz+\frac{3-\ga}{4} \espw)\pa_y\mz \espz \ dy \\
          =&-\frac12\int_0^\infty  (\frac{\ga+1}{4}\pa_y\mz+\frac{3-\ga}{4}\pa_y\mw+\frac{\ga+1}{4}\pa_y\espz+\frac{3-\ga}{4}\pa_y\espw)\left| \espz\right|^2+(\frac{\ga+1}{4}\espz+\frac{3-\ga}{4} \espw)\pa_y\mz \espz \ dy \\
          \lesssim & \Big(\norm{\pa_y\mz}_{L^\infty}+\norm{\pa_y\mw}_{L^\infty}+\norm{\pa_y\espz}_{L^\infty}+\norm{\pa_y\espw}_{L^\infty}\Big)(\norm{\espz}_{L^2}^2 +\norm{\espw}_{L^2}^2 )
        \end{align*}
        where we have used H\"{o}lder inequality in the last step. Next, we apply the assumptions of Proposition \ref{bootstrap lemma}, Lemma \ref{prop payi mz mw} , and Lemma \ref{Linfty decay} to obtain
        \be\label{L2 Nz}
            \Big| \int_0^\infty \nz \espz \ dy \Big| \lesssim (e^{-\epzero \tau}+B^{1-\frac{1}{2(\ndmu-1)}} e^{-(1-\frac{1}{2(\ndmu-1)})(\mu-\frac12)\eptwo \tau}e^{\frac{-1}{2(\ndmu-1)}\epone\tau})B^2 e^{(3-2\mu)\tau}.
        \ee
        
        \noindent\underline{The nonlinear term $\nmz$}:
        By using H\"{o}lder inequality:
        \be\label{l2 Tz}
          \Big|  \int_0^\infty \nmz \espz \ dy\Big|  \lesssim \norm{\nmz}_{L^2} \norm{\espz}_{L^2} \lesssim B e^{-\epzero \tau}e^{(\frac{3}{2}-\mu)\tau} 
        \ee
        where Lemma \ref{prop Tz Tw /y} and the assumptions of Proposition \ref{bootstrap lemma} have been used. Combing \eqref{L2 linear}, \eqref{L2 Sz}, \eqref{L2 Nz}, \eqref{l2 Tz}, and choosing $\al=\frac{\epone}{2\ndmu-\mu}$, we derive
        \begin{align*}
            \frac12\pa_\tau(\norm{\espz}_{L^2}^2 +  \norm{\espw}_{L^2}^2) \leq & (\frac32-\mu) (\norm{\espz}_{L^2}^2 +  \norm{\espw}_{L^2}^2) +Ce^{-\epone \tau}+CB^2 e^{-\frac{\mu\epone}{2\ndmu-\mu}\tau} e^{(3-2\mu)\tau}\\
            &+CBe^{(\frac{3}{2}-\mu)\tau} (G[\mu])^{\frac12}+B e^{-\epzero \tau}e^{(\frac{3}{2}-\mu)\tau}\\
            &+C(e^{-\epzero \tau}+B^{1-\frac{1}{2(\ndmu-1)}} e^{-(1-\frac{1}{2(\ndmu-1)})(\mu-\frac12)\eptwo \tau}e^{-\frac{1}{2(\ndmu-1)}\epone\tau})B^2 e^{(3-2\mu)\tau}\\
            \leq & (\frac32-\mu) (\norm{\espz}_{L^2}^2 +  \norm{\espw}_{L^2}^2) +CB^2 e^{-\frac{\mu\epone}{2\ndmu-\mu}\tau} e^{(3-2\mu)\tau}+CBe^{(\frac{3}{2}-\mu)\tau}(G[\mu])^{\frac12} \\
            &+B^{1-\frac{1}{2(\ndmu-1)}} e^{-(1-\frac{1}{2(\ndmu-1)})(\mu-\frac12)\eptwo \tau}e^{-\frac{1}{2(\ndmu-1)}\epone\tau}B^2 e^{(3-2\mu)\tau}
        \end{align*}
        where $C$ is a constant. At the last inequality, we have used $e^{-\epzero\tau}<e^{-\epone\tau}\leq e^{-\frac{\mu\epone}{2\ndmu-\mu}\tau}$. 
\end{proof}

We are now ready to prove Proposition  \ref{bootstrap lemma}. 

\subsection{Proof of Proposition  \ref{bootstrap lemma}}
        We begin with \eqref{pay z w}. 
        
   \noindent \underline{Proof of \eqref{pay z w}}: From Lemma \ref{lemma nonlinear H^1}, we have
    \begin{align*}
    \pa_\tau\Big(e^{(2\mu-1)\tau}(\norm{\pa_y\espz}_{L^2}^2+\norm{\pa_y\espw}_{L^2}^2)\Big)\leq &   C   e^{(2\mu-1-\epone) \tau} + CB^2 e^{-\frac{\epone \mu}{2(\ndmu-1)} \tau }e^{(2\mu-1)(1-\eptwo) \tau} \\
          &+CB y_0^{\frac{-4\mu +1}{2}} 
            e^{\frac{(-4\mu+1)(1-\mu)}{2}(\tau-\tau_0)}e^{(\mu-\frac12)(2-\eptwo)\tau}.
    \end{align*}
    Integrating the above ODE inequality over $[\tau_0, \tau]$, we derive
    \begin{align*}
        &\norm{(\pa_y\espz(\tau,\cdot), \pa_y \espw(\tau,\cdot))}_{L^2}^2 \\
        &\leq e^{-(2\mu-1)(\tau-\tau_0)} \norm{(\pa_y\espz(\tau_0,\cdot), \pa_y \espw(\tau_0,\cdot))}_{L^2}^2 
         + C  e^{-(2\mu-1)\tau}\int_{\tau_0}^{\tau} e^{(2\mu-1-\epone) \tau'} \ d\tau'\\
        &\quad + CB^2 e^{-(2\mu-1)\tau}\int_{\tau_0}^{\tau} e^{-\frac{\epone \mu}{2(\ndmu-1)} \tau' }e^{(2\mu-1)(1-\eptwo) \tau'} d\tau'\\
        &\quad + CB y_0^{\frac{-4\mu +1}{2}}e^{\frac{-(-4\mu+1)(1-\mu)}{2}\tau_0} e^{-(2\mu-1)\tau}\int_{\tau_0}^{\tau}
            e^{\frac{(-4\mu+1)(1-\mu)}{2}\tau'}e^{(\mu-\frac12)(2-\eptwo)\tau'}d\tau'\\
        &=: I^1+II^1+III^1+IV^1.
    \end{align*}
We analyze the terms sequentially, following the order of choosing $\tau_0$, $\sigma_0$, and $y_0$. We first analyze $III^1$. We observe that the sign of $-\frac{\epone \mu}{2(\ndmu-1)} +(2\mu-1)(1-\eptwo)$ is indeterminate. To verify it, we recall the definition \eqref{def ndmu} of $\ndmu$. In particular, we note that $0<\frac{\epone \mu}{2(\ndmu-1)} \to 0$ as $\mu\to 1$, while $\eptwo$ is small. Consequently, the sign of the above expression may vary with the value of $\mu$, and it is thus natural to distinguish the following two cases: when $-\frac{\epone \mu}{2(\ndmu-1)} +(2\mu-1)(1-\eptwo)\neq 0$,
\begin{align*}
    III^1 &= 
        B^2e^{-(2\mu-1)\eptwo \tau}\frac{ C e^{-\frac{\epone \mu}{2(\ndmu-1)}\tau_0 }}{-\frac{\epone \mu}{2(\ndmu-1)} +(2\mu-1)(1-\eptwo)}\Big(  e^{-\frac{\epone \mu}{2(\ndmu-1)}(\tau-\tau_0) }-e^{-(2\mu-1)(1-\eptwo)(\tau- \tau_0)}\Big)\\
        &\leq  B^2e^{-(2\mu-1)\eptwo \tau}\frac{ C e^{-\frac{\epone \mu}{2(\ndmu-1)}\tau_0 }}{\left|-\frac{\epone \mu}{2(\ndmu-1)} +(2\mu-1)(1-\eptwo)\right|},\ \text{ for any }\tau\in[\tau_0,\tau_1],
\end{align*}
where we have used $\frac{\epone \mu}{2(\ndmu-1)}>0$ and $(2\mu-1)(1-\eptwo)\geq 0$ in the last step. And, when $-\frac{\epone \mu}{2(\ndmu-1)} +(\mu-\frac12)(2-\eptwo)= 0$, we have $(2\mu-1)(1-\eptwo)>0$ and
$$
    III^1 =e^{-(2\mu-1)\eptwo \tau} CB^2 e^{-(2\mu-1)(1-\eptwo)\tau}(\tau-\tau_0)\leq  B^2e^{-(2\mu-1)\eptwo\tau} \frac{C e^{-(2\mu-1)(1-\eptwo)\tau_0}}{e(2\mu-1)(1-\eptwo)}
$$
where we have used the fact that $e^{-ax}x\leq \frac{1}{ae}$ with $a>0$ for any $x\geq 0$.
We then choose $\tau_0$ sufficiently large such that 
        \beqa\label{B}
       \frac{ C e^{-\frac{\epone \mu}{2(\ndmu-1)}\tau_0 }}{\left|-\frac{\epone \mu}{2(\ndmu-1)} +(2\mu-1)(1-\eptwo)\right|} \leq \frac{1}{4}, \ \text{when }\left|-\frac{\epone \mu}{2(\ndmu-1)} +(\mu-\frac12)(2-\eptwo)\right| \neq 0,\\
       \frac{C e^{-(2\mu-1)(1-\eptwo)\tau_0}}{e(2\mu-1)(1-\eptwo)} \leq \frac{1}{4}, \ \text{when }\left|-\frac{\epone \mu}{2(\ndmu-1)} +(\mu-\frac12)(2-\eptwo)\right| = 0,
    \eeqa
   to obtain
    \be\label{H1 III}
        III^1 \leq \frac{1}{4}B^2e^{-(2\mu-1)\eptwo \tau}.
    \ee
    We next analyze $II^1$. Similarly, $2\mu-1\in[0,1]$ and $\epone$ is small. The sign of $2\mu-1-\epone$ is indeterminate as well. Hence, we split it into two cases: when $2\mu-1-\epone \neq 0$,
    \begin{align*}
       II^1&= e^{-(2\mu-1)\eptwo\tau}\frac{C}{2\mu-1-\epone}\Big( e^{-(\epone-(2\mu-1)\eptwo) \tau}- e^{-(2\mu-1)(1-\eptwo)\tau}e^{(2\mu-1-\epone)\tau_0}\Big)\\
       &\leq e^{-(2\mu-1)\eptwo\tau}\frac{C}{\left|2\mu-1-\epone\right|}\Big( e^{-(\epone-(2\mu-1)\eptwo) \tau}+ e^{-(2\mu-1)(1-\eptwo)\tau}e^{-\left| 2\mu-1-\epone\right|\tau_0}\Big).
    \end{align*}
When $2\mu-1-\epone=0$, then $(2\mu-1)(1-\eptwo)\neq 0$ due to \eqref{sequence a_1, a_2, a_0}, and
$$
    II^1= e^{-(2\mu-1)\eptwo\tau} C e^{-(2\mu-1)(1-\eptwo)\tau} (\tau-\tau_0) \leq e^{-(2\mu-1)\eptwo\tau}\frac{Ce^{-1-(2\mu-1)(1-\eptwo)\tau_0}}{(2\mu-1)(1-\eptwo)}
$$
where we also used $e^{-ax}x\leq \frac{1}{ae}$ with $a>0$ for any $x\geq 0$ in the last inequality. Now, by choosing $\tau_0$ sufficient large such that
    \beqa\label{tau0}
        \frac{C}{\left|2\mu-1-\epone\right|}\Big( e^{-(\epone-(2\mu-1)\eptwo) \tau}+ e^{-(2\mu-1)(1-\eptwo)\tau}e^{-\left| 2\mu-1-\epone\right|\tau_0}\Big)\leq\frac{1}{4}B^2,\ \text{when }2\mu-1-\epone\neq 0,\\
        \frac{Ce^{-1-(2\mu-1)(1-\eptwo)\tau_0}}{(2\mu-1)(1-\eptwo)}\leq \frac{1}{4}B^2,\ \text{when }2\mu-1-\epone= 0,
    \eeqa
we have 
\be\label{H1 II}
    II^1\leq \frac{1}{4}B^2e^{-(2\mu-1)\eptwo \tau}.
\ee
For $I^1$, by the initial condition $\norm{(\espz(\tau_0,\cdot), \espw(\tau_0,\cdot))}_{H^\ndmu}\leq \sigma_0$, we have 
    \begin{align}\label{H1 I}
        I^1&\le e^{-(2\mu-1)\eptwo\tau}  e^{-(2\mu-1)(1-\eptwo)\tau}e^{(2\mu-1)\tau_0}\sigma_0^2\leq \frac{1}{4}B^2e^{-(2\mu-1)\eptwo \tau}
    \end{align}
    if we choose $\sigma_0$ be sufficiently small such that
    \be \label{delta1}
        e^{-(2\mu-1)(1-\eptwo)\tau_0}e^{(2\mu-1)\tau_0}\sigma_0^2\leq \frac{1}{4}B^2.
    \ee
    For the last piece $IV^1$, we split it into two cases: when $\frac{(-4\mu+1)(1-\mu)}{2}+(\mu-\frac12)(2-\eptwo)\neq 0$,
    \begin{align*}
        IV^1 &= e^{-(2\mu-1)\eptwo\tau} \frac{ CB y_0^{\frac{-4\mu +1}{2}}}{\frac{(-4\mu+1)(1-\mu)}{2}+(\mu-\frac12)(2-\eptwo)}\Big(e^{-\frac{(-4\mu+1)(1-\mu)}{2}\tau_0}e^{\big(\frac{(-4\mu+1)(1-\mu)}{2}+(\mu-\frac12)\eptwo\big)\tau} \\
        &-e^{(\mu-\frac12)\eptwo\tau_0}e^{-(\mu-\frac12)(2-2\eptwo)(\tau-\tau_0)}\Big)\\
        &\leq e^{-(2\mu-1)\eptwo\tau} \frac{2 e^{(\mu-\frac12)\eptwo\tau_0}CB y_0^{\frac{-4\mu +1}{2}}}{\left|\frac{(-4\mu+1)(1-\mu)}{2}+(\mu-\frac12)(2-\eptwo)\right|}.
     \end{align*}
    We have used $\frac{(-4\mu+1)(1-\mu)}{2}+(\mu-\frac12)\eptwo<0$ and $(\mu-\frac12)(2-2\eptwo)>0$ from \eqref{sequence a_1, a_2, a_0} in the last inequality. When $\frac{(-4\mu+1)(1-\mu)}{2}+(\mu-\frac12)(2-\eptwo)= 0$, we then have $(2\mu-1)(1-\eptwo)>0$ and compute
     \begin{align*}
         IV^1 &= CB y_0^{\frac{-4\mu +1}{2}}e^{\frac{-(-4\mu+1)(1-\mu)}{2}\tau_0} e^{-(2\mu-1)\tau}(\tau-\tau_0)\\
        &=e^{-(2\mu-1)\eptwo\tau} CB y_0^{\frac{-4\mu +1}{2}}e^{(\mu-\frac12)\eptwo\tau_0}e^{-(2\mu-1)(1-\eptwo)(\tau-\tau_0)} (\tau-\tau_0)\\
        &\leq e^{-(2\mu-1)\eptwo\tau}B \frac{Cy_0^{\frac{-4\mu +1}{2}}e^{(\mu-\frac12)\eptwo\tau_0-1}}{(2\mu-1)(1-\eptwo)} 
     \end{align*}
    where we have replaced $\frac{(-4\mu+1)(1-\mu)}{2}$ by $-(\mu-\frac12)(2-\eptwo)$ in the third equality and used $e^{-ax}x\leq \frac{1}{ae}$ with $a>0$ for any $x\geq 0$ in the last inequality. Now, choosing 
    \be
            y_0=e^{\vartheta \tau_0}, \ \text{where } \vartheta>0 \text{ to be determined}.
    \ee
    Then
    \be\label{y0}
    \frac{Cy_0^{\frac{-4\mu +1}{2}}e^{(\mu-\frac12)\eptwo\tau_0-1}}{(2\mu-1)(1-\eptwo)} = \frac{C}{(2\mu-1)(1-\eptwo)} e^{\big(\frac{-4\mu +1}{2}\vartheta+(\mu-\frac12)\eptwo\big)\tau_0-1}.
    \ee
    Once we have
    \be\label{vartheta}
        \frac{-4\mu +1}{2}\vartheta+(\mu-\frac12)\eptwo<0
    \ee
    together with $\tau_0$ sufficiently large,
    we establish
    \be\label{H1 IV}
        IV^1\leq \frac{1}{4}B^2 e^{-(2\mu-1)\eptwo\tau}.
    \ee
    We conclude that, by selecting $\tau_0$, $\sigma_0$, $y_0$, and $\vartheta$ so as to satisfy \eqref{B}, \eqref{tau0}, \eqref{delta1}, \eqref{y0}, \eqref{vartheta} the estimates \eqref{H1 III}, \eqref{H1 II}, \eqref{H1 I} and \eqref{H1 IV} are derived, thereby establishing \eqref{pay z w}.

\

\noindent\underline{Proof of \eqref{payn z w}}: From Lemma \ref{ode Hn}, 
	\beqa
		 \pa_\tau (e^{\al^2\tau}F_{\ndmu}(\tau))
	&\le   C { y_0^{\frac32-2\mu-\ndmu}}e ^{-(1-\mu)(\ndmu +2\mu -\frac{3}{2}) (\tau-\tau_0)}  e^{\al^2\tau}e^{-\epone \tau} +Ce^{\al^2\tau}e^{-\epzero \tau}e^{-\epone \tau}\\
     &+ CB^{1-\frac{1}{2(\ndmu-1)}}e^{\al^2\tau}e^{-(\mu-\frac12)(1-\frac{1}{2(\ndmu-1)})\eptwo\tau} e^{-(2+\frac{1}{2(\ndmu-1)})\epone \tau}
	\eeqa
    where
    \beqa
        F_{\ndmu}(\tau):=&\norm{\espz(\tau,\cdot)}_{\dot H^{\ndmu}}^2+b \norm{\espw(\tau,\cdot)}_{\dot H^{\ndmu}}^2 +  C(\norm{\frac{\espz}{y^{\ndmu}}(\tau,\cdot)}_{L^2}^2 + b\norm{\frac{\espw}{y^{\ndmu}}(\tau,\cdot)}_{L^2}^2)\\
        \al^2:=&(1-\mu)\ndmu - \frac32 + \frac{\mu}{2} .
    \eeqa
    We integrate it over the interval $[\tau_0,\tau]$ and obtain
    \begin{align*}
        F_{\ndmu}(\tau) & = e^{-\al^2 (\tau-\tau_0)}  F_{\ndmu}(\tau_0) + C { y_0^{\frac32-2\mu-\ndmu}}e^{-\al^2 \tau}e ^{(1-\mu)(\ndmu +2\mu -\frac32)\tau_0}\int_{\tau_0}^\tau e ^{\big(-(1-\mu)(\ndmu +2\mu -\frac32)+\al^2-\epone \big)\tau'}\ d\tau' \\
        &+ Ce^{-\al^2\tau}\int_{\tau_0}^\tau e^{(\al^2-\epzero-\epone)\tau'} \ d\tau' + CB^{1-\frac{1}{2(\ndmu-1)}}e^{-\al^2\tau}\int_{\tau_0}^\tau e^{-(\mu-\frac12)(1-\frac{1}{2(\ndmu-1)})\eptwo\tau'} e^{-(2+\frac{1}{2(\ndmu-1)})\epone \tau'} e^{\al^2 \tau'}\ d\tau'\\
        &=: I^2+II^2+III^2+IV^2.
    \end{align*}
    Following the same procedure as in the proof of \eqref{pay z w}, we determine the parameters in the order $\tau_0$, $\sigma_0$, and $y_0$. We begin with the estimate of $IV^2$. Using \eqref{sequence a_1, a_2, a_0}, we compute
    \[
    \kappa^2:=-(\mu-\frac12)(1-\frac{1}{2(\ndmu-1)})\eptwo -(2+\frac{1}{2(\ndmu-1)})\epone+\al^2 >0.
    \]
    Consequently,
    \beqa\label{IV2}
        IV^4 \leq e^{-2\epone\tau}CB^{1-\frac{1}{2(\ndmu-1)}}(\kappa^2)^{-1}e^{-(\mu-\frac12)(1-\frac{1}{2(\ndmu-1)})\eptwo\tau} e^{-\frac{1}{2(\ndmu-1)}\epone \tau}
    \eeqa
    Next, we consider $III^2$. Since $\al^2-\epzero-\epone>0$ from \eqref{sequence a_1, a_2, a_0}, it follows that
\beqa\label{III2}
    III^2 \leq C e^{-(\epzero+\epone)\tau}\leq Ce^{-2\epone\tau} e^{-(\epzero-\epone)\tau}.
\eeqa
From \eqref{IV2} and \eqref{III2}, by choosing $\tau_0$ large enough such that
\be\label{tau01}
   CB^{1-\frac{1}{2(\ndmu-1)}}(\kappa^2)^{-1}e^{-(\mu-\frac12)(1-\frac{1}{2(\ndmu-1)})\eptwo\tau_0} e^{-\frac{1}{2(\ndmu-1)}\epone \tau_0}+ Ce^{-(\epzero-\epone)\tau_0} \leq \frac12
\ee
we establish
\be
    III^2+IV^2 \leq \frac12e^{-2\epone\tau}.
\ee
    We now turn to $II^2$. Note that $ -(1-\mu)(\ndmu +2\mu -\frac32)+\al^2 =\mu(2\mu-3) <0$
    for $\mu\in[\frac12,1)$.   
Hence,
\beqa\label{II2}
    II^2 &\leq \frac{C{ y_0^{\frac32-2\mu-\ndmu}}}{-\mu(2\mu-3)+\epone}e^{-\al^2 \tau} e ^{\al^2\tau_0}e^{-\epone \tau_0} = e^{-2\epone\tau} \frac{C{ y_0^{\frac32-2\mu-\ndmu}}e^{\epone\tau_0}}{-\mu(2\mu-3)+\epone}e^{-(\al^2-2\epone)(\tau-\tau_0)} \\
    &\leq e^{-2\epone\tau}\frac{Ce^{\big({(\frac32-2\mu-\ndmu)}\vartheta+\epone\big)\tau_0}}{\mu(2\mu-3)+\epone}
\eeqa
where we used \eqref{y0} to replace $y_0$ by $e^{\vartheta\tau_0}$. {Then, by selecting $\vartheta$ such that
\be\label{vartheta 1}
    (\frac32-2\mu-\ndmu)\vartheta+\epone <0
\ee}
we deduce
\be
    II^2 \leq \frac14  e^{-2\epone\tau} 
\ee
for $\tau_0$ large enough.
Finally, for the term $I^2$, we have
    \be
        e^{-\al^2 (\tau-\tau_0)}  F_{\ndmu}(\tau_0) \leq e^{-2\epone\tau}e^{-(\al^2-2\epone)\tau}C b e^{\al^2 \tau_0}  \sigma_0^2 \leq \frac{1}{4}e^{-2\epone \tau} 
    \ee
    by selecting $\sigma_0$ sufficiently small such that
    \be\label{delta11}
        Cb e^{\al^2 \tau_0}\sigma_0^2 \leq \frac{1}{4}.
    \ee
    By choosing $\tau_0$ sufficiently large as in \eqref{tau0} and \eqref{tau01}, $\sigma_0$ sufficiently small 
    as in \eqref{delta1} and \eqref{delta11}, and $\vartheta$ 
    as in \eqref{vartheta} and \eqref{vartheta 1}, we combine the above estimates and  conclude that
    \be
        F_{\ndmu}(\tau)\leq  e^{-2\epone \tau}
    \ee
    which completes the proof.

\

\noindent\underline{Proof of \eqref{L2}}: Now, we close the bootstrap assumption of the $L^2$ norm. From Lemma \ref{lemma L2},
 \begin{align*}
    & \pa_\tau\Big(e^{-(3-2\mu)\tau}(\norm{\espz}_{L^2}^2+\norm{\espw}_{L^2}^2)\Big) \\ \leq&  CB^2 e^{-\frac{\mu\epone}{2\ndmu-\mu}\tau} +CBe^{-(\frac{3}{2}-2\mu)\tau}(G[\mu])^{\frac12} +CB^{3-\frac{1}{2(\ndmu-1)}} e^{-(1-\frac{1}{2(\ndmu-1)})(\mu-\frac12)\eptwo \tau}e^{-\frac{1}{2(\ndmu-1)}\epone\tau} .
 \end{align*}
 We integrate it in the interval $[\tau_0,\tau]$ and obtain
 \begin{align*}
    \norm{\espz(\tau,\cdot)}_{L^2}^2+\norm{\espw(\tau,\cdot)}_{L^2}^2 &\leq e^{(3-2\mu)(\tau-\tau_0)}\Big(\norm{\espz(\tau_0,\cdot)}_{L^2}^2+\norm{\espw(\tau_0,\cdot)}_{L^2}^2\Big)\\
   &+ B^2e^{(3-2\mu)\tau}C(\frac{\mu\epone}{2\ndmu-\mu})^{-1}e^{-\frac{\mu\epone}{2\ndmu-\mu}\tau_0}\\
    &+B^2e^{(3-2\mu)\tau}\frac{CB^{1-\frac{1}{2(\ndmu-1)}}e^{-(1-\frac{1}{2(\ndmu-1)})(\mu-\frac12)\eptwo \tau_0}e^{-\frac{1}{2(\ndmu-1)}\epone\tau_0}}{(1-\frac{1}{2(\ndmu-1)})(\mu-\frac12)\eptwo+\frac{1}{2(\ndmu-1)}\epone}   \\
    &+CBe^{(3-2\mu)\tau}\int_{\tau_0}^\tau e^{-(\frac32-\mu)\tau'}(G[\mu])^{\frac12}\ dy
 \end{align*}
 where we have used $\frac{\mu\epone}{2\ndmu-\mu}>0$ and $(1-\frac{1}{2(\ndmu-1)})(\mu-\frac12)\eptwo+\frac{1}{2(\ndmu-1)}\epone>0$ to obtain the above inequality.
 The summation of the first four terms can be bounded by $\frac{1}{4}B^2e^{(3-2\mu)(\tau-\tau_0)}$ by choosing $\tau_0$ and $\sigma_0$ that satisfy
 \beqa
       e^{-(3-2\mu)\tau_0}\sigma_0^2&\leq \frac14 B^2,\\
      C(\frac{\mu\epone}{2\ndmu-\mu})^{-1}e^{-\frac{\mu\epone}{2\ndmu-\mu}\tau_0}, \ \frac{CB^{1-\frac{1}{2(\ndmu-1)}}e^{-(1-\frac{1}{2(\ndmu-1)})(\mu-\frac12)\eptwo \tau_0}e^{-\frac{1}{2(\ndmu-1)}\epone\tau_0}}{(1-\frac{1}{2(\ndmu-1)})(\mu-\frac12)\eptwo+\frac{1}{2(\ndmu-1)}\epone}  &\leq \frac{1}{4}.
 \eeqa
 It remains to show 
 \be\label{V}
    V:=C\int_{\tau_0}^\tau e^{-(\frac32-\mu)\tau'}(G[\mu])^{\frac12}\ dy\leq \frac{1}{4}B.
 \ee
 By recalling $G[\mu]$ as defined in \eqref{Fmu}, we bound it case by case. When $\mu\in[\frac12,\frac34)$, we compute
 \begin{align*}
   Cy_0^{\frac32-2\mu}\int_{\tau_0}^\tau e^{-(\frac32-\mu)\tau'} e^{(\frac32-2\mu)(\tau'-\tau_0)}\ dy &\leq \frac{y_0^{\frac32-2\mu}}{\mu}e^{-(\frac{3}{2}-\mu)\tau_0} = \frac1\mu e^{\Big((\frac{3}{2}-2\mu)\vartheta -\frac32+\mu\Big)\tau_0} = \frac{1}{\mu} e^{-\mu\tau_0} \leq \frac{1}{4}B
 \end{align*}
 for $\tau_0$ sufficiently large with 
$\vartheta=1$ such that 
 \be\label{y_0}
 y_0=e^{\tau_0}.
 \ee
 It is a routine matter to verify that \eqref{vartheta} and \eqref{vartheta 1} hold for $\vartheta=1$. We next show that with $y_0=e^{\tau_0}$, \eqref{V} holds for $\mu=\frac34$ and $\mu\in(\frac34,1)$. The case $\mu\in(\frac34,1)$ is treated analogously to the first case by noting that $3-4\mu<0$.  When $\mu=\frac34$, for $\tau_0$ sufficiently large, we have
 \be
        V = C\int_{\tau_0}^\tau e^{-\frac34\tau'}\sqrt{\ln 2+\tau'}\ dy \leq   C\int_{\tau_0}^\tau e^{-\frac12\tau}\ dy \leq \frac14 B.
 \ee
This completes the proof.
 
\section{Proof of the main theorems}  

We will first prove the Theorem \ref{thm 4.1}.
\begin{proof}[Proof of Theorem \ref{thm 4.1}]
The existence of a well-posed solution in $H^\ndmu$ follows directly from Proposition \ref{bootstrap lemma} via a continuity argument. The decay estimate \eqref{decay fact} is a consequence of Proposition \ref{bootstrap lemma}, Proposition \ref{prop payi mz mw}, and Lemma 
\ref{Linfty decay}. Indeed, by using Lemma 
\ref{Linfty decay}, we have
\begin{align*}
    &\norm{\big(\pa_y\tilde z(\tau,\cdot)-\pa_y \mathring z(\tau,\cdot), \pa_y\tilde w(\tau,\cdot)\big)}_{L^\infty} \\&\leq \norm{(\pa_y\mz,\pa_y\mw)}_{L^\infty} + \norm{(\pa_y\espz,\pa_y\espw)}_{L^\infty} \\
    &\leq \norm{(\pa_y\mz,\pa_y\mw)}_{L^\infty} + \norm{\big(\pa_y\espz(\tau,\cdot), \pa_y\espw(\tau,\cdot)\big)}_{L^2}^{1-\frac{1}{2(\ndmu-1)}}\norm{\big(\pa_y^\ndmu\espz(\tau,\cdot), \pa_y^\ndmu\espw(\tau,\cdot)\big)}_{L^2}^{\frac{1}{2(\ndmu-1)}}\\
    & \lesssim e^{-\epzero\tau} + e^{-(\mu-\frac12)(1-\frac{1}{2(\ndmu-1)})\eptwo\tau}e^{-\frac{\epone}{2(\ndmu-1)}\tau}.
\end{align*}
where we have applied Proposition \ref{bootstrap lemma} and Proposition \ref{prop payi mz mw} to obtain the last inequality.
The statement on the finite codimensional stability follows directly from \eqref{boundary condition} of $I^{\lfloor \beta \rfloor}$ together with Proposition \ref{trapping}.
\end{proof}

\begin{proof}[Proof of Theorem \ref{thm 1.1}]
    Direct computation shows
    \begin{align*}
        \norm{z}_{L^2_x}^2 &= \int_0^\infty z^2(t,x) \ dx = \int_0^\infty (\frac{1}{\mu})^{2\delta} e^{(2\mu-2)\tau}\hat z^2(\tau,y)(\frac{1}{\mu})^{\delta} e^{- \tau} \ dy = (\frac{1}{\mu})^{3\delta} e^{(2\mu- 3)\tau}\int_0^\infty \hat z^2(\tau,y) \ dy\\
       & \leq Ce^{(2\mu- 3)\tau} \big(\norm{\mathring z}_{L^2}^2+\norm{\mz}_{L^2}^2+\norm{\espz}_{L^2}^2 \big).
    \end{align*}
 By Proposition \ref{prop: mathring z}, Proposition \ref{prop payi mz mw}, together with the estimate \eqref{L2}, we obtain the $L^2$-boundedness of $z$. The same argument yields the boundedness of $\norm{w}_{L^2_x}$. 
 Hence, we deduce the $L^2$ boundedness of $u$ and $c$. 
 
 The asymptotic behavior \eqref{1.7} is a direct consequence of \eqref{decay fact} and Proposition \ref{prop: mathring z}.
 
    The positivity of $\rho$ can be shown provided that the initial data has strictly positive density in the interior of the domain. Introducing the Lagrangian flow map defined by
    \[\frac{d}{dt}X(t,x) = u(t,X(t,x)),\quad X(t_0,x)=x,\]and integrating the first equation of \eqref{Euler} along the flow map, we have 
    \begin{align*}
        \rho(t,X(t,x)) = \rho(t_0,X(t_0,x)) e^{-\int_{t_0}^t u_x(s,X(s,x))ds}.
    \end{align*}
    Since the solution is smooth for $t_0<t<0$ 
    and the gradient blow-up occurs only at $(t,x)=(0,0)$, the density remains strictly positive in the interior of the domain. 
\end{proof}

\section*{Acknowledgments}
JJ and JL were supported in part by the NSF grant DMS-2306910.

\appendix

\section{Interpolation inequalities}
\begin{lemma}\label{hardy 1}
	Let $n\in \mathbb N_{\geq 3}$ and $\mathtt a>0$. Consider the function $f\in H^n([0,\infty))$ satisfies $f(0)=\pa_yf(0)=\dots=\pa_y^{n-1}f(0)=0$. Then, for any $i\in\{0,\dots,n-1\}$, we have
	\be\label{5.1}
		\norm{y^{-(n-i)} \pa_y^i f}_{L^2([0, \infty))} \leq  A\norm{\pa_y^n f}_{L^2([0, \infty))}
	\ee
	where $A$ is a positive constant.
\end{lemma}
\begin{proof}
    Direct computation shows that for $ f\in H^n([0,\infty))$ with $f(0)=\pa_yf(0)=\dots=\pa_y^{n-1}f(0)=0$,
    \begin{align*}
        \int_0^\infty \frac{(\pa_y^i f)^2}{y^{2(n-i)}} dy =- \int_0^\infty \frac{(\pa_y^i f)^2}{y^{2(n-i-1)}} (\frac{1}{y})' dy = \int_0^\infty \frac{2\pa_y^i f\pa_y^{i+1} f}{y^{2(n-i-1)+1}}  dy -2(n-i-1)  \int_0^\infty \frac{(\pa_y^i f)^2}{y^{2(n-i)}} dy.
    \end{align*}
    Hence,
    \begin{align*}
        (n-i-\frac12)  \int_0^\infty \frac{(\pa_y^i f)^2}{y^{2(n-i)}} dy = \int_0^\infty \frac{\pa_y^i f\pa_y^{i+1} f}{y^{2(n-i-1)+1}}  dy \ \ \Rightarrow \ \  \norm{\frac{\pa_y^i f}{y^{n-i}}}_{L^2([0,\infty))} \leq A \norm{\frac{\pa_y^{i+1} f}{y^{n-i-1}}}_{L^2([0,\infty))}. 
    \end{align*}  
    Iterating this estimate for $i,i+1,\dots,n-1$ gives \eqref{5.1}.
\end{proof}

\begin{lemma}\label{G-N with weight}
	Let $n\in \mathbb N_{\geq 2}$.  
    Consider the function $f\in H^n([0,\infty))$ satisfies $f(0)=\pa_yf(0)=\dots=\pa_y^{n-1}f(0)=0$. Then, for any $i\in\{1,\dots,n-1\}$ and $y_*>0$, we have
	\be
		\norm{y^{i-n} \pa_y^i f}_{L^2([y_*, \infty))} \leq C_* \norm{y^{-n}f}_{L^2}^{1-\theta} \norm{\pa_y^n f}_{L^2}^{\theta} 
	\ee
	where $\theta=\frac{i}{n}$, and $C_*:=C_*(y_*)$ is a positive constant. 
\end{lemma}
\begin{proof}
	We first define $f_\lambda(y) = f(\lambda y)$ and compute 
	\[
		\int_{a}^{b} (\pa_y^i f(y))^2 \ dy = \lambda^{1-2i}\int_{ a/\lambda}^{b/\lambda} (\pa_y^i f_\lambda(y))^2 \ dy, \ \ \int_{a}^{b} y^{-2n} (f(y))^2 \ dy = \lambda^{1-2n}\int_{ a/\lambda}^{b/\lambda} y^{-2n} (f_\lambda(y))^2 \ dy.
	\] 
	 We then set $\lambda_j =  2^j y_*$, and define $f_j:=f_{\lambda_j}$ for $y\in[2^jy_*, 2^{j+1}y_*]$. Then, 
	 \begin{align*}
	 	\int_{y_*}^\infty y^{-2(n-i)} (\pa_y^i f)^2 \ dy 
        &\leq \sum_{j=0}^\infty \lambda_j^{1-2n}  \int_{1 }^{2}    (\pa_y^i f_j)^2 \ dy \\
		&\leq  \sum_{j=0}^\infty  \lambda_j^{1-2n}C_1 \Big( \norm{\pa_y^n f_j}_{L^2([1,2])}^{\frac{2i}{n}}\norm{f_j}_{L^2([1,2])}^{2-\frac{2i}{n}}+\norm{f_j}_{L^2([1,2])} \Big)\\
        &\leq \sum_{j=0}^\infty C_1  \Big(  \norm{\pa_y^n f}_{L^2([\lambda_j,2\lambda_j])}^{\frac{2i}{n}}\norm{ y^{-n}f}_{L^2([\lambda_j,2\lambda_j])}^{2-\frac{2i}{n}}+ \norm{ y^{-n}f}_{L^2([\lambda_j,2\lambda_j])}\Big) 
	 \end{align*}
     where we have used the Gagliardo–Nirenberg inequality in the second inequality.
     We then use the H\"{o}lder inequality on the first summation and together with Lemma \ref{hardy 1}, we deduce the result.   
\end{proof}

\begin{lemma}\label{Linfty decay}
	Let $f \in\dot H^1([0,\infty)) \cap \dot H^n([0,\infty))$. Suppose $\pa_y ^j f(0) =0$ for some $j\in\{1,\dots, n-1\}$. Then, 
	\be
		\norm{\pa_y^j f}_{L^\infty} \lesssim \norm{\pa_y f}_{L^2}^{1-\frac{2j-1}{2(n-1)}} \norm{\pa_y^n f}_{L^2}^{\frac{2j-1}{2(n-1)}}. 
	\ee
\end{lemma}
\begin{proof}
	Direct computation shows
    \begin{align*}
        (\pa_y^j f)^2 = \int_0^y 2\pa_y^j f \pa_y^{j+1} f d\tilde y.
    \end{align*}
    Hence, by using Holder inequality, we obtain
    \[
        \norm{\pa_y^j f}_{L^\infty} \lesssim \norm{\pa_y^j f}_{L^2}^{\frac{1}{2}}\norm{\pa_y^{j+1}f}_{L^2}^{\frac{1}{2}}.
    \]
    By the Gagliardo–Nirenberg interpolation inequality
    \[
        \norm{\pa_y^jf}_{L^2} \lesssim \norm{\pa_y f}_{L^2}^{1-\frac{j-1}{n-1}} \norm{\pa_y^n f}_{L^2}^{\frac{j-1}{n-1}}, \ \ \ 
        \norm{\pa_y^{j+1}f}_{L^2} \lesssim \norm{\pa_y f}_{L^2}^{1-\frac{j}{n-1}} \norm{\pa_y^n f}_{L^2}^{\frac{j}{n-1}}, 
    \]
    we obtain the desired inequality. 
\end{proof}

\section{Trivial mode instability}\label{eigenvalue subsection}
We discuss trivial mode instability of a simple wave state $(\bar z(y), 0)$ associated with symmetries of the system. 
Since $w$ is secondary, we will let $w\equiv 0$ and consider 
the equation for $z$: 
\be\label{3.1}
    \pa_t z + \frac{\ga+1}{4}z\pa_x z =0. 
\ee
Our discussion closely follows \cite{CGN18} and \cite{EF00} on self similarity in shocks for Burgers equation. 
If $z(t,x)$ is a solution to \eqref{3.1}, by time and space translation, Galilean transformation, space and time scaling invariances, 
\be
    \frac{\al}{\nu} z(\frac{t-t_0}{\nu}, \frac{x-x_0-vt}{\al})+\frac{4}{\ga+1} v
\ee
are also solutions to \eqref{3.1}. In particular, $z(t,x) = -\delta(T-t)^{\delta-1} \bar z (\frac{x-x_0-vt}{(T-t)^\delta}) + \frac{4}{\ga+1} v$ is also a solution. 
We can then compute the spatial counterpart of the associated infinitesimal generators of the above transformations:
\be
    \Lambda_{v}= \delta \pa_y  +\frac{4}{\ga+1},\ \ \Lambda_{x_0} = \pa_y , \ \ \Lambda_{T} = (1-\delta)\text{Id} + \delta y \pa_y,\ \ \Lambda_{\al} = \text{Id} - y\pa_y
\ee
where $\frac{4}{\ga+1}$ is the constant function and Id is identity. 

With $w\equiv 0$, the linearized operator (cf.\eqref{4.5}) is given by 
\be
    \mathcal L = (\mu-1+\frac{\ga+1}{4}\pa_y \bar z)+(y+\frac{\ga+1}{4}\bar z)\pa_y
\ee
where we recall $\bar z$ satisfies \eqref{2.8}: 
\[
       (y+\frac{\ga+1}{4}\bar z)\pa_y \bar z- (1-\mu)\bar z =0.
\]
In the following we discuss the unstable part of the spectrum of $\mathcal L$.  
\begin{proposition}\label{trivial mode}
    Let $\ga\in(1,3)$ and $\mu\in[\frac12,1)$.  
    The unstable spectrum consisting of non-positive eigenvalues of $\mathcal L$ 
    on the H\"{o}lder continuous functions $C^{[\frac{1}{1-\mu}], \frac{1}{1-\mu} - [\frac{1}{1-\mu}]}
    $ is given by  
    \be\notag 
            \mathfrak R(\mathcal L) = 
            \begin{cases}
            \{j (1-\mu)-1 \ | \ j\in \mathbb{N}_{\ge 0}\cap [0,\frac{1}{1-\mu} ]  \}, & \text{if } \mu \in \ssmu  \\ 
            \{j (1-\mu)-1 \ | \ j\in \mathbb{N} \cap [1, \lfloor\frac{1}{1-\mu}\rfloor ] \} \cup \{0\}, & \text{if } \mu \notin \ssmu 
            \end{cases}
    \ee
    The eigenfunctions related to symmetries are 
    \be\label{5.17}
         \mathcal L(\Lambda_{x_0} \bar z) = -(\Lambda_{x_0} \bar z),\ \ \mathcal L(\Lambda_{T} \bar z) = -\mu (\Lambda_{T} \bar z),\ \  \mathcal L(\Lambda_{v} \bar z) = - (1-\mu)\Lambda_{v} \bar z,\ \ \mathcal L(\Lambda_\al \bar z)=0. 
    \ee 
    Moreover, these eigenfunctions are given by
    \begin{align}\label{eigenfunction}
        \mathcal L(\phi_a) = a \phi_a,\qquad \phi_a = \frac{(-\bar z)^{\frac{1+a}{1-\mu}}}{1-\mu+(-\bar z )^{\frac{\mu}{1-\mu}}}\text{ with } a \in \mathfrak R(\mathcal L).
    \end{align}
\end{proposition}   
\begin{proof}
We first recall \eqref{B formula}. By rewriting $y =K \frac{(\frac{2\mu}{1+\ga}-\bar U)^{\frac1\mu -1}}{ \bar U ^\frac1\mu}$ 
and taking $K=\frac12$, we obtain
\be
    \frac{4\mu}{1+\ga} y = -\bar z +(-\bar z)^{\frac{1}{1-\mu}}.
\ee
Consider the eigenvalue problem $ \mathcal L(\phi_a) = a \phi_a$ and we may derive 
\begin{align*}
    \frac{d\phi_a}{d\bar z} = \phi_a \frac{-(a+1)-(\frac{a}{1-\mu}+1)(-\bar z)^{\frac{\mu}{1-\mu}}}{(\mu-1)\bar z +(-\bar z)^\frac{1}{1-\mu}}.
\end{align*}
It then follows that
\be
    \phi_a \in \text{Span}\{\frac{(-\bar z)^{\frac{1+a}{1-\mu}}}{1-\mu+(-\bar z )^{\frac{\mu}{1-\mu}}}\}.
\ee
Let $a\le 0$ be given. In order for $\phi_a$ to be $\frac{1}{1-\mu}$-H\"{o}lder continuous, we have that $1+a = j (1-\mu)$ for $0\le j \le \frac{1}{1-\mu}$ and $j\in \mathbb N_{\ge 0}$. If  $\frac{1}{1-\mu}$ is an integer, the eigenfunctions are smooth and the unstable spectrum is given by $\{  -1 + j (1-\mu) : 0 \le j \le \frac{1}{1-\mu}  \}$.  
If $\frac{1}{1-\mu}$ is not an integer, $\phi_{-1}$ is only $\frac{\mu}{1-\mu}$-H\"{o}lder continuous. the unstable part of the spectrum is given by $\{  -1 + j (1-\mu) : 1 \le j \le [\frac{1}{1-\mu}] \}\cup \{0\}$. 
\end{proof}

\begin{remark}
     The freedom $K$ (see \eqref{lemma 3.4}) corresponds to the scaling invariance.  
     Notice that  we see the eigenvalues $a=-1$ and $a=\mu-1$ only when $\mu\in\ssmu$, namely when $\frac{1}{1-\mu}$ is an integer, since $\Lambda_{x_0} \bar z$ and $\Lambda_{v} \bar z$ in \eqref{5.17} 
belong to 
$C^{[\frac{1}{1-\mu}], \frac{1}{1-\mu} - [\frac{1}{1-\mu}]}
    $ only if $\mu\in\ssmu$ (cf. \eqref{eigenfunction}). 
     Moreover, when $\mu=\frac{1}{2}$, direct computation shows $\Lambda_v\bar z = -\frac{\ga+1}{4} \Lambda_T \bar z$.
     Therefore, when $\mu=\frac{1}{2}$, the eigenfunctions associated with time translation and Galilean Transformation coincide. 
\end{remark}

\section{On the regularity index $\ndmu$}\label{Appendix C}

We next consider the weighted operator 
$\mathcal L^m$ considered in Lemma \ref{weighted L^2 estimates}, 
\be
    \mathcal L^m \phi= (\mu-1)\phi +\frac{\ga+1}{4}\pa_y \bar z \phi +(y+\frac{\ga+1}{4}\bar z ) \pa_y \phi + m(1+\frac{\ga+1}{4}\frac{\bar z}{y})\phi.
\ee
In other words, $\mathcal L^m \phi = \frac{1}{y^m} \mathcal L (y^m \phi) $. Here $\phi$ can be thought of $\frac{Z}{y^m}$ for $m=\ndmu$ in the context of Lemma  \ref{weighted L^2 estimates}. It is clear that positive $m$ dampens the dynamics and shifts the spectrum of $\mathcal L$. 
In fact, if $\mathcal L^m  \phi_a = a\phi_a$, 
we obtain
\begin{align*}
    \frac{d\phi_a}{d \bar z} =  \phi_a \Big(\frac{-(a+1)-(\frac{a}{1-\mu}+1)(-\bar z)^{\frac{\mu}{1-\mu}}}{(\mu-1)\bar z +(-\bar z)^\frac{1}{1-\mu}}-m\frac{ -1-\frac{1}{1-\mu}(-\bar z)^{\frac{\mu}{1-\mu}} }{-\bar z +(-\bar z)^{\frac{1}{1-\mu}}}\Big).
\end{align*}
Hence
\beqa
    \mathcal L^m \phi_a  = a\phi_a,\quad \text{ with } \quad \phi_a = \frac{y^{-m}(-\bar z)^{\frac{1+a}{1-\mu}}}{1-\mu+(-\bar z )^{\frac{\mu}{1-\mu}}}.
\eeqa
When $y\ll 1$, we have
\be
    \phi_a \sim y^{\frac{1+a}{1-\mu}-m} + o(y^{\frac{1+a}{1-\mu}-m}).
\ee
In order to have $\phi_a\in L^2([0,1])$, the exponent should satisfy $\frac{1+a}{1-\mu}-m >-\frac12$, which is equivalent to $a> m(1-\mu)-\frac{3}{2}+\frac{\mu}{2}$. If $m=\lceil \beta+\frac12 \rceil$, then 
\be\label{a_lb}
   a> m(1-\mu) - (\beta+\frac12)(1-\mu) = ( \lceil \beta+\frac12 \rceil-(\beta+\frac12))(1-\mu).  
\ee
If $\beta+\frac12 \notin \mathbb N $, the values of $a$ satisfying \eqref{a_lb}  have a strict lower bound and thus $\mathcal L^m$ for $m={\lceil \beta+\frac12 \rceil}$ is strictly positive in $L^2$. On the other hand,   
when $\beta+\frac12 \in \mathbb N$, $ \lceil \beta+\frac12 \rceil-(\beta+\frac12)=0$. The condition \eqref{a_lb} reduces to $a>0$, and  
the admissible range of $a$ collapses to the threshold value $a=0$ without spectral gap. 
Consequently, the strict coercivity 
of $\mathcal L^m$ for $m={\lceil \beta+\frac12 \rceil}$ fails at those values of $\mu$, which act as critical parameters with respect to the strength of weight. Similarly, the same criticality emerges in the linear operator $\mathcal L_m = \partial_y^m \mathcal L$ for $m={\lceil \beta+\frac12 \rceil}$ when $\beta+\frac12 \in \mathbb N$ in $L^2$.  
To attain a strict coercivity at these critical values in the standard $L^2$-based spaces, we work with the regularity index $\ndmu$ as given in \eqref{def ndmu} rather than ${\lceil \beta+\frac12 \rceil}$. 

\begin{remark}
Requiring one more degree of regularity at the critical values ($\mu \in \mathcal H$) induces non-integrable singular behavior  (in the sense of $L^2$) of interaction terms between low modulation variables and the profile itself. This motivates us to introduce one more constraint $w_2=0$, which  removes such logarithmical non-integrable singular behavior. 
\end{remark}

\section{Time translation symmetry and modulation}\label{sec: TT} 

In this section, we discuss time translation symmetry and constraints in relation to the modulation coefficients $z_1$ and $w_1$. In our main theorem, we have taken $z_1=w_1\equiv 0$, which  
fix the blowup time of our solutions as $t=0$. To address more general cases, we may consider   
the modulation parameter,  denoted by $T(t)$ and define the self-similar variables (cf. \eqref{tau y}) by 
\be 
\tau = - \frac{\log \mu (T(t)- t)}{\mu}	\ \text{ and } \  y = \frac{x}{(T(t)-t)^\delta} 
\ee
where $T(t):\R\to \R$. The gradient blowup time is then defined by the $T_*$ such that $T(T_*)=T_*$. 
It is straightforward to compute
\[
	\pa_t\tau = e^{\mu \tau}( e^{\mu \tau}T_\tau +1)^{-1},\ \pa_t y =    e^{\mu \tau}( e^{\mu \tau}T_\tau +1)^{-1}y,\ \pa_x\tau=0,\ \pa_x y=(\frac{1}{\mu})^{-\delta} e^{ \tau}.
\]
By using the above identities, we see that the self-similar unknowns $(\hat z, \hat w)$ (cf. \eqref{selfsimilar ansatz} where $-t$ is replaced by $T(t)-t$),  
 satisfy 
\be\begin{cases}
	\partial_\tau \hat z + \Big(y+(e^{\mu \tau}T_\tau+1)\hat \la_1\Big) \partial_y \hat z + (\mu-1)\hat z= 0,\\
	\partial_\tau \hat w + \Big(y+(e^{\mu \tau}T_\tau+1)\hat \la_2 \Big) \partial_y \hat w + (\mu-1)\hat w = 0,
	\end{cases}
\ee
where $\la_1$ and $\la_2$ are defined in \eqref{def la1 la2}. Note that  \eqref{ss EE zw} is recovered if $T$ does not vary in time and indeed for our main result we have normalized $T\equiv 0$: see \eqref{tau y}. 
Then, $(\tilde z, \tilde w)$ satisfy
\begin{align}
		&\partial_\tau \tilde z + \Big(y+\frac{\ga+1}{4}(e^{\mu \tau}T_\tau+1)\mathring z \Big) \partial_y \tilde z + \Big(\frac{\ga+1}{4}(e^{\mu \tau}T_\tau+1)\partial_y \mathring z +\mu-1\Big)\tilde z +\frac{3-\ga}{4}(e^{\mu \tau}T_\tau+1)\partial_y\mathring z \ \tilde w \notag \\ 
        &\quad \ \  + \frac{\ga+1}{4}e^{\mu \tau}T_\tau \mathring z\pa_y \mathring z+(e^{\mu \tau}T_\tau+1)(\frac{\ga+1}{4}\tilde z+\frac{3-\ga}{4}\tilde w)\partial_y \tilde z=\frac{\ga+1}{4} \bar z (1-\chi) \pa_y \mathring z , \label{PDE of tilde z w with T*}\\
		&\partial_\tau \tilde w + \Big(y+\frac{3-\ga}{4}(e^{\mu \tau}T_\tau+1)\mathring z \Big) \partial_y \tilde w + (\mu-1)\tilde w +(e^{\mu \tau}T_\tau+1)(\frac{3-\ga}{4}\tilde z+\frac{\ga+1}{4}\tilde w)\partial_y \tilde w= 0.\notag  
\end{align}
Our stationary vacuum boundary conditions lead to 
\be\label{zero BC 1}
    \hat z(\tau,0) = \hat w(\tau,0) = \tilde z(\tau, 0) =\bar z(0)=0,\ \text{ for all }\tau>\tau_0.
\ee
The system of $(z_0,w_0)$ is trivial due to the boundary condition \eqref{zero BC 1}. 
We then check the system for $(z_1,w_1)$
\begin{align}
		&\partial_\tau z_1 -\mu z_1  -\frac{(3-\ga)\mu}{\ga+1}w_1  +(\frac{\ga+1}{4}z_1 +\frac{3-\ga}{4}w_1 )z_1 \notag  \\ 
        &\ \ +\Big(-2\mu z_1 -\frac{(3-\ga)\mu}{\ga+1}w_1+ \frac{4\mu^2}{\ga+1}+(\frac{\ga+1}{4}z_1 +\frac{3-\ga}{4}w_1 )z_1\Big)e^{\mu \tau}T_\tau = 0, \label{ODE T*}\\
		&\partial_\tau w_1  +\frac{2(\ga-1)\mu}{\ga+1} w_1  +\Big(\frac{3-\ga}{4}z_1+\frac{\ga+1}{4}w_1 \Big)w_1+\Big(-\frac{(3-\ga)\mu}{\ga+1}+\frac{3-\ga}{4}z_1+\frac{\ga+1}{4}w_1\Big)e^{\mu \tau}T_\tau w_1 = 0. \notag 
\end{align}

Observe that the first $z_1$ equation of \eqref{ODE T*} has a non-degenerate coefficient in front of $T_\tau$. We then demand 
\be\label{z1=0}
z_1 (\tau)=0  \ \text{ for all } \ \tau\ge \tau_0. 
\ee
    With \eqref{z1=0}, the first equation in \eqref{ODE T*} gives rise to 
\be\label{ode of T'}
	e^{\mu \tau}T_\tau+1 
    = \frac{\mu}{\mu-\frac{3-\ga}{4}w_1}.
\ee
Conversely, observe that  if \eqref{ode of T'} holds and if $z_1(\tau_0)=0$, \eqref{z1=0} is dynamically preserved. We remark that the initial vanishing condition $z_1(\tau_0)=0$ is not a restriction. In fact due to the time translation symmetry, for given initial $z$-slope $\pa_x z (t_0,0)$, one can find $T(t_0)$ such that $\pa_x z (t_0,0) = \pa_x \bar z (t_0, 0)$ where $\bar z (t_0,x) = \delta (T(t_0)-t_0)^{\delta-1} \bar z (\frac{x}{ (T(t_0)-t_0)^{\delta} })$ (cf. \eqref{selfsimilar ansatz}), which will give rise to $z_1(\tau_0)=0$. 

From \eqref{ODE T*} and \eqref{ode of T'}, we see that
\begin{align}\label{ode w_1}
	\partial_\tau w_1  + \Big(1-\frac{3-\ga}{\ga+1}\frac{\mu}{\mu-\frac{3-\ga}{4}w_1}\Big) \mu w_1  +\frac{(\ga+1)\mu}{4\mu-(3-\ga)w_1}\big(w_1\big)^2&= 0.
\end{align}
We first show the following properties of $w_1$: 
\begin{lemma}\label{lemma w_1}
	Let $\ga\in(1,3)$ and $\mu\in[\frac12,1)$. Suppose that the initial data satisfy $w_1(\tau_0) = \epsilon_0$ with $\epsilon_0>0$ sufficiently small. Then the solution to \eqref{ode w_1} exists for all $\tau>\tau_0$, is strictly decreasing to zero, and satisfies the exponential decay estimate
	\be\label{estimate of w_1}
	0<w_1(\tau)< e^{\al^2 \tau_0}w_1(\tau_0)e^{-\al^2 \tau},\qquad \tau \geq \tau_0,
	\ee
for some constant $\alpha \neq 0$. Furthermore, if $w_1(\tau_0) = 0$, then $w_1 \equiv 0$ for all $\tau > \tau_0$.
\end{lemma}

\begin{proof}
	We will prove it by a continuity argument. Since the $w_1(0)=\epsilon_0$ is sufficiently small, $\mu-\frac{3-\ga}{4}w_1(0)>0$ initially. Consequently, $\pa_\tau w_1(\tau_0)<0$. By the Picard–Lindelöf theorem, there exists $\nu_0 > 0$ such that the solution to the ODE \eqref{ode w_1} with initial condition $w_1(\tau_0) = \epsilon_0$ exists on the interval $\tau \in [\tau_0,\tau_0+\nu_0)$. Moreover, $\nu_0$ can be chosen sufficiently small so that $w_1(\tau_0)-w_1(\tau)\ll 1$. It then follows that 
	\[
		1-\frac{3-\ga}{\ga+1}\frac{\mu}{\mu-\frac{3-\ga}{4}w_1}>0, \ w_1\in(0,\epsilon) \text{ for }\tau\in[\tau_0,\tau_0+\nu_0).
	\]
	In particular 
	\[
		\pa_\tau w_1<0\text{ for }\tau\in[\tau_0,\tau_0+\nu_0).
	\]
	By a standard continuity argument, the solution can be extended to all $\tau \in [\tau_0,\infty)$, and $w_1$ is strictly decreasing. Furthermore, note that $w_1 \equiv 0$ is itself a solution to the ODE \eqref{ode w_1}. By uniqueness of solutions guaranteed by the Picard–Lindelöf theorem, this establishes the second part of the statement. Now, we remain to show \eqref{estimate of w_1}. Since $w_1\in[0,\epsilon_0)$. Then, there exists an nonzero real number $\al$ such that
	\[
		\pa_\tau w_1 < -\al^2 w_1.
	\]
	It then follows 
	\[
		w_1(\tau)< w_1(\tau_0)e^{-\al^2 (\tau-\tau_0)}.
	\]
	This completes the proof.
\end{proof}
Now, back to \eqref{ode of T'}, we have the following property of the blowup time $T_*$. Assume $T(t_0)=0$. 
\begin{lemma}
	Let $z_1\equiv 0$ and $w_1(\tau_0)=\epsilon_0$ to be sufficiently small, we have
	\be
		T_*  < 2\epsilon_0 . 
	\ee
\end{lemma}
\begin{proof}
	By using \eqref{ode of T'}, we compute
	\[
		T_\tau = e^{-\mu\tau} \frac{\frac{3-\ga}{4}w_1}{\mu-\frac{3-\ga}{4}w_1}<e^{-\mu\tau} \frac{\frac{3-\ga}{4}\epsilon_0}{\mu-\frac{3-\ga}{4}\epsilon_0},
	\]
	where we have used that $f(x) = \frac{x}{\mu-x}$ is an increasing function for $x\in[0,\mu)$ and $w_1$ is strictly decreasing from $\epsilon_0$ to 0 as proved in Lemma \ref{lemma w_1}. Hence, 
	\be
		T_*  = \int_{\tau_0}^\infty e^{-\mu\tau} \frac{\frac{3-\ga}{4}w_1(\tau)}{\mu-\frac{3-\ga}{4}w_1(\tau)} \ d\tau < \int_{\tau_0}^\infty e^{-\mu\tau} \frac{\frac{3-\ga}{4}\epsilon_0}{\mu-\frac{3-\ga}{4}\epsilon_0} \ d\tau = \frac{\frac{3-\ga}{4}\epsilon_0}{\mu^2-\frac{3-\ga}{4}\epsilon_0\mu}< 2\epsilon_0,
	\ee
	where we have used $\ga\in(1,3)$ and $\mu\in[\frac12,1)$ in the last inequality.
\end{proof}

Therefore, we see that the condition $z_1=w_1\equiv 0$ fixes the blowup time as $t=0$. Moreover, using the natural constraint $z_1\equiv 0$, we may allow nonvanishing $w_1(\tau_0)>0$, which determines the blowup time $T_*$.

\end{document}